\documentclass[english]{article}
\usepackage{caption}
\usepackage[T1]{fontenc}
\usepackage[latin9]{inputenc}
\usepackage{xcolor}
\usepackage{babel}
\usepackage{verbatim}
\usepackage{amsmath}
\usepackage{amsthm}
\usepackage{amssymb}
\usepackage{graphicx}

\newcommand{\be}{\begin{equation}}
\newcommand{\ee}{\end{equation}}
\newcommand{\ba}{\begin{array}}
\newcommand{\ea}{\end{array}}
\newcommand{\dsp}{\displaystyle}
\newcommand{\N}{\mathbb{N}}
\newcommand{\Z}{\mathbb{Z}}
\newcommand{\R}{\mathbb{R}}
\newcommand{\Exp}{\mathbb{E}}
\newcommand{\Prob}{\mathbb{P}}

\newcommand{\speed}{\sigma}
\newcommand{\indic}{{\bf 1}}
\newtheorem{lemma}{Lemma}[section]
\newtheorem{example}{Example}[section]
\newtheorem{corollary}{Corollary}[section]
\newtheorem{definition}{Definition}[section]
\newtheorem{remark}{Remark}[section]
\newcommand{\thmref}[1]{Theorem \ref{thm:#1}}

\makeatletter

\AtBeginDocument{}
\AtBeginDocument{\providecommand\thmref[1]{\ref{thm:#1}}}
\AtBeginDocument{}
\AtBeginDocument{}

 \theoremstyle{remark}
 
\theoremstyle{plain}

 \theoremstyle{plain}
  
\newtheorem{proposition}{Proposition}[section]
\newtheorem{assumption}{Assumption}[section]
\newtheorem{theorem}{Theorem}[section]

\AtBeginDocument{}

\AtBeginDocument{}

\makeatother

  \providecommand{\lemmaname}{Lemma}
  \providecommand{\remarkname}{Remark}
\providecommand{\theoremname}{Theorem}

\begin{document}

\title{Hydrodynamics and relaxation limit for multilane exclusion process  
and related hyperbolic systems}

\author{G. Amir$^a$, C. Bahadoran$^b$, O. Busani$^c$, E. Saada$^d$}
\maketitle
$$ \ba{l}
^a\,\mbox{\small Department of Mathematics, Bar Ilan University,} \\
\quad \mbox{\small 5290002 Ramat Gan, Israel. E-mail: gideon.amir@biu.ac.il}\\
^b\,\mbox{\small Laboratoire de Math\'ematiques Blaise Pascal, 
Universit\'e Clermont Auvergne,} \\
\quad \mbox{\small 63177 Aubi\`ere, France. E-mail: christophe.bahadoran@uca.fr}\\
^c\, 
\mbox{\small University of Edinburgh, 5321, James Clerk Maxwell Building,} \\
\quad \mbox{\small
Peter Guthrie Tait Road, Edinburgh, United Kingdom.}\\
\quad\mbox{\small E-mail:	obusani@ed.ac.uk}\\
^d\, \mbox{\small 
CNRS, UMR 8145, MAP5, Universit\'e Paris Cit\'e, } \\
\quad \mbox{\small Campus Saint-Germain-des-Pr\'es, 75270 Paris cedex 06, France.}\\
\quad\mbox{\small 
 E-mail: Ellen.Saada@mi.parisdescartes.fr}\\
\ea
$$
\begin{abstract}
We investigate the hydrodynamic behaviour and local equilibrium of 
the multilane exclusion process, whose invariant measures were studied in
our previous paper \cite{mlt1a}. The dynamics on each lane follows 
a hyperbolic time scaling, whereas the interlane dynamics has an 
arbitrary 
time scaling. We prove the following: \textit{(i)} the hydrodynamic 
behavior of the global density  (i.e. summed over all lanes) is governed 
by a scalar conservation law; \textit{(ii)} the latter, as well as 
the limit of individual lanes, is the relaxation limit of a weakly 
coupled hyperbolic  system of balance laws that approximates 
the particle system.   
For the hydrodynamic limit, 
to highlight  new phenomena arising   in our model, 
a precise computation of the flux function, with the transitions 
between different possible shapes (and a physical interpretation thereof), 
is given for the two-lane model.  
\end{abstract}
\section{Introduction}\label{sec:intro}
 The one-dimensional totally asymmetric simple exclusion process (abbreviated
as TASEP) is a popular simplified microscopic model of traffic-flow on a one 
lane highway (\cite{css}). Its hydrodynamic limit under hyperbolic time scaling was first established in \cite{Rost81} under $1-0$ step initial condition, and obtained under general initial conditions in
\cite{Rezakhanlou91}. It is given by a scalar conservation 
law known in traffic-flow modeling as the {\em car traffic equation} or 
{\em Lighthill-Witham model}:
\be\label{hdl_tasep}
\partial_t\rho(t,x)+\partial_x[f(\rho(t,x))]=0,
\ee
where $\rho(t,x)\in[0;1]$ is the density of cars at time $t\geq 0$
and spatial location $x\in\R$.
The function $f$, given here by
\be\label{current_density}
f(\rho)=\rho(1-\rho),
\ee
is called the {\em flux function}, or {\em current-density relation} 
in traffic-flow modeling, and yields the local flux as a closed
function of the sole local density. Equation \eqref{hdl_tasep} is meant 
in the sense of {\em entropy conditions}, that select the unique 
physical solution (\cite{serre})  among many possible ones.
The strict concavity of \eqref{current_density} implies spontaneous 
creation of increasing shocks, which may be viewed as a simplified 
mechanism for the formation of traffic jams, that is regions with 
a sharp transition from low density to high density.
 More generally, traffic flow modelling through PDE's or microscopic models, 
 and their mutual relations,  has  attracted sustained interest in the  
 mathematical analysis and physics communities, 
 see e.g. \cite{ar, imz,mrr,dffrr}. \\ \\
Multilane asymmetric exclusion processes are natural generalizations 
of TASEP  where particles perform an asymmetric exclusion process 
on each lane, with lane-dependent parameters, and additionally change 
lanes according to a certain transverse jump kernel. For instance, 
this may be relevant to incorporate the effect of overtaking. 
Alternatively, a lane can be viewed as a species of particles 
and such models as multi-species exclusion processes. Particles 
on different lanes may have different speeds (and move in different 
directions), in which case the lane (or species) can be interpreted 
as a kinetic parameter. \\ \\
>From a mathematical standpoint, the asymmetric multilane exclusion 
is an interesting intermediate model between one and two dimensional 
asymmetric exclusion process, especially when it comes to studying 
the structure of invariant measures, that is well understood in the 
former case (\cite{Liggett1976, Bramson2002, BMM}) but still widely 
open in the latter (\cite{ligd}). Recently, a fairly complete 
characterization of invariant measures was obtained 
for a wide class of multilane asymmetric models (\cite{mlt1a}) 
as well as for their symmetric counterpart (\cite{rvw}).\\ \\
Related models have been studied in the physics or mathematics literature, 
whether in a multilane or multi-species perspective, like for instance 
two-lane cellular automata for traffic-flow (\cite{bkjs}) or 
two-species interacting exclusions,  see e.g.
(\cite{ft,spo,ps2}). \\ \\
We are concerned here with the hydrodynamic behaviour under hyperbolic 
scaling of the class of multilane asymmetric exclusion processes considered 
in \cite{mlt1a}. 
In the case of {\em symmetric} multilane exclusion, under diffusive time scaling 
on each lane and no time rescaling for transverse jumps, the hydrodynamic limit 
was addressed in \cite{fghnr} and found to be given by a system of linear 
diffusions with linear balance terms.
In our setting, the problem was recently studied heuristically 
through mean-field approximation and Monte-Carlo simulations (\cite{cura,cekt}).   
 Rather than viewing the model as a unique exclusion process, we privilege the
point of view of decoupling the dynamics on each lane (longitudinal dynamics) 
and the interlane (or transverse) dynamics. We shall therefore decompose the 
generator of the process as
\be\label{decomp_gen_intro}
 {\mathcal L}^N  = N L_h+\theta(N)L_v = N\left(
L_h+\frac{\theta(N)}{N}L_v
\right)
\ee
where $N\to+\infty$ is the space scaling parameter, $\theta(N)\to+\infty$, 
$L_h$ and $L_v$ respectively denote generators of the longitudinal and 
transverse jumps. Therefore the longitudinal and transverse time scales are 
decoupled, although this includes the usual situation $\theta(N)=N$ where 
$ {\mathcal L}^N=N\,L$ for the fixed generator $L=L_h+L_v$ speeded 
up in time. In particular, \eqref{decomp_gen_intro} includes the regime 
$\theta(N)\ll N$ of {\em weak coupling} between lanes. This regime is natural 
to model overtaking, whose rate is expected to be negligible with 
respect to longitudinal motion.  It will bring out  a new type of 
scaling condition  requiring ad hoc analysis. \\ \\
For the study of hydrodynamics, we make a reversibility assumption on the 
transverse dynamics. This implies existence of a family of product invariant 
measures which are homogeneous on each lane but with transverse inhomogeneity. 
 This assumption allows us to obtain an explicit hydrodynamic limit. It is satisfied for instance (but not only) whenever only nearest-lane changes are allowed (which is automatic for two lanes) with nondegenerate rates, which is a natural hypothesis in traffic flow context. However, as such, reversibility has no clear traffic-flow interpretation.    \\ \\
Our first main result (Theorem \ref{thm:local_equilibrium}) 
is the hydrodynamic limit and conservation of local equilibrium in the weak sense.
We prove that the hydrodynamic limit of the {\em global} density field 
(i.e. summed over all lanes) is given by entropy solutions of a scalar 
conservation law of the form \eqref{hdl_tasep}. 
We also show that
when the number of lanes $n\to +\infty$, under suitable rescaling, 
it is possible to obtain a singular type of limit that is not governed 
by a scalar conservation law  (Theorem \ref{thm:approximating_G}). \\ \\
Besides the global density, we derive the density fields of each lane;  however,
 these densities cannot be described by a set of 
evolution equations, but are functions of the total density, as we now explain. \\ \\
The limiting flux function and lane densities can be understood as 
{\em relaxation limits.} To this end, we approximate the 
microscopic model with a system of conservation laws with stiff 
relaxation term, whose solution $\rho^\varepsilon$ depends on $\varepsilon$; namely,
\be\label{relax_system_intro}
\partial_t\rho^\varepsilon_i(t,x)
+\partial_x f_i\left[\rho^\varepsilon_i(t,x)\right]
=\varepsilon^{-1}c_i[\rho^\varepsilon(t,x)],\quad i=0,\ldots,n-1
\ee
where $\varepsilon=\theta(N)^{-1}\to 0$.
Here $n$ is the number of lanes and $\rho=(\rho_0,\ldots,\rho_{n-1})$, 
where $\rho_i$ denotes the density field on lane $i$,
\be\label{fluxes_intro}
f_i(\rho_i):=\gamma_i \rho_i(1-\rho_i)  
\ee
is the flux on lane $i$, where $\gamma_i$ is the mean drift of a particle, and
\be\label{trans_intro}
c_i(\rho)=\sum_{j=0}^{n-1}\left[
q(j,i)\rho_j(1-\rho_i)-q(i,j)\rho_i(1-\rho_j)
\right]
\ee
where $q(i,j)$ denotes the transverse jump rate from lane $i$ to lane $j$.  In fact, \eqref{relax_system_intro} is obtained by replacing the longitudinal 
dynamics by the hydrodynamic limit of single-lane ASEP and the transverse 
dynamics by their mean-field approximation. 
The system \eqref{relax_system_intro} is known as a hyperbolic relaxation system 
with {\em weak coupling}, see \cite{chen,hanat,nat},
i.e. independent scalar conservation laws coupled only through their source 
terms.\\ \\
Our second object of study is the relaxation limit $\varepsilon\to 0$ 
for \eqref{relax_system_intro}.
Relaxation limits were obtained  in some particular cases of 
weak coupling such as the Jin-Xin model and a more general class 
of discrete kinetic systems (\cite{nat}).
The relaxation limit, like the hydrodynamic limit, involves 
a local equilibrium closure of an initially non-closed conservation law. 
Indeed, by summing the equations in \eqref{relax_system_intro} 
we obtain a conservation law 
\be\label{not_closed}
\partial_t R^\varepsilon+\partial_x
\sum_{i=0}^{n-1}f_i\left[\rho^\varepsilon_i\right]=0
\ee
for the total density 
\be\label{total_intro}
R^\varepsilon(t,x):=\sum_{i=0}^{n-1}\rho_i^\varepsilon(t,x).
\ee
The flux function in \eqref{not_closed} is a function of $\rho$ but 
{\em not a closed} function of $R^\varepsilon$. In order to close 
 equation \eqref{not_closed}, one must prove that in the limit 
 $\varepsilon\to 0$, $(\rho_0,\ldots,\rho_{n-1})$ lies in an 
 {\em equilibrium manifold} $\mathcal F$ parametrized by the total
  density, say $R$, 
so that the flux 
depends only on $R$. This manifold is given by
\be\label{def_f_intro_0}
\mathcal F:=\left\{
\rho\in[0;1]^n:\,c_i(\rho)=0,\quad\forall i=0,\ldots,n-1
\right\}.
\ee
Under reversibility assumptions on $q(i,j)$, we can show 
(in Proposition \ref{lemma_phi_super}) 
that this manifold is given by
\be\label{def_f_intro}
\mathcal F:=\left\{
\rho\in[0;1]^n:\,\forall i,j\in\{0,\ldots,n-1\},\,
q(i,j)\rho_i(1-\rho_j)=q(j,i)\rho_j(1-\rho_i)
\right\}
\ee
and can indeed be parametrized by the total density, i.e., the function
\be\label{param_intro}
\psi:\rho\in[0;1]^n\mapsto\sum_{i=0}^{n-1}\rho_i\in[0;n]
\ee
is a bijection.
Besides, $\mathcal F$ coincides with the set of vectors $\rho$ for which 
there exists a product invariant measure for the multilane exclusion process 
with mean density $\rho_i$ on lane $i=0,\ldots,n-1$
(Theorem \ref{thm:characterization_lemma_gen}). Remark that if the 
lane number $i$ is viewed as a kinetic parameter through the mean particle 
velocity $\gamma_i$ in \eqref{fluxes_intro}, we can understand the function 
$i\mapsto \psi^{-1}(R)_i$ as the velocity distribution when the total 
density is $R$, that is, the analogue in this context of a Maxwellian.\\ \\
Our second main result (Theorem \ref{th:relax_limit})  states that,  
as $\varepsilon\to 0$, the 
global density \eqref{total_intro}
converges locally in $L^1$ to a limiting field $R(t,x)$ that is, the entropy 
solution to a conservation law of the form \eqref{hdl_tasep}, and that 
the lane densities vector is given by the ``Maxwellian'' $\rho=\psi^{-1}(R)$. 
The flux function arising both in the relaxation and hydrodynamic limit 
is the projection $f:=F\circ\psi^{-1}$ of the total flux function 
\be\label{total_flux_intro}
F(\rho):=\sum_{i=0}^{n-1}f_i(\rho_i)
\ee
on the equilibrium manifold $\mathcal F$. 
The relaxation limits obtained in this setting are indeed similar to those 
obtained in Theorem \ref{thm:local_equilibrium} for the particle system.\\ \\
For the hydrodynamic limit of Theorem \ref{thm:local_equilibrium}, 
to highlight  new phenomena arising   in our model, 
a precise computation of the flux function, with the transitions 
between different possible shapes (and a physical interpretation thereof), 
is given for the two-lane model, 
These include, as in the KLS or AS models (\cite{KLS, PS, AS}), fluxes 
that are bell-shaped or with two maxima and one minimum;  but additionally, 
when the two drifts have opposite signs, fluxes with one positive 
maximum and one negative minimum.
However, this analysis reveals (see Theorem \ref{thm:phase_trans}, 
and Appendix \ref{app:graph} for graphical illustrations) an unexpected 
non-monotone transition of the flux shape with respect to the strength 
of interlane coupling. 
When one of the lanes is symmetric and the other one asymmetric, 
we can obtain a singular non-differentiable type of flux similar 
to the one obtained recently in \cite{ESZ} for the facilitated ASEP. 
Altogether, the different possible shapes and transitions are found 
to reflect the interplay and competition between lanes. \\ \\
 Theorem \ref{thm:local_equilibrium} is valid in the weak coupling regime 
$\theta(N)\ll N$, as well as $\theta(N)\sim N$ and $\theta(N)\gg N$. 
In the former case, the transverse dynamics slows down the local 
equilibrium mechanism, and a minimal growth condition is required on 
$\theta(N)$, cf. \eqref{assumption_relax}. This condition is related 
to the diameter of the transverse random walk graph, and disappears when 
this diameter is $1$, i.e., for two lanes or more generally mean-field walks.  
Hence the interest of an analysis going beyond the two-lane model. \\ \\
We note that  hydrodynamic and relaxation results 
of the same nature were established in \cite{TK1,TK2} for the particle 
and relaxation system when the left-hand side of \eqref{relax_system_intro} 
is replaced by linear transport equations, which corresponds to 
interaction-free longitudinal particle dynamics (that is, particles  
are independent as long as they stay on the same lane). The relaxation 
terms considered there have a different structure, and the proof in 
this non-interacting situation was based on the moment method; this is not 
feasible in our setting, which requires the full hydrodynamic limit machinery, 
though we believe our approach would also be valid in the linear case.\\ \\
We end up with an outline of the methods of proof. \\
For the hydrodynamic limit, we  rely on  the general scheme 
developed in \cite{Bahadoran2002} for hyperbolic type limits, that reduces 
the hydrodynamic limit for the Cauchy problem to the case of step initial 
functions (the so-called {\em Riemann problem}), for which we have an 
explicit variational representation that we reproduce at microscopic level. 
An essential property for this reduction is the so-called 
 {\em macroscopic stability} property, for which a  significant 
refinement of \cite[Lemma 3.1]{BMM} is required when $\theta(N)\ll N$.  
This  accounts for 
condition \eqref{assumption_relax}, which arises from short-time 
analysis of transverse couplings, see Lemma \ref{lemma_coalescence}.  \\
Condition \eqref{assumption_relax} is also involved in the derivation  
of Riemann hydrodynamics, as well as the derivation of individual lane 
profiles from the global one. Indeed for these purposes, we need  
one and two-block estimates that can be obtained here by revisiting an argument 
from \cite{Rezakhanlou91}. The latter was based on a so-called 
{\em interface property} for single-lane ASEP (\cite{Liggett1976}),
whereby the number of sign changes between two coupled systems cannot 
increase. However, this property fails for multilane models. Here, 
thanks to condition \eqref{assumption_relax}, we can show 
that the interface property remains true up to macroscopically 
negligible errors, which we call the {\em quasi-interface} property.  \\
For the relaxation limit, we  conditions for the system 
with a suitable family of dissipative entropies with respect to the 
relaxation terms in \eqref{relax_system_intro}. These are combinations 
of Kru\v{z}kov entropies (\cite{kru,serre}) relative to equilibrium states. 
The entropy dissipation enables us to prove relaxation to the equilibrium 
manifold, and then close the entropy conditions.
The special reversible and monotone structure 
of the relaxation terms plays here an important role for these entropies 
to be dissipative. \\ \\
The paper is organized as follows. 
In Section \ref{sec:model}, we introduce the model and assumptions, 
then the equilibrum manifold, and we
describe the relevant invariant measures for the hydrodynamic limit. 
In Section \ref{subsec:hydro_gen}, we state our main results 
for the hydrodynamic limit and relaxation limit. In Section 
\ref{subsec:examples}, general properties of the flux functions are stated 
as well as various examples and the more precise treatment of the two-lane model.
The corresponding proofs and graphical illustrations for the 
two-lane model are given respectively in Section \ref{sec_phases} and 
in Appendix \ref{app:graph},  where phase transitions are discussed.   
In Section \ref{sec:proof_inv}, we establish the structure of the 
equilibrium manifold under reversibility assumptions and general properties 
of the flux.  The results stated in Section \ref{subsec:hydro_gen} 
are then established in Sections \ref{sec:proof_hydro} 
(for the hydrodynamic limit) and \ref{sec:proof_relax} (for the relaxation limit), 
and the singular limit is proved in Section \ref{sec:proof_limit}.
\section{Multilane exclusion processes and their invariant measures}\label{sec:model}
\subsection{ Multilane exclusion process}\label{subsec:setup}
\mbox{}\\ \\
 \textbf{State space.}  Let 
$V=\Z\times W$, where 
\be\label{def_W}
W:=\{0,\ldots,n-1\}.
\ee
We call $V$ a {\em ladder}.
An element $x$ of $V$ will be generically written 
in the form $x=(x(0),x(1))$, with $x(0)\in\Z$ and $x(1)\in W$. 
The set of particle configurations on $V$ is denoted by
$\mathcal X:=\{0;1\}^V$,
that is, a compact polish space with respect to product topology. 
For $\eta\in\mathcal X$ and $x\in V$, $\eta(x)$ denotes the 
number of particles at $x$.
In traffic-flow modeling, 
we may think of  $V$ as a highway, where
for $i\in W$,
\begin{equation}\label{eq:general lane}
\mathbb{L}_{i}:=
\left\{
 x\in V:\,x(0)\in\mathbb{Z},\,x(1)=i
\right\} 
\end{equation}
denotes the $i$'th lane. Then
$x\in V$ is interpreted as spatial location $x(0)$ on lane   $x(1)$. 
For  $i\in W$, we denote by $\eta^i$ the particle configuration on $\Z$  defined by
\be\label{config_lane}\eta^{i}\left(z\right)=\eta\left(z,i\right),\ee
that is, the configuration on lane $i$.
 Another interpretation is that $i\in W$ represents a particle species;
 then $\eta(z,i)=\eta^i(z)$ is the number of particles of species $i$ at site $z$.
We also denote the total number of particles at $z\in\Z$ by
\be\label{config_total}
\overline\eta(z)=\sum_{i\in W}\eta^i(z).
\ee
 {\bf Dynamics.}
For $\eta\in\mathcal X$ and $x,y\in V$, denote by $\eta^{x,y}$ the new configuration 
after a particle has jumped, if possible at all, from $x$ to $y$: that is, 
\[
\eta^{x,y}\left(w\right)=\left\{ \begin{array}{cc}
\eta\left(w\right) & w\neq x,y\\
\eta(x)-1 & w=x\\
\eta(y)+1 & w=y
\end{array}\right.,
\] 
\begin{definition}\label{def:kernel}
We call  {\em kernel} on a nonempty countable set $S$ 
a mapping $\pi:S\times S\to[0;+\infty)$ such that
\be\label{cond_lig}
\sup_{x\in S}
\sum_{y\in S}[\pi(x,y)+\pi(y,x)]
<+\infty.
\ee
\end{definition}
Let $q(.,.)$ be a kernel on $W$, and for each $i\in W$, let $q_i(.,.)$ be a translation invariant
 kernel on $\Z$. By translation invariant  we mean that there exists a summable function 
 $Q_i:\Z\to[0;+\infty)$ such that 
\be\label{kernel_trans}
\forall u,v\in\Z,\quad q_i(u,v)=Q_i(v-u).
\ee
We define kernels $p_h(.,.)$ and $p_v(.,.)$ on $V$ by
\begin{eqnarray}
\label{kernel_h}
p_h(x,y) & := & \dsp\sum_{i\in W} q_i[x(0),y(0)]{\bf 1}_{\{x(1)=y(1)=i\}}\\
\label{kernel_v}
p_v(x,y) & := & q[x(1),y(1)]{\bf 1}_{\{x(0)=y(0)\}}.
\end{eqnarray}
where $h$ stands for ``horizontal'' and $v$ for ``vertical''  
(in accordance with the interpretation of the model
given later on):
For $\theta\geq 0$, we define the kernel $p^\theta$ on $V$ by
\be\label{restrict_kernel}
p^\theta(x,y):=p_h(x,y)+\theta p_v(x,y).
\ee
The ladder process with kernel $p^\theta(.,.)$ is the Markov process  
$\left(\eta_{t}\right)_{t\geq0}$ on $\mathcal{X}$ 
(\cite{liggett2012interacting}) with generator
\begin{eqnarray}
\nonumber
L^\theta f\left(\eta\right)
& := & \sum_{x,y\in V}p^\theta\left(x,y\right)\eta\left(x\right)\left(1-\eta\left(y\right)\right)
\left(f\left(\eta^{x,y}\right)-f\left(\eta\right)\right)\\
\label{eq:generator of the Exclusion}
& = & (L_h+\theta L_v)  f\left(\eta\right)
\end{eqnarray}
where $L_h$ and $L_v$ are defined respectively by replacing 
$p^\theta(.,.)$ by $p_h(.,.)$ and $p_v(.,.)$. In 
\eqref{eq:generator of the Exclusion},  $f$ is a local function on $V$, 
i.e., a function depending on finitely many coordinates.

Thus $L^\theta$ is the generator of a simple exclusion process (SEP) 
on $V$ with jump kernel $p^\theta$, where $L_h$ corresponds 
to jumps along a lane and $L_v$ corresponds to interlane jumps.
The latter occur with a time scaling parameter $\theta$. \\ \\
 The ladder process can be constructed 
through the so-called
\textit{Harris graphical representation} (\cite{Harris72}). 
Suppose $\left(\Omega,\mathfrak{F},\mathbb{P}\right)$
is a probability space that supports a family  
$\mathcal{N}=\left\{ \mathcal{N}_{\left(x,y\right)}:\left(x,y\right)\in V\right\}$
(called  Harris system)  of independent Poisson processes 
$\mathcal{N}_{(x,y)}$ with respective
intensities $p^\theta\left(x,y\right)$. 
For a given  $\omega \in \Omega$, 
we let the  process evolve  according to the following rule:
if there is a particle at site $x\in V$ at time $t^{-}$ where 
 $t\in\mathcal{N}_{\left(x,y\right)}$, 
it shall attempt to jump to site $y$. The attempt is suppressed if
at time $t^{-}$ site $y$ is occupied.  
\begin{remark}\label{rk:indep}
\textit{For $\theta=0$, i.e. $L^\theta=L_h$, the process generated 
by $L^\theta$ is a collection of independent SEP's on different lanes 
with jump kernel $q_i(.,.)$ on lane $i$, whereas for $\theta>0$, 
$\theta$ encodes the strength of interlane coupling.}
\end{remark}
\subsection{Assumptions}\label{subsec:multi} 
In the sequel, a family of real numbers $\rho_i$ indexed by $i\in W$ 
will be denoted by $(\rho_i)_{i=0,\ldots,n-1}$, or 
$(\rho_0,\ldots,\rho_{n-1})$. \\ \\
Recall Definition \ref{def:kernel}.
For $x,y\in S$ such that $x\neq y$, and $\ell\in\N$, 
 we write $x\stackrel{\ell}{\rightarrow}_\pi y$ if there exists  
 a path $(x=x_0,\ldots,x_{\ell}=y)$ of length $\ell$ such that 
 $\pi(x_k,x_{k+1})>0$ for $k=0,\ldots,\ell-1$. We write $x\rightarrow_\pi y$ 
 if there exists $\ell\in\N$ such that
$x\stackrel{\ell}{\rightarrow}_\pi y$. We omit mention of $\pi$
whenever there is no ambiguity on the kernel. 
We say $x$ and $y$ are $\pi$-connected if  $x\rightarrow_\pi y$
or $y\rightarrow_\pi x$. 
We set the following definition. 
\begin{definition}\label{def_irred}
A kernel $\pi$ will be called {\em weakly irreducible} if, 
for every $x,y\in S$ such that $x\neq y$, $x\rightarrow_\pi y$ 
or $y\rightarrow_\pi x$.
\end{definition}
We refer to the above property as {\em weak} irreducibility 
as opposed to irreducibility for which one requires 
$x\rightarrow_\pi y$ {\em and} $y\rightarrow_\pi x$. In the context 
of interacting particle systems, this is a more natural assumption, 
since it includes for instance totally asymmetric jumps which do not 
satisfy full irreducibility (see examples 
in Section \ref{subsec:examples}).  \\ \\ 
We make the following assumptions 
on the transition kernel $p^\theta$ 
(see \eqref{kernel_h}--\eqref{restrict_kernel})
 of the ladder process. 
\begin{assumption}\label{assumption_ker}\mbox{}\\ \\
 \textit{(i)}  \textit{For every $i=0,\ldots,n-1$,  there exists $d_i\geq 0$ 
 and $l_i\geq 0$ such that $d_i+l_i>0$ and }
\be\label{nearest}
Q_i(z)=d_i{\bf 1}_{\{z=1\}}+l_i{\bf 1}_{\{z=-1\}}.
\ee
\textit{(ii)}  \textit{The kernel $q(.,.)$ is weakly irreducible.}
\end{assumption}
\noindent On lane $i\in W$, $\gamma_i$ denotes the mean drift, that is,
\be\label{def_drift}
\gamma_i:=\sum_{z\in\Z}zQ_i(z)=d_i-l_i.
\ee 
 Note that the drift may have different signs 
(including $0$) on different lanes.
For instance, we may have a TASEP on each lane, 
but going left to right on some lanes and right to left on others.
If one thinks of the multilane TASEP as a highway connecting 
two cities A and B, then 
cars can travel  from city A to city B on some lanes and
from city B to city A on others.
\begin{assumption}\label{assumption_rev}
\textit{For every irreducibility class $\mathcal C$ of $q(.,.)$, 
there exists a nonzero reversible measure 
$\lambda_.=(\lambda_i)_{i\in\mathcal C}$ for the chain 
restricted to $\mathcal C$. That is, for all  $i,j\in\mathcal C$,}
\be\label{reversible_class}
\lambda_i q(i,j)=\lambda_j q(j,i).
\ee
\end{assumption}
\begin{remark}\label{remark_rev} 
\textit{Regarding Assumption \ref{assumption_rev}: \\ 
1. If $\mathcal C$ is a singleton, the condition is void. In this case, 
any measure on $\mathcal C$
(that is any constant value $\lambda_i$ assigned to the unique element $i$ 
of $\mathcal C$) is reversible. \\ 
2. Since the restriction of the kernel to $\mathcal C$ is irreducible,  
the reversible measure $\lambda_.$ is positive on $\mathcal C$ and  
unique up to a multiplicative factor. }
\end{remark}
 Several examples of models satisfying the above assumptions will be given in 
 Section \ref{subsec:examples}. 
\subsection{Invariant measures}\label{subsec:inv-meas}
In the sequel, we denote by $(\tau_k)_{k\in\mathbb{Z}}$ 
the group of space shifts on $\Z$, defined as follows. 
The shift operator $\tau_k$ acts on a particle configuration
$\eta\in\mathcal X$ through
\begin{equation}\label{shift_config}
(\tau_k\eta)(z,w):=\eta(z+k,w),\quad\forall (z,w)\in\mathbb{Z}\times W.
\end{equation}
We write $\tau$ instead of $\tau_{1}$. We denote by
$\mathcal{S}$ the set of all probability measures on $\mathcal{X}$ that
are invariant under $\tau$, and by $\mathcal I$ the set of probability 
measures on $\mathcal I$ that are invariant for the generator 
\eqref{eq:generator of the Exclusion}. We are interested here in the 
(compact convex)  set $\mathcal I\cap\mathcal S$. By Choquet-Deny theorem, 
every one of its elements is a mixture of its extremal elements, the set 
of which is denoted by $(\mathcal I\cap\mathcal S)_e$. \\ \\
It is well known  that product Bernoulli measures are invariant 
for translation-invariant exclusion processes on $\Z$ 
(\cite{liggett2012interacting}). 
It follows from Remark \ref{rk:indep} that the family of  measures
 $\nu^{\rho_0,\ldots,\rho_{n-1}}$ defined for 
 $(\rho_0,\ldots,\rho_{n-1})\in[0,1]^n$ by
\be\label{def_inv_gen}
\nu^{\rho_0,\ldots,\rho_{n-1}}\left\{
\eta(z,i)=1
\right\} =\rho_i,\quad (z,i)\in \Z\times W
\ee
 is invariant for the  generator \eqref{eq:generator of the Exclusion} 
 when $\theta=0$. When $\theta>0$, the vertical kernel is present, and 
 we expect the coupling of lanes to select a family of invariant measures 
 where $(\rho_0,\ldots,\rho_{n-1})$ is restricted to a one-dimensional 
 relaxation  (or equilibrium)  manifold defined by the relations 
\be\label{detailed_q}
\rho_i(1-\rho_j)q(i,j)=\rho_j(1-\rho_i)q(j,i)
\ee  
for every $(i,j)\in W^2$ such that $i\neq j$. 
We denote this manifold by
\be\label{def_set_f}
\mathcal F:=\{
(\rho_0,\ldots,\rho_{n-1})\in[0,1]^n\mbox{ satisfying relation }\eqref{detailed_q}
\}.
\ee
 The following proposition shows that  we can parametrize the set 
$\mathcal F$  by a global density parameter $\rho\in[0,n]$. 
See Section \ref{sec:proof_inv} for its proof. 
\begin{proposition}\label{lemma_phi_super}
\mbox{}\\ \\
(i) The mapping $\psi:\mathcal F\to[0,n]$ defined by
\be\label{def_psi_super}
\psi(\rho_0,\ldots,\rho_{n-1}):=\sum_{i=0}^{n-1}\rho_i
\ee
is a bijection.\\ 
(ii) For every $i\in W$,  the mapping 
\be\label{inverse_mapping}
\widetilde{\rho}_i:\rho\mapsto\widetilde{\rho}_i(\rho)
:=\left(\psi^{-1}(\rho)\right)_i\ee
 is continuous and nondecreasing. 
\end{proposition}
It follows from Proposition  \ref{lemma_phi_super} that if we set
\be\label{def_inv_rho}
\nu_\rho:=\nu^{\widetilde{\rho}_0(\rho),\ldots,\widetilde{\rho}_{n-1}(\rho)}
\ee
for every $\rho\in[0,n]$, then 
\be\label{mean_inv}
\int_\mathcal X\eta(0)\nu_\rho(d\eta)=\rho
\ee
\be\label{repar}
\left\{
\nu^{\rho_0,\ldots,\rho_{n-1}}:\,(\rho_0,\ldots,\rho_{n-1})\in\mathcal F
\right\}
=\{\nu_\rho,\,\rho\in[0,n]\}.
\ee
The following result is an extension of \cite[Theorem 2.1]{mlt1a} 
and can be  established similarly.  Therefore we omit its proof. 
\begin{theorem}
\label{thm:characterization_lemma_gen}
\be\label{extension_charac_gen}
(\mathcal I\cap\mathcal S)_e=\{\nu^{\rho_0,\ldots,\rho_{n-1}},\,
(\rho_0,\ldots,\rho_{n-1})\in\mathcal F\}=\{\nu_\rho,\,\rho\in[0,n]\}.
\ee
\end{theorem}
\begin{remark}\label{remark_inv_meas}
\textit
{In \cite[Theorems 2.2 and 2.3]{mlt1a}, we studied more generally 
the structure of the set $\mathcal I_e$ of extremal invariant measures 
for two-lane exclusion processes and explained in 
\cite[Appendix A]{mlt1a} how these results could be partially extended 
to multilane processes. Although the class of multilane processes 
mentioned in \cite{mlt1a} is less general than the one considered 
in this paper, the same approach could also apply here. See also item 3. 
of Remark \ref{remark_der} about the role played by 
Proposition \ref{prop:flux_multi} below in this extension. 
However, the full characterization of $\mathcal I_e$ is not 
required for our proof of  hydrodynamic limit,  as explained more precisely in Remark \ref{remark_av} below. 
We only need to know the simpler part of that statement, i.e. that the measures in \eqref {extension_charac_gen} are invariant. This follows from the same computation as in \cite{mlt1a}, using Assumption \ref{assumption_ker}, (i) and Assumption \ref{assumption_rev}.}
\end{remark}
 \section{Hydrodynamics, convergence and relaxation}\label{subsec:hydro_gen}
Before stating our results in Subsections \ref{subsec:results} 
and \ref{subsec:relax},  we introduce the necessary definitions
 in Subsection \ref{subsec:def}. 
\subsection{Definitions}\label{subsec:def}
\mbox{}\\ \\
{\em Empirical measures and density profiles.} 
Let $\mathcal{M}$ be the set of measures on $\mathbb{R}$ equipped
with the vague topology,  and $N\in\N^*:=\{1,2,\cdots\}$ represent 
the scaling parameter, that is the inverse of the macroscopic distance 
between two consecutive sites on a lane. For a particle configuration 
 $\xi\in \{0,\ldots,n\}^\Z$,
we define its associated empirical measure at scale $N$ by
\be\label{def_empirical}
\alpha^{N}\left(\xi,dx\right):
= N^{-1}\sum_{z\in \Z}\xi\left(
z\right
)\delta_{\frac{z}{N}}\left(dx\right).
\ee
In particular, if  $\eta$ is a configuration for the multilane 
 asymmetric simple exclusion process (ASEP), 
recall definitions \eqref{config_lane} and \eqref{config_total}. 
Then for $i\in W$,
$\alpha^N(\overline{\eta},dx)$, $\alpha^N(\eta^i,dx)$  
respectively represent the global empirical measure 
and the empirical measures on lane $i$.\\ \\
 Let  $u(.)$ be a $[0,n]$-valued Borel function on $\R$.  
 We say a sequence  $\left(\xi^{N}\right)_{N\in\mathbb{N}^*}$  
 of random $\mathcal X$-valued configurations 
 has \emph{density profile} $u\left(\cdot\right)$
if 
 $\alpha^{N}\left({\overline\xi}^{N},dx\right)$  converges to $u\left(x\right)dx$
in probability as $N\rightarrow\infty$;  that is,  for every
$\epsilon>0$ and every continuous function $\phi:\mathbb{R}\rightarrow\mathbb{R}$
with compact support, 
\begin{equation}
\lim_{N\rightarrow\infty}\mathbb{P}\left(\left|N^{-1}
\sum_{z\in\mathbb{Z}}
\phi\left(\frac{z}{N}\right){\overline\xi}^{N}\left(z\right)
-\int\phi(y)u(y)dy\right|>\epsilon\right)=0.
\label{eq:hydrodef}
\end{equation}
 {\em Local Gibbs states.} 
 In the following,  $\mathcal B_\rho$ denotes the Bernoulli 
distribution with parameter $\rho\in[0,1]$.
Let $(\xi^N)_{N\in\N^*}$ be a sequence of random configurations 
on $\mathcal X$.
Recall the mappings  $\widetilde{\rho}_i$ defined in  
Proposition \ref{lemma_phi_super}.  We say  $(\xi^N)_{N\in\N^*}$ 
is a  {local Gibbs state}  (abbreviated by l.g.s. later on) 
with global profile $u(.)$, or with lane profiles  $u^0(.),\ldots,u^{n-1}(.)$,  where
\be\label{profile_lane}
u^i(.):=\widetilde{\rho}_i[u(.)],\quad i\in W ,
\ee
if the law $\mu^N$ of $\xi^N$ writes
\be\label{local_gibbs}
\mu^N(d\xi):=\bigotimes_{x\in\Z,\,i\in W}\mathcal B_{u^{N,i}_x}[d\xi(x,i)]
\ee
where
\be\label{gibbs_lane}
u^{N,i}_x=\widetilde{\rho}_i[u^N_x],\quad i\in W ,\,x\in\Z
\ee
and $(u^{N}_x)_{N\in\N^*,\,x\in\Z}$ is a $[0,n]$-valued family 
such that, for every $a<b$ in $\R$,
\be\label{local_gibbs_1}
\lim_{N\to+\infty}
\int_a^b\left|u^{N}_{\lfloor Nx\rfloor}-u(x)\right|dx=0.
\ee
 Note that the expression {\em product measure with slowly varying parameter}, as e.g. in \cite{kipnis2013scaling}, would be more precise, since product measures are a very special case of Gibbs measures. However we use {\em local Gibbs state} because it is shorter and also commonly used in this  context. \\ \\
 In the special case of so-called {\em Riemann} profiles,  that is 
\be\label{first_riemann}
u(x)  =\alpha1_{\left\{x\leq 0\right\} }
+\beta1_{\left\{x>0\right\} }=:R_{\alpha,\beta}(x)
\ee
where  $\alpha,\beta\in[0,n]$,   a natural choice is 
\be\label{gibbs_riemann}
u^N_x:=\alpha{\bf 1}_{\{x\leq 0\}}+\beta{\bf 1}_{\{x> 0\}},\quad x\in\Z. 
\ee
The measure $\mu^N$ in \eqref{local_gibbs} (which no longer depends on $N$) 
is then denoted by $\mu_{\alpha,\beta}$.\\ \\
{\em Local equilibrium.} Let $f$ be a local function of $\mathcal X$. 
For $(\rho_0,\ldots,\rho_{n-1})\in[0,1]^n$ and $\rho\in[0,n]$, 
we define 
\begin{eqnarray}\label{partial_average}
\langle f\rangle(\rho_0,\ldots,\rho_{n-1}) & := & \int_{\mathcal X}f(\eta)d\nu^{\rho_0,\ldots,\rho_{n-1}}(\eta),\\
\label{eq_average}
\overline{f}(\rho) & := & \int_{\mathcal X}f(\eta)d\nu_\rho(\eta)
=\langle f\rangle(\widetilde{\rho}_0(\rho),\ldots,\widetilde{\rho}_{n-1}(\rho)).
\end{eqnarray}
Remark that $\overline{f}$ is continuous, and that it is increasing 
if $f$ is nondecreasing. 
We say $(\xi^N)_{N\in\N^*}$  satisfies the {\em weak local equilibrium}
 property  (abbreviated by w.l.e. later on)  with profile $u(.)$, or with lane profiles  $u^0(.),\ldots,u^{n-1}(.)$ 
 if for every $\varphi\in C^0_K(\R)$, and every local function $f$ 
 on $\mathcal X$, the following limit holds in probability:
\begin{eqnarray} 
N^{-1}\sum_{x\in\Z}\varphi\left(\frac{x}{N}\right)
\tau_x f(\xi^N) & 
\stackrel{N\to+\infty}{\longrightarrow} 
&  \int \varphi(x)\langle f\rangle(u^0(x),\ldots,u^{n-1}(x))dx
\nonumber\\
& = &  \int \varphi(x)\overline{f}[u(x)]dx.
\label{wle}
\end{eqnarray}
We  say $(\xi^N)_{N\in\N^*}$ 
satisfies the {\em strong local equilibrium} property with profile $u(.)$, 
or lane profiles  $u^0(.),\ldots,u^{n-1}(.)$  
if, for every point of continuity $x\in\R$ of $u(.)$,  
\be\label{sle}
\lim_{N\to+\infty}
\Exp\left\{\tau_{\lfloor Nx\rfloor}f(\xi^N)\right\}
=\langle f\rangle (u^0(x),\ldots,u^{n-1}(x))
=\overline{f}[u(x)].
\ee
By the law of large numbers, a local Gibbs state $(\xi^N)_{N\in\N^*}$ satisfies 
the  weak local equilibrium property with the same profiles. 
The weak local equilibrium property implies 
that $(\xi^{N,i})_{N\in\N^*}$ has density profile $u^i(.)$, hence that 
$(\overline{\xi}^N)_{N\in\N^*}$ has density profile $u(.)$, 
as can be seen by choosing
\be\label{observable} f(\eta)=\eta^i(0), \quad 
\langle f\rangle(\rho^0,\ldots,\rho^{n-1})=\rho^i,\quad
\quad \overline{f}(\rho)=\widetilde{\rho}_i(\rho).
\ee 
 \subsection{Hydrodynamic limit and convergence results}\label{subsec:results}
 Let $(\theta(N))_{N\in\N^*}$ be a positive integer-valued sequence such that 
\be\label{assumption_relax_weak}
\lim_{N\to+\infty}
\theta(N)=+\infty.
\ee
Note that when the limit in \eqref{assumption_relax_weak} is finite, 
we expect a different limiting behaviour, 
see  remark \ref{rk:scaling}. 
 When $n\geq 3$, depending on $n$ and $q(.,.)$, we need a stronger 
growth assumption on $\theta(N)$, still allowing $\theta(N)\ll N$ 
with some restriction. Namely, let 
\begin{eqnarray}\label{connection}
n^*=n^*(n,q(.,.)) & := & \max\{d_q(i,j):\,(i,j)\in W,\,i<j\},\quad\mbox{where}\\
\label{distance_q}
d_q(i,j) & := & \inf\{\ell\in\N:\,i\stackrel{\ell}{\rightarrow}_q j
\mbox{ or }j\stackrel{\ell}{\rightarrow}_q i\}. 
\end{eqnarray}
In words, $n^*$ is  the maximal  value over all pairs $i\neq j\in W$ 
of the minimum length of a path connecting $i$ to $j$  by $q(.,.)$ or 
its reverse $\check{q}(i,j):=q(j,i)$. This is also the diameter of $W$ 
for the unoriented graph distance induced by the kernel $q(.,.)$. Let
\be\label{number_steps}
m^*:=m^*(n,q(.,.))=\left\lfloor \frac{n^*}{2}\right\rfloor\left(
n^*-1-\left\lfloor \frac{n^*}{2}\right\rfloor
\right)+n^*.
\ee
Then we make the assumption that
\be\label{assumption_relax}
\lim_{N\to+\infty}
 \frac{\theta(N)}{
N^{1-\frac{1}{m^*}}
}
=+\infty.
\ee
\begin{remark}\label{rk_relax}
\textit{When $n=2$, or more generally when} 
\be
\label{onestep}
q^*:=\min\{
q(i,j)+q(j,i):\,(i,j)\in\{0,\cdots,n-1\}^2,\,i<j
\}>0
\ee
\textit{we have $n^*=1$, hence $m^*=1$ (see Example \ref{example_coalescence}), 
hence \eqref{assumption_relax} reduces
to \eqref{assumption_relax_weak}. }
\end{remark}
We consider a sequence of processes $\eta^N_.=(\eta_t^N)_{t\geq 0}$ such that, 
for each $N\in\N^*$, $\eta^N_.$ has generator
 $L^{\theta(N)/N}$ (cf. \eqref{eq:generator of the Exclusion}). 
 We are interested in the speeded up process under hyperbolic scaling, 
 i.e., $(\eta^N_{Nt})_{t\geq 0}$, which has generator 
\be\label{gen_relax}
\mathcal L^N:=NL^{\theta(N)/N}=NL_h+\theta(N)L_v.
\ee 
 To describe the hydrodynamic behaviour of these processes, we need to define  
 the so-called macroscopic flux function $G$ that appears 
in the hydrodynamic limit. As usual in this context, the {\em macroscopic} 
flux function is expressed as an equilibrium average of a 
{\em microscopic} flux function, that is a function of the microscopic 
configuration $\eta\in\mathcal X$. For our model, the latter writes
\be\label{def_microflux_ladder}
j:=\sum_{k\in W}j_k
\ee
where, for $k\in W$  (recall \eqref{nearest}),
\begin{eqnarray}\label{jh}
 j_k(\eta):=d_k\eta^k(0)[1-\eta^k(1)]-l_k\eta^k(1)[1-\eta^k(0)] 
\end{eqnarray}
represents the microscopic current on lane $k\in W $.  Note that, for every $x\in\Z$,
\be\label{micro_gradient}
\mathcal L^N[\overline{\eta}(x)]
=NL_h[\overline{\eta}(x)]=N\left(\tau_{x-1}j(\eta)-\tau_x j(\eta)\right).
\ee
Indeed, since the vertical  component  $L_v$ of the generator contains 
only vertical jumps,  it leaves $\overline{\eta}(z)$ unchanged, and 
thus does not contribute to \eqref{micro_gradient}. 
The macroscopic flux is then given by,  for $\rho\in[0,n]$, 
\be\label{def_macroflux_ladder}
G\left(\rho\right)=\int j(\eta)\nu_{\rho}\left(d\eta\right).
\ee
 Using  \eqref{def_inv_rho} and \eqref{repar},  this yields
\be\label{def_flux_super}
G(\rho):=\sum_{i=0}^{n-1}\gamma_i G_0\left[
\widetilde{\rho}_i(\rho)
\right]
\ee
where   the mean drift
$\gamma_i$ was defined in \eqref{def_drift},  and
$G_0$ is the flux function of the single-lane TASEP, given by
\begin{equation}\label{flux_tasep}
G_{0}\left(\alpha\right)=\alpha\left(1-\alpha\right)\quad  \forall \alpha\in[0,1]. 
%
\end{equation} 
 Properties and examples of the flux function defined 
 by \eqref{def_macroflux_ladder}--\eqref{flux_tasep} 
 are studied in Section \ref{subsec:examples}. 
We can now state  hydrodynamic limit and local equilibrium 
results for the multilane ASEP. We consider the conservation law
\begin{align}
 \partial_{t}u + \partial_{x}G\left(u\right)& = 0 \nonumber \\
 u\left(0,x\right)  & =u_{0}\left(x\right)\label{eq:hydrodynamic equation}
\end{align}
with $G$ given by \eqref{def_flux_super}. 
An important particular case 
is the {\em Riemann} initial conditions, 
that is (cf. \eqref{first_riemann})
\begin{equation}\label{eq:init-riemann}
u_0(x) =R_{\alpha,\beta}(x),\quad x\in\R
\end{equation}
\begin{theorem}
\label{thm:local_equilibrium} Let $u_0(.)$ be a $[0,n]$-valued 
Borel function on $\R$, and  
$\left(\eta_{t}^{N}\right)_{t\geq 0,N\in\mathbb{N}^*}$  
be a sequence of processes with generator 
 $L^{\theta(N)/N}$. 
  Then the following statements hold under assumption \eqref{assumption_relax}: \\ 
(i)  Assume $(\overline{\eta}_{0}^{N})_{N\in\N^*}$ has density  profile  
$u_{0}\left(\cdot\right)$.
Then  $(\overline{\eta}^{N}_{Nt})_{N\in\N^*}$  has 
density profile  $u(t,.)$,  the entropy solution at time $t$ to 
the conservation law \eqref{eq:hydrodynamic equation}.\\ 
(ii) Assume $(\eta^N_0)_{N\in\N^*}$ is a local Gibbs state with  
lane profiles $u_0^0(.),\dots,u_0^{n-1}(.)$. 
Then $(\eta^N_{Nt})_{N\in\N^*}$ satisfies weak local equilibrium with lane profiles 
   $u^0(t,.),\ldots,u^{n-1}(t,.)$   defined from  $u(t,.)$  as 
  in \eqref{profile_lane}.\\
(iii) Assume 
$\eta^N_0\sim \mu_{\alpha,\beta}$
for some  $0\leq\alpha,\beta\leq n$, where $\mu_{\alpha,\beta}$ 
was defined below \eqref{gibbs_riemann}.  
Then 
 $(\eta^{N}_{Nt})_{N\in\N^*}$  satisfies  strong local equilibrium with lane profiles 
  $u^0(.),\ldots,u^{n-1}(.)$  defined by \eqref{profile_lane}. 
 \end{theorem}
The hydrodynamic profiles  $u^i(t,.)$  of individual lanes in statements 
 \textit{(ii)--(iii)} of Theorem \ref{thm:local_equilibrium} 
 do not obey an autonomous equation.  We next show that they can be 
 interpreted as {\em relaxation limits} for a hyperbolic system of 
 scalar conservation laws with so-called {\em weak coupling}. 
\subsection{Relaxation limit}\label{subsec:relax}
When the lanes are not coupled by a vertical kernel,
 by Remark \ref{rk:indep}, the hydrodynamic limit of the system is 
 given by a {\em system} 
of {\em uncoupled} conservation laws:
\be\label{system}
\left\{
\ba{lll}
\partial_t\rho_0+\partial_x[\gamma_0\rho_0(1-\rho_0)] & = & 0\\ 
 & \vdots  &  \\
\partial_t\rho_{n-1}+\partial_x[\gamma_{n-1}\rho_{n-1}(1-\rho_{n-1})]
& = & 0
\ea
\right.
\ee
where the entropy solution is picked for each of these equations.  
Indeed the $i$-th equation in \eqref{system} was shown 
(\cite{Rezakhanlou91}) to be the hydrodynamic limit for a
 simple exclusion process on $\Z$ with mean drift $\gamma_i$.  \\ \\
We can view the vertical dynamics as adding creation/annihilation 
terms on each lane. This can be formally understood  as a relaxation 
system obtained from
\eqref{system} by adding fast balance terms on the right-hand side, 
whose equilibrium
manifold is exactly $\mathcal F$ defined by 
the relations \eqref{detailed_q}.  \\ \\
Indeed, using the generator \eqref{gen_relax} defined 
from \eqref{eq:generator of the Exclusion}--\eqref{nearest}, 
 we see that 
the time evolutions of the expected densities on each lane 
are given by,  for $i\in W$, 
\be\label{coupled_system_1}
\ba{lll}
\dsp\frac{d}{dt}\Exp[\eta^i_{Nt}(x)]
 & = &
\dsp N\Exp\left[
d_i\eta^i_{Nt}(x-1)(1-\eta^i_{Nt}(x))-l_i\eta^i_{Nt}(x)(1-\eta^i_{Nt}(x-1))
\right]\\ 
&-&
N\Exp
\left[
d_i\eta^i_{Nt}(x)(1-\eta^i_{Nt}(x+1))-l_i\eta^i_{Nt}(x+1)(1-\eta^i_{Nt}(x))
\right]\\ 
& + & \dsp 
\theta(N)\Exp
\sum_{j\in W}\left[
q(j,i)\eta^j_{Nt}(x)(1-\eta^i_{Nt}(x))-q(i,j)\eta^i_{Nt}(x)(1-\eta^j_{Nt}(x))
\right].
\ea
\ee
In \eqref{coupled_system_1}, the first two lines are 
produced by the horizontal part of the dynamics, that is ASEP on each lane, and 
the third line by the vertical part.\\
Following the usual heuristic in hydrodynamic limits, if we believe {\em a priori} 
that the measures \eqref{def_inv_gen} describe local equilibrium states, 
we replace the above expectations with expectations computed under these
measures, and approximate the lattice gradients with  
continuous space gradients. This yields formally the following 
relaxation system for $i=0,\ldots,n-1$:
\be\label{relax_system}
\ba{lll}
\partial_t\rho_i+\gamma_i\partial_x[\rho_i(1-\rho_i)] 
& = & \theta(N)\sum_{j\in W}[q(j,i)\rho_j(1-\rho_i)-q(i,j)\rho_i(1-\rho_j)]\\ 
\ea
\ee
where (forgetting about the particle system)  $\theta(N)$  
now plays the role of the relaxation parameter. 
It is important to note that the interlane jumps do not produce a lattice gradient 
(it should be viewed as a creation rather than a transport term), 
hence the time scaling of order  $\theta(N)$ 
does not disappear from the right-hand side of \eqref{relax_system}.
Adding the equations in \eqref{relax_system} shows that 
\be\label{conserved_quantity}
{\rho}:=\sum_{i\in W}\rho_i
\ee
is a conserved quantity, which satisfies the  conservation law
\be\label{non_closed}
\partial_t{\rho}+\partial_x\sum_{i\in W}\gamma_i\rho_i(1-\rho_i)=0.
\ee
Note that for finite $N$, \eqref{non_closed} is not a closed equation, since the flux 
\be\label{nonclosed_flux}
\sum_{i\in W}\gamma_i\rho_i(1-\rho_i)=\sum_{i\in W}\gamma_i G_0(\rho_i)\ee
in \eqref{non_closed} cannot be expressed as a  function
of ${\rho}$. However, it is expected that in the limit $N\to+\infty$, 
 $(\rho_0,\ldots,\rho_{n-1})$  converges to the so-called 
 {\em equilibrium manifold} defined by 
declaring that each right-hand side should vanish. This is justified heuristically by
the fact that otherwise the left-hand side would become infinite in the limit.
On this manifold, that will turn out to be exactly $\mathcal F$ defined 
by the relations \eqref{detailed_q},  the densities  $\rho_0,\ldots,\rho_{n-1}$ 
are (see Lemma \ref{lemma_phi_super})  functions $\rho_i=\widetilde{\rho}_i(\rho)$ 
of the total density. Thus the flux \eqref{nonclosed_flux} is exactly 
equal to \eqref{def_flux_super}.\\ \\
In the next results, we investigate 
the relaxation limit for \eqref{relax_system}  
and show that the limiting global density is governed 
by the same conservation law as the hydrodynamic limit 
in Theorem \ref{thm:local_equilibrium}. Moreover, 
we prove that in the limit the $\rho_i$'s coincide 
with the individual lane densities in Theorem \ref{thm:local_equilibrium}.  \\ \\
 We begin by recalling from \cite{hanat} the definition and properties of 
 entropy solutions for a general hyperbolic system with weak coupling.  We consider systems of the form 
\be\label{general_system}
\partial_t\rho_i+\partial_x[f_i(\rho_i)] = c_i(\rho_0,\ldots,\rho_{n-1})
\ee
where the flux functions $f_i:[0,1]\to\R$, and the reaction terms 
$c_i:[0,1]^n\to\R$, $i=0,\ldots,n-1$, are functions of class $C^1$.  We assume that the functions $c_i$ satisfy the conditions
\be\label{assumption_system_exclusion_2}
\rho_i=0  \Rightarrow  c_i(\rho_0,\ldots,\rho_{n-1})\geq 0,\quad
\rho_i=1 \Rightarrow c_i(\rho_0,\ldots,\rho_{n-1})\leq 0
\ee
for all $i=0,\ldots,n-1$ and $(\rho_0,\ldots,\rho_{n-1})\in[0;1]^n$.
Assumption
\eqref{assumption_system_exclusion_2} is meant  to have indeed solutions 
$(\rho_0,\cdots,\rho_{n-1})$ of \eqref{general_system} (in a sense defined 
right below) with values in $[0;1]^ n$, see below and Appendix \ref{app:exist}.  \\ \\
We call {\em entropy}  a convex function  $h\in C^1([0,1],\R)$.
The $i$-th {\em entropy flux} associated to $h$ (defined 
modulo a constant) is $g_i(\rho)=\int h'(\rho)f'_i(\rho)d\rho$. 
We then say $(h,g_i)$ is an {\em entropy-flux pair} for the $i$-th equation.
\begin{definition}\label{def:relax_entropy}
Let be given  $\rho^0_i\in L^\infty([0,+\infty)\times\R;[0,1])$ 
for $i=0,\ldots,n-1$.
The family of functions
$\rho^i(.,.)\in L^\infty([0,+\infty)\times\R;[0,1])$, 
where $i=0,\ldots, n-1$, is an {\rm entropy solution to} 
\eqref{general_system}  with initial data $(\rho^0_0,\ldots,\rho^0_{n-1})$,  
if and only if the following conditions are satisfied for every 
$i=0,\ldots,n-1$: \\ \\
1. (Entropy conditions) For every entropy $h:[0,1]\to\R$ 
with  associated entropy flux $g_i$, 
\be\label{entropy_cond}
\partial_t h[\rho_i(t,x)]+\partial_x g_i[\rho_i(t,x)]
\leq h'[\rho_i(t,x)]c_i[\rho_0(t,x),\ldots,\rho_{n-1}(t,x)]
\ee
in the sense of distributions on $[0,+\infty)\times\R$. \\ 
2. (Initial conditions) For every $a<b$ in $\R$,
\be\label{initial_cond}
\lim_{t\to 0}\int_a^b\left|\rho_i(t,x)-\rho_i^0(x)\right|=0.
\ee
\end{definition}
 Existence and uniqueness of an entropy solution in the sense 
 of Definition \ref{def:relax_entropy} are proved in \cite{hanat}  
 when the solution takes values in $\R^n$ and functions $f_i,c_i$ 
 are defined on $\R^n$. In Appendix \ref{app:exist}, we show that under condition
\eqref{assumption_system_exclusion_2}, this remains true for solutions 
taking values in $[0;1]^n$ and functions $f_i,c_i$ defined on $[0;1]^n$.  
\begin{remark}\label{remark_kru}
\textit{As in the case of a single scalar conservation law,  
it is equivalent to impose the entropy conditions  \eqref{entropy_cond} 
for $C^1$ entropies $h$, or to impose it for one of the following families of 
{\rm Kru\v{z}kov} entropy pairs (\cite{kru, serre}):}
\begin{eqnarray}\label{kruzkov}
h_c(\rho)&=&|\rho-c|,\quad g_{i,c}(\rho)
={\rm sgn}(\rho-c)[f_i(\rho)-f_i(c)],\quad c\in[0,1],\\
\label{kruzkov_plus}
h_{c+}(\rho)&=&(\rho-c)^+,\quad g_{i,c+}(\rho)
={\bf 1}_{\{\rho>c\}}[f_i(\rho)-f_i(c)],\quad c\in[0,1],\\
\label{kruzkov_min}
h_{c-}(\rho)&=&(\rho-c)^-,\quad g_{i,c-}(\rho)
=-{\bf 1}_{\{\rho<c\}}[f_i(\rho)-f_i(c)],\quad c\in[0,1].
\end{eqnarray}
\textit{Indeed, each of these families can be obtained by regularization of 
smooth convex entropies, and conversely generates by mixture all 
convex entropies (up to linear functions). }
\end{remark}
\begin{remark}
\label{remark_scalar}
\textit{Specializing \eqref{entropy_cond}--\eqref{initial_cond} 
to the case $n=1$, $c_0=0$, $f_0=f$, $\rho_0(.,.)=\rho(.,.)$, 
$\rho^0_0(.)=\rho^0(.)$, we recover the usual definition 
of entropy solutions to the scalar conservation law}
\be\label{relax_law}
\partial_t \rho(t,x)+\partial_x[f(\rho(t,x))]=0,
\ee
 namely
\be\label{entropy_namely}
\partial_t h[\rho(t,x)]+\partial_x g[\rho(t,x)]\leq 0
\ee
for any entropy-flux pair $(h,g)$.
\end{remark}
In the sequel,  elements of $[0;1]^n$ will be denoted in vector form, 
$\vec{\rho}=(\rho_0,\ldots,\rho_{n-1})$, where $\rho_i\in[0;1]$ for every $i\in\{0,\ldots,n-1\}$. This notation will also be used if $\vec{\rho}$ denotes a $[0;1]^n$-valued function, its components being scalar functions $\rho_i$ for $i\in\{0,\ldots,n-1\}$.
 We now state the relaxation and equilibrium limit for 
 \eqref{relax_system} (in fact, a slightly more general version 
 thereof based on \eqref{general_system}).
 Recall Assumptions \ref{assumption_ker} and \ref{assumption_rev}. 
\begin{theorem} 
\label{th:relax_limit}
Let $\vec{\rho}_\varepsilon(.,.)=[\rho_{0,\varepsilon}(.,.),
\ldots,\rho_{n-1,\varepsilon}(.,.)]$ denote the entropy solution to 
\be\label{general_relax_system}
\partial_t\rho_i+\partial_x[f_i(\rho_i)] 
= \varepsilon^{-1}c_i(\rho_0,\ldots,\rho_{n-1}),\quad  i=0,\ldots,n-1
\ee
with initial data 
$\vec{\rho}^{\,0}(.)=[\rho^0_0(.),\ldots,\rho^0_{n-1}(.)]$, where $c_i$ 
%
is of the form
\be\label{form_relax_1}
c_i(\vec{\rho})=\sum_{j=0}^{n-1}[c_{ji}(\rho)-c_{ij}(\rho)]
\ee
with
\be\label{form_relax_2}
c_{ij}(\vec{\rho})=q(i,j)\rho_i(1-\rho_j)
\ee
where we assume that  the kernel $q(.,.)$ on $\{0,\ldots,n-1\}$ 
satisfies assumption \ref{assumption_ker}, (ii) 
and assumption \ref{assumption_rev}. 
Define the {\rm relaxation flux function} $f:[0,1]\to\R$ by
\be\label{relax_flux}
f(\rho):=\sum_{i=0}^{n-1}f_i[\widetilde{\rho}_i(\rho)]
\ee
where $\widetilde{\rho}_i(\rho)$ is defined in Proposition \ref{lemma_phi_super}.
Then, as $\varepsilon\to 0$:\\ 
(i) (Relaxation limit) The global density field
\be\label{global_density}
R_\varepsilon(t,x):=\sum_{i=0}^{n-1}\rho_{i,\varepsilon}(t,x)
\ee
converges in $L^1_{\rm loc}((0,+\infty)\times\R)$ to $\rho(.,.)$ 
defined as the entropy solution to \eqref{relax_law}
with initial data
\be\label{initial_relax}
R^0(x):=\sum_{i=0}^{n-1}\rho^0_i(x).
\ee
(ii) (Equilibrium limit). For each $i=0,\ldots,n-1$, 
$\rho_{i,\varepsilon}(.,.)$ converges in  $L^1_{\rm loc}((0,+\infty)\times\R)$ 
to $\rho_i(.,.)$ defined by
\be\label{lane_density}
\rho_i(t,x)=\widetilde{\rho}_i[\rho(t,x)].
\ee  
\end{theorem}
\begin{remark}\label{rk:bounded_sol}
The balance terms defined by \eqref{form_relax_1}--\eqref{form_relax_2} 
do satisfy condition \eqref{assumption_system_exclusion_2}. Indeed if 
$\rho_i=1$ (resp. $\rho_i=0$), the first (resp. second) term in brackets 
in \eqref{form_relax_1} vanishes, thus $c_i(\vec{\rho})$ is nonpositive 
(resp. nonnegative).
\end{remark}

\begin{remark}\label{rk:scaling}
Recalling the remark after \eqref{assumption_relax_weak}, 
when $\theta(N)$ has a finite limit, the hydrodynamic limit is expected 
to be governed by entropy solutions of \eqref{relax_system}, where  
$\theta(N)$ is replaced by this limit. The corresponding notion of 
entropy solution is recalled in Subsection \ref{subsec:relax} below. 
We leave this result for future investigation. \\ \\
In light of this, Theorem \ref{thm:local_equilibrium} can be understood 
as follows.  Suppose we first take constant $\theta(N)=\varepsilon^{-1}$ 
and $N\to+\infty$.  The above conjecture predicts the hydrodynamic limit 
\eqref{general_system}, \eqref{form_relax_1}--\eqref{form_relax_2}, 
with $f_i(\rho_i)=\gamma_i\rho_i(1-\rho_i)$. Next, letting 
$\varepsilon\to 0$ in \eqref{general_system} leads to the 
relaxation limit \eqref{relax_law}, \eqref{relax_flux} 
in Theorem \ref{th:relax_limit}. Thus, this double limit yields the 
same limit as in Theorem \ref{thm:local_equilibrium}, where 
$\varepsilon\to 0$ and $N\to+\infty$ simultaneously.
\end{remark}
\section{Examples and  a singular limit}\label{subsec:examples}
 In order to state examples and some properties, 
 it is convenient to reformulate Definition \ref{def_irred} 
 of weak irreducibility in an equivalent way. 
 The following Lemma is proved in Section \ref{sec:proof_inv}. 
\begin{lemma}\label{lemma_class}
Let $\pi(.,.)$ be the transition kernel for 
a Markov jump process on a finite set $S$. Then
the following statements are equivalent:\\ 
(1)  The kernel $\pi(.,.)$ is weakly irreducible.\\ 
(2) There is a unique labeling $(\Gamma_\alpha)_{\alpha=0,\ldots,m-1}$ 
of the irreducibility classes of $\pi(.,.)$
that satisfies the following property: for every $\alpha\in\{0,\ldots,m-2\}$, 
there exists 
$i\in\Gamma_\alpha$ and $j\in\Gamma_{\alpha+1}$ such that $\pi(i,j)>0$. \\ 
Then, $\Gamma_{m-1}$ is the unique recurrent class.
\end{lemma}
Denote by $(\mathcal C_\alpha)_{\alpha=0,\ldots,m-1}$, the labeling
of the irreducibility classes of $q(.,.)$ induced by condition \textit{(ii)} 
of Assumption \ref{assumption_ker} and Lemma \ref{lemma_class}. 
We arbitrarily select a nonzero reversible measure on $\mathcal C_\alpha$ 
and denote it by $\lambda^\alpha_.$.
For $\alpha=0,\ldots,m-1$,
let us define
\be\label{lane_numbers}
n_\alpha:=\left|\mathcal C_\alpha\right|,\quad
N_\alpha:= 
\sum_{\beta=\alpha+1}^{m-1} n_\beta
\ee
to be, respectively, the number of lanes in class  $\mathcal C_\alpha$, 
and the number of lanes  ahead of this class. 
In particular, $N_{m-1}=0$; by extension we set $N_{-1}=n$.
Lane $i\in \{0,\ldots,n-1\}$ belongs to class  $\mathcal C_\alpha$ 
for $\alpha\in\{0,\ldots,m-1\}$ if and only if 
 $i\in\left\{
n-N_{\alpha-1}+1,\ldots,n-N_\alpha
\right\}$.  \\ \\
We next illustrate condition \textit{(ii)} of 
Assumption \ref{assumption_ker}  and Assumption \ref{assumption_rev} 
for the ``vertical'' (that is, interlane) jump kernel.   
\begin{example}\label{ex:twolane}
Consider the two-lane SEP, that is $W=\{0,1\}$ 
 and  transition probability
 \be\label{eq:intensity_c-2lane}  
 q(0;1)=p,\quad q(1;0)=q,
 \ee 
 with $p,q\geq 0$ and $p+q>0$. \\ 
\textit{(i)} If $pq>0$, then $m=1$, $\mathcal C_0=\{0,1\}$, 
$n_0=2$, $N_0=0$, $N_{-1}=2$.
A reversible measure is given by $\lambda^\alpha_0=1$, 
$\lambda^\alpha_1=\frac{p}{q}$.\\ 
\textit{(ii)} If $q=0<p$, then $m=2$,  $\mathcal C_0=\{0\}$,  
$\mathcal C_1=\{1\}$,  $n_0=n_1=1$, $N_1=0$, $N_0=1$, $N_{-1}=2$.
For $\alpha\in\{0,1\}$ and  the unique element $i$ of $\mathcal C_\alpha$, 
any value  $\lambda^\alpha_i>0$ yields a nonzero  reversible measure 
on $\mathcal C_\alpha$.
\end{example}
For three lanes, we have the following possibilities 
(up to a permutation of lanes).
\begin{example}\label{ex:three}
Assume $W=\{0,1,2\}$, $q(0,1)=p>0$, $q(1,0)=q\geq 0$, $q(1,2)=r>0$, 
$q(2,1)=s\geq 0$, $q(0,2)=t\geq 0$, $q(2,0)=u\geq 0$. \\ 
\textit{(i)} If $q=u=s=0$, then $m=3$, $\mathcal C_\alpha=\{\alpha\}$ 
for $\alpha\in\{0,1,2\}$,
$n_0=n_1=n_2=1$, $N_2=0$, $N_1=1$,  $N_0=2$, $N_{-1}=3$.  Any constant 
$\lambda^\alpha_.>0$ yields a nonzero  reversible measure on $\mathcal C_\alpha$.\\ 
\textit{(ii)} If $u=s=0<q$, then $m=2$, $\mathcal C_0=\{0,1\}$, 
$\mathcal C_1=\{2\}$, $n_0=2$, $n_1=1$, $N_1=0$, $N_0=1$, $N_{-1}=3$.
Assumption \ref{assumption_rev} holds for instance with
$\lambda^0_0=1$, $\lambda^0_1=\frac{p}{q}$ and $\lambda^1_2=1$.\\ 
\textit{(iii)} If $q>0$ and $s>0$, then $m=1$, $\mathcal C_\alpha=\{0,1,2\}$, $N_0=0$, 
$N_{-1}=3$.  Assumption \ref{assumption_rev}  holds if and only if
\be\label{cond_rev}
 t=u=0,\quad\mbox{or}\quad u\neq 0\quad 
 \mbox{ and}\quad\frac{t}{u}=\frac{p}{q}\frac{r}{s}.
\ee
In this case, a reversible measure is given by
\be\label{rev_meas_3}
\lambda^0_0=1,\quad
\lambda^0_1=\frac{p}{q},\quad
\lambda^0_2=\frac{p}{q}\frac{r}{s}=\frac{t}{u}.
\ee
\end{example}
\begin{example}\label{ex:nn}
We can generalize Example \ref{ex:twolane} to $n$ lanes, 
with $W=\{0,\ldots,n-1\}$, and the kernel $q(.,.)$ given by 
\begin{equation}
q\left(i,j\right)=\left\{ \begin{array}{ccc}
 0 & \mbox{if} & \vert i-j\vert\neq 1\\
 p &  \mbox{if} &  j=i+1<n \\
 q  & \mbox{if} &  j=i-1>0
\end{array}\right.\label{eq:intensity_c}
\end{equation}
with $p,q\geq 0$ and $p+q>0$.
In other words,  considering $\mathbb{L}_0$ as the top lane 
and $\mathbb{L}_{n-1}$ as the bottom lane, $p$ is the rate 
at which particles go down and  $q$ the rate of going up. 
Without loss of generality, let us consider $p>0$ and $q\geq 0$. 
Then there are two cases: \\ 
{\em Case 1.} If $q>0$, then $m=0$ and $\mathcal C_0=W$. 
A reversible measure on $\mathcal C_0$ is
\be\label{rev_nlane}
\lambda^0_i=\left(\frac{p}{q}\right)^i,\quad i\in\mathcal C_0.
\ee
{\em Case 2.} If $q=0$, then $m=n$ and $\mathcal C_i=\{i\}$ for $i\in W$. 
A reversible measure on $\mathcal C_i$ is any positive constant value 
associated with $i$.
\end{example}
 In the following example  the kernel $q(.,.)$ 
is irreducible and  Assumption \ref{assumption_rev} 
is satisfied.  In the traffic interpretation, it could model 
a highway interchange. 
\begin{example}\label{example_tree}
Let $W$ be the set of vertices of a tree. If $i,j\in W$, 
we write 
$i\sim j$ if and only if 
$i$ and $j$ are connected by an 
edge of the tree. 
 We assume that the kernel $q(.,.)$ in \eqref{restrict_kernel} satisfies
\be\label{assume_q}
q(i,j)>0\mbox{ if and only if } i\sim j.
\ee 
\end{example}
\begin{proof}[Proof of  Assumption \ref{assumption_rev}]
For $i,j\in W$  and $i\sim j$, we define
\be\label{def_rij_neighbor}
r(i,j):=\frac{q(i,j)}{q(j,i)},\quad i,j\in W.
\ee
We extend this definition to an arbitrary $(i,j)\in W^2$ by setting 
\be\label{def_rij}
r(i,j)=\prod_{k=0}^{l-1}r(i_k,i_{k+1})
\ee
where $(i=i_0,\ldots,i_l=j)$ is the unique path  from $i$ to $j$ on the tree.  
Fix an arbitrary vertex $i^*\in W$ (for convenience the reader may think of 
$i^*$ as the root of the tree, but in fact we do not need the tree to be rooted).  
For every $\lambda_{i^*}\geq 0$, define a family $(\lambda_i)_{i\in W}$   
by the relation 
\be\label{rev_tree}
\lambda_i= \lambda_{i^*}r(i^*,i)
\ee
where  $(i_0=i^*,\ldots,i_l=i)$  is the unique path connecting  $i^*$  to $i$. 
Then $\lambda_.$ satisfies \eqref{reversible_class} for all $i,j\in W$.
Conversely, any $\lambda_.$  satisfying \eqref{reversible_class} for all 
 $i,j\in\{0,\ldots,n-1\}$  is of the form 
\eqref{rev_tree}.
\end{proof}
  In relation to \eqref{detailed_q}--\eqref{def_set_f},
we  now define
a mapping $\phi_r$ from $[0,1]$ to $[0,1]$ by
\be\label{def_phi}
\phi_r(\rho):=\frac{r\rho}{1-\rho+r\rho},\quad\forall\rho\in[0,1].
\ee
Note that $\phi_r$ is an increasing continuous bijection, and
\be\label{properties_phi}
\phi_{r'}\circ\phi_r=\phi_{r{r'}},\quad\phi_{1/r}=\phi_{r}^{-1}.
\ee
Defining again $r(i,j)$ by \eqref{def_rij_neighbor} whenever $q(j,i)>0$,
the relation \eqref{detailed_q} becomes 
\be\label{detailed_eq}
\rho_j=\phi_{r(i,j)}(\rho_i).
\ee
By \eqref{properties_phi}  and \eqref{def_rij}, 
\be\label{down_the_tree}
\phi_{r(i,j)}=\phi_{r(i_{m-1},i_m)}\circ\cdots\circ
\phi_{r(i_0,i_1)}.
\ee
Then, we have 
\be\label{def_subset_gen}
\mathcal F:=\left\{
(\rho_i)_{i\in W}\in[0,1]^W:\, \forall i\in W,\,\rho_i=\phi_{r(i^*,i)}(\rho_{i^*})
\right\}.
\ee
Equation \eqref{def_subset_gen} gives an explicit construction of 
$\mathcal F$ as indexed by the density at  some ``reference'' vertex $i^*$.
Properties \eqref{properties_phi} and \eqref{down_the_tree} 
imply that the set $\mathcal F$ given by \eqref{def_subset_gen} 
does not depend on the choice of $i^*$. 
\subsection{The shape of the flux function and a phase transition}\label{subsec:shape}
The following proposition,  proved in Section \ref{sec:proof_inv},  
states some properties of the flux function. 
 Recall the notation introduced in \eqref{lane_numbers}. 
\begin{proposition}\label{prop:flux_multi}
\mbox{}\\ \\
(i) For each $\alpha\in\{0,\ldots,m-1\}$, the restriction of $G$ 
to  $[N_{\alpha},N_{\alpha-1}]$ 
is continuously differentiable and depends only on the restriction 
of $q(.,.)$ to $\mathcal C^\alpha\times\mathcal C^\alpha$.\\ 
(ii)  The flux function 
$G$ has the following one-sided derivatives 
at $N_\alpha$:   for $-1\leq \alpha<m-1$, resp. $-1<\alpha\leq m-1$, it holds that 
\be\label{one_sided}
G'(N_\alpha-)= -\sum_{i\in\mathcal C_{\alpha+1}}\gamma_{i} f^{\alpha+1}_i,
\quad \mbox{resp.}\quad
G'(N_\alpha+)=\sum_{i\in\mathcal C_\alpha} \gamma_{i}e^{\alpha}_i
\ee
 where 
\be\label{def_ealpha}
e_i^\alpha:=
\frac{\lambda^\alpha_i}{\sum_{j\in W}\lambda^\alpha_j},
\quad
f_i^\alpha:=
\frac{\left(\lambda^\alpha_i\right)^{-1}}{\sum_{j\in W}\left(\lambda^\alpha_j\right)^{-1}}
\ee
 (iii) Let $\alpha\in\{0,\ldots,m-1\}$. If there exists 
 $i\in\mathcal C_\alpha$ such that
\be\label{cond_diff}
\sum_{j\in\mathcal C_\alpha:\,\lambda^\alpha_j=\lambda^\alpha_i}\gamma_j\neq 0,
\ee
then the derivative of $G$ vanishes finitely many times 
on $[N_\alpha,N_{\alpha-1}]$. Otherwise, 
$G$ is identically $0$ on $[N_\alpha,N_{\alpha-1}]$.
\end{proposition}
\begin{remark}\label{remark_der}
\mbox{}\\ \\
\textit{1. It follows from \textit{(ii)} of  Proposition  \ref{prop:flux_multi} 
that if lanes in classes $\alpha$ and $\alpha+1$ have either all 
nonnegative or all nonpositive drifts (with the two classes on the 
same side), and at least one lane in at least one  of these classes has nonzero drift,
then $G$ is not differentiable at $N_\alpha$.
This generalizes the 
non-differentiability of the flux in Example \ref{ex:periodic} below.\\ 
2. By condition \eqref{reversible_class}, level sets of 
$\lambda^\alpha_.$ are singletons or maximal subsets of $\mathcal C_\alpha$
on which $q(.,.)$ is symmetric. Condition \eqref{cond_diff} can be rephrased 
saying that the total drift on at least one such subset
is nonzero. \\ 
3. Proposition \ref{prop:flux_multi}, \textit{(iii)} implies that 
for a given integer $k\in\{1,\ldots,n\}$,  the equation 
 $G(\rho+k)-G(\rho)=0$  (with unknown $\rho\in[0,n]$) 
 has finitely many solutions. This implies that there are only 
 finitely many possible shocks of integer amplitude for 
 the hydrodynamic equation. This property plays an important role 
 to extend the arguments and results of \cite[Theorem 2.2]{mlt1a}, 
 and prove that up to 
  horizontal  translations, there are only 
 finitely many measures in $\mathcal I_e\setminus(\mathcal I\cap\mathcal S)_e$.}
\end{remark}
We illustrate Theorem \ref{thm:local_equilibrium} and 
Proposition \ref{prop:flux_multi}  by computing flux functions  for our examples.   
 We begin with the two-lane model of Example \ref{ex:twolane} 
 that we  develop later in this section. 
 The following computation is taken from \cite[Section 5.2]{mlt1a}, 
 where 
 \be\label{def_ratio}r:=\frac{q}{p}\in[0;+\infty]
\ee
(if  $p=0$  the expressions below must 
 be taken in the sense of their limit as $r\to+\infty)$.  
\begin{equation}
G(\rho)=  (\gamma_0+\gamma_1)\frac{\rho}{2} 
\left(1-\frac{\rho}{2}\right) 
+  (\gamma_0-\gamma_1)(1-\rho) \varphi(\rho)
-  (\gamma_0+\gamma_1)\varphi(\rho)^2  \label{eq:Gviaphi_0}
\end{equation}
with
\begin{eqnarray}
\varphi(\rho)
&=& \frac{1}{2}\left(\frac{r+1}{r-1}\right)
\left(1-\sqrt{\psi(\rho)}\right)\mbox{ if }r\neq 1\nonumber\\
&  =  &  0  \mbox{ if } r=1\label{eq:phi}\\
\label{eq:psi}
\psi(\rho)&=& 1+ \left(\frac{r-1}{r+1}\right)^2 \rho(\rho-2).
\end{eqnarray}
In particular, $G$ depends on $p,q$ only through $r$. 
To emphasize dependence of $G$ on the parameters of the model, 
whenever necessary, we write $G=G_{\gamma_0,\gamma_1,r}$.
The following properties, taken from 
\cite[Proposition 4.8 and Section 5.2]{mlt1a}, are useful 
to reduce the number and range of parameters of the model. 
\begin{proposition}\label{prop:mlt1}
The function $G_{\gamma_0,\gamma_1,r}$ satisfies the following:
\begin{enumerate}
\item Symmetry:
\be\label{sym}
G_{\gamma_0,\gamma_1,r}(2-\rho)=G_{\gamma_1,\gamma_0,r}(\rho)
=G_{\gamma_0,\gamma_1,r^{-1}}(\rho).
\ee
\item Homogeneity: if $\gamma_0+\gamma_1\neq 0$,
\be\label{hom}
G_{\gamma_0,\gamma_1,r}(\rho)=(\gamma_0+\gamma_1)
G_{\frac{\gamma_0}{\gamma_0+\gamma_1},\frac{\gamma_1}{\gamma_0+\gamma_1},r}(\rho).
\ee
\item \label{sign_prop} Sign: if $\gamma_0\gamma_1<0$, 
$G$ vanishes on $(0;2)$ if and only if
\be\label{cond_vanish}
\min\left(
q\gamma_0+p\gamma_1;p\gamma_0+q\gamma_1
\right)<0.
\ee
In this case, the zero is unique,
$G$ has a single minimum and a single maximum,  of opposite signs, on $(0;2)$. 
\end{enumerate}
\end{proposition}
\noindent In the following special cases, 
the function $G$  
has a simple explicit expression.
\begin{example}\label{ex:periodic}
Assume $p=0<q$. Then 
\be\label{flux_degenerate}
G(\rho)=\left\{
\ba{lll}
\dsp \gamma_0\rho(1-\rho) & \mbox{if} & \rho\in[0,1]\\ 
\dsp \gamma_1(\rho-1)(2-\rho) & \mbox{if} & \rho\in(1,2].
\ea
\right.
\ee
In particular, when $\gamma_0=\gamma_1$,  the flux is a function of period $1$ 
 whose restriction to 
$[0,1]$ is the TASEP flux. 
Note that  there is a point of non-differentiability ($\rho=1$), 
a property  not seen in usual  
single-lane models with product invariant measures. 
 Moreover, if $\gamma_0\gamma_1=0$, 
the flux function 
is strictly concave on $[0;1]$ or $[1;2]$ and identically zero 
on its complement: similar unusual behaviour has been established 
recently for the asymmetric facilitated exclusion process (\cite{ESZ}). 
\end{example}
\begin{example}\label{ex:sym}
Assume $p=q>0$. Then 
\be\label{flux_sym}
G(\rho)=\frac{\gamma_0+\gamma_1}{4}\rho(2-\rho).
\ee
Here, unless $\gamma_0+\gamma_1=0$, the flux has 
the same shape as the single-lane TASEP flux 
(from which it is obtained by a scale change). 
It is in particular strictly concave.
\end{example}
\begin{remark}\label{rk:incases}
\textit{In cases $\gamma_0=\gamma_1=0$ of Example \ref{ex:periodic}, 
and $\gamma_0+\gamma_1=0$ of Example \ref{ex:sym}, the flux function 
is identically $0$, and the entropy solution to 
\eqref{eq:hydrodynamic equation} does not evolve in time. 
In these situations, the hyperbolic time scale is irrelevant, 
and one expects a nontrivial diffusive hydrodynamic limit 
to arise under diffusive time scaling.}
\end{remark}
 Following are examples beyond the two-lane model. 
\begin{example}
\label{flux_ex:three}
Consider case \textit{(ii)} in Example \ref{ex:three}. Then
\be\label{flux_ex_three}
G(\rho)=\left\{
\ba{lll}
\gamma_2 \rho(1-\rho) & \mbox{if} & \rho\in[0,1]\\
G_{\gamma_0,\gamma_1,q/p}(\rho-1) & \mbox{if} & \rho\in[1,3]
\ea
\right.
\ee
 where the second line refers to the two-lane flux. 
For $\alpha=0$, unless $p=q$ and $\gamma_1=-\gamma_0$, 
condition \eqref{cond_diff} holds for the flux function 
on $[N_\alpha,N_{\alpha-1}]=[1,3]$.
\end{example}
\begin{example}\label{ex:periodic_multi}
Consider Example \ref{ex:nn} with $p=1$, $q=0$ and a TASEP 
from left to right on each lane, that is $0=l_i<d_i$.
In this case we have
\be\label{rhotilde_multi}
\widetilde{\rho}_i(\rho)=(\rho-i){\bf 1}_{[i,i+1]}(\rho)+{\bf 1}_{\{i +1<\rho\}}
\ee
so that we obtain the flux function
\be\label{flux_hills}
G^n(\rho)=\sum_{i=0}^{n-1}d_i(\rho -i)(i +1 -\rho){\bf 1}_{(i,i+1)}(\rho)
\ee
that consists of a succession of hills of height $d_i$ separated by valleys. 
\end{example}
For the model  of Example \ref{ex:nn},  in view of \eqref{rhotilde_multi},
part \textit{(ii)} of Theorem \ref{thm:local_equilibrium} says that the density 
on lane $i=\lfloor  u(t,x)\rfloor$ is $ u(t,x)-i$, lanes with lower index 
are full and lanes with higher index are empty. Thus all the current 
is carried out by the lane of index $\lfloor  u(t,x)\rfloor$. \\ \\
 We now come back to the two-lane model. 
The following theorem yields precise information on the shape 
of the flux function for the two-lane model \eqref{eq:intensity_c-2lane} 
and shows in particular that  phase transitions occur 
 according to the parameters of the model. 
We prove that the flux function $G$ 
may have either zero, one or two inflexion points, and possibly 
change sign at critical values that can be expressed explicitly. 
\begin{theorem}
\label{thm:phase_trans}
We assume that
 $d\in[1/2;+\infty)$ and $r\in[1;+\infty)$. 
Then the number of inflexion points of $G=G_{d,1-d,r}$ 
is given by the following phase diagram. \\ \\
Let 
\be\label{critical_d}
\frac{1}{2}<\widetilde{d}_1:=\frac{1}{2}+\frac{1}{6}\sqrt{3+2\sqrt{3}}<
\widetilde{d}_0:=\frac{1}{2}+\frac{\sqrt{3}}{4}
<\overline{d}_1:=\frac{1}{2}+\frac{1}{4}\sqrt{2\sqrt{3}}.
\ee
Then there exist functions  of $d$ 
\be\label{critical_r}
\overline{r}_1:[ 1/2;\overline{d}_1]\to[1;+\infty),\,
r_3:[\widetilde{d}_1;+\infty)\to[1;+\infty),\,
r_4:[\widetilde{d}_1;1)\to[1;+\infty)
\ee
with the following properties:\\ \\
1, (i): 
$r_4$ is increasing  in $d$, 
$\overline{r}_1$ and $r_3$ are decreasing  in $d$, 
 and all these three functions are continuous; besides,
\be\label{lim_r4}
\lim_{d\to 1-}r_4(d)=+\infty.
\ee
1, (ii): It holds that
\begin{eqnarray*}
d=\widetilde{d}_0 & \Rightarrow & 
\overline{r}_1(d)=r_3(d) \\
d\neq\widetilde{d}_0 & \Rightarrow & \overline{r}_1(d)<r_3(d).
\end{eqnarray*}
2, (i): For $d\in[1/2;\widetilde{d}_1]$:
if $r\leq \overline{r}_1(d)$, $G$ is strictly concave; 
if $r>\overline{r}_1(d)$, $G$ has two inflexion points;\\ 
2, (ii): For $d\in[\widetilde{d}_1;\widetilde{d}_0]$:
if $r\leq \overline{r}_1(d)$, $G$ is strictly concave; 
if $\overline{r}_1(d)<r< r_3(d)$, or $r>r_4(d)$, $G$ has  two inflexion points; 
if $r_3(d)\leq r\leq r_4(d)$, $G$ has a single inflexion point;\\ 
2, (iii): For $d\in[\widetilde{d}_0;1)$:
if $r\leq r_3(d)$, $G$ is strictly concave; 
if $r_3(d)< r< r_4(d)$, $G$ has a single inflexion point; 
if $r>r_4(d)$, $G$ has  two inflexion points. \\ 
2, (iv): For $d=1$: 
if $r\leq r_3(1)=1+\sqrt{2}$, $G$ is strictly concave; 
if  $r>r_3(1)$,  $G$ has a single inflexion point. \\ 
2, (v): For $d\in(1;+\infty)$:
if $r\leq r_3(d)$, $G$ is positive on $(0;2)$ and strictly concave; 
if  $r_3(d)<r\leq \frac{d}{d-1}$,  $G$ is positive on $(0;2)$ and 
has a single inflexion point; 
if $r>\frac{d}{d-1}$, $G$  has 
a single inflexion point, a single zero on $(0;2)$, is positive 
with a single local maximum on the left of this zero, 
negative with a single local minimum on the right.
\end{theorem}
Graphical illustrations of the above theorem,  and 
 interpretations thereof, are given in Appendix \ref{app:graph}.
\begin{remark}\label{rk:phases}
\mbox{}\\ \\
\textit{(i)} \textit{ By the symmetry and homogeneity properties 
in Proposition \ref{prop:mlt1}, 
the above theorem yields the number of inflexion points of 
$G_{\gamma_0,\gamma_1,r}$ for all values of $\gamma_0\in\R$, $\gamma_1\in\R$ 
and  $r>0$ such that $(\gamma_0,\gamma_1)\neq (0,0)$. The case 
$\gamma_0+\gamma_1=0$ is attained by the limit }
\be\label{lim_G}
\lim_{d\to+\infty}\frac{1}{d}G_{d,1-d,r}=G_{1,-1,r}.
\ee
\textit{(ii)} \textit{ The value of  $\overline{r}_1(d)$ is explicit: 
cf.  \eqref{def:A1}--\eqref{def:bar-r1}.
The values of $r_3(d)$ and $r_4(d)$ could be written explicitly 
as roots of a cubic equation, but we omit the complicated formulas which  
would shed no additional light
on our results.}\\ 
\textit{(iii)} \textit{Cases 2, \textit{(iv)--(v)} are (as regards inflexion points) 
a natural extension 
of 2, \textit{(iii)} and could be included in the latter using the extension 
$r_4(d)=+\infty$ for $d\in[1;+\infty)$.}
\end{remark}
 \subsection{Many-lane limit}\label{subsec:many}
 In this subsection,  we investigate a singular behaviour of this model 
as the number of lanes  
 $n\to+\infty$ and the density is properly renormalized.  
We consider  the $n$-lane model 
 of Example \ref{ex:nn}. 
For each $n$, we consider a sequence $(\eta^{n,N})_N$ of such processes.
Since the density range for the total density is $[0,n]$, to have a fixed range 
$[0,1]$, we consider for \eqref{config_total}
the empirical measure renormalized by the number of lanes:
\be\label{empirical_norm}
\widehat{\alpha}^{N}\left(\overline{\eta}^{n,N},dx\right)
:=n^{-1}\alpha^{N}\left(\overline{\eta}^{n,N},dx\right).
\ee
For each $n$, for a given Cauchy datum $u^n_0$, by Theorem  
\ref{thm:local_equilibrium}, the unnormalized empirical measure 
$\alpha^{n,N}_{Nt}$ converges as $N\to+\infty$ to   $u^n(t,.)$,  
the entropy solution at time $t$ of \eqref{eq:hydrodynamic equation}
with a flux function $G^n$ defined on $[0,n]$ by \eqref{flux_hills}. 
The normalized measure \eqref{empirical_norm} 
 converges to  the normalized profile 
 $\widehat{u}^n(t,.):=n^{-1}u^n(t,.)$.  The latter is the entropy solution 
to the normalized conservation law
\be\label{norm_burgers}
\partial_t u +\partial_x \widehat{G}^n(u)=0
\ee
where the normalized flux function
\be\label{normalized_flux}
\widehat{G}^n(u):=n^{-1}G^n(nu)
\ee
is now defined on $[0,1]$. Let $F\in C^{1}[0,1]$ be  positive 
on $(0,1)$ with $F\left(0\right)=F\left(1\right)=0$.
In  Example \ref{ex:nn},  we consider
\be\label{slow_multi}
d_i=4nF(i/n).
\ee
The resulting normalized flux 
\be\label{normalized_hills}
\widehat{G}^n(\rho)
=4\sum_{i=0}^{n-1}{\bf 1}_{(i/n,(i+1)/n)}(\rho)F(i/n)(n\rho-i)(i+1-n\rho)
\ee
is a highly oscillating function (as the hills now have width $1/n$) and the height
of the $i$-th local maxima is $F(i/n)$. Thus in the maximal current regime 
$\widehat{G}^n$ approximates $F$.
We are interested in the behaviour of $\widehat{u}^n$ as $n\to+\infty$ 
when the normalized initial condition $\widehat{u}^n_0$ is fixed.
In particular,  it is natural to ask if there is  any connection  
with the conservation law
\be\label{burgers_H}
\partial_t u+\partial_x F(u)=0.
\ee
The following theorem yields an answer for Riemann conditions $\widehat{u}_0$.  
In the following, for $x\in\R$ and $A\subset\R$, the distance from $x$ to $A$ 
is defined as usual by
\be\label{def_distance_point}
d(x,A):=\inf\{d(x,a):\,a\in A\}
\ee
\begin{theorem}
\label{thm:approximating_G}
Assume $\widehat{u}_0$ is the Riemann condition \eqref{eq:init-riemann}.
Define $\widehat{u}(.,1)$ as follows: if $\alpha>\beta$, $\widehat{u}(.,1)$ 
is the entropy solution at time $1$ of the Riemann problem for \eqref{burgers_H}. 
If $\alpha\leq\beta$, $\widehat{u}(.,1)=\widehat{u}_0(.)$. For $v\in\R$, 
let $\widehat{U}(v,1)$ denote the closed interval 
with bounds $\widehat{u}(v\pm,1)$. 
Then 
\be\label{convergence_interval}
 \lim_{n\to+\infty} d[\widehat{u}^n(v,1);\widehat{U}(v,1)]=0
\ee
locally uniformly in $v$. In particular, $\widehat{u}^n(.,1)$ 
converges to $\widehat{u}(.,1)$ at every point of continuity of the latter.
\end{theorem}
Since the Riemann problem is expected to characterize the Cauchy problem,
this result suggests the following. If we consider  the $n$-lane  model 
\eqref{eq:intensity_c} 
with $N\to+\infty$ and $n=n(N)\to+\infty$ with
$n/N\to 0$,  starting from a Cauchy datum $\widehat{u}_0$ for the 
normalized measure \eqref{empirical_norm}, the latter should converge 
at time $Nt$ to the solution $ \widehat{u}(t,.)$ of a singular evolution problem. 
Solutions of this problem should coincide with entropy solutions of \eqref{burgers_H} 
for nondecreasing initial data, remain fixed for nonincreasing data, and involve 
some interaction between these behaviors for general data.
 \section{Proofs of  Lemma \ref{lemma_class}, 
 Proposition \ref{thm:characterization_lemma_gen} 
 and Proposition \ref{prop:flux_multi} }\label{sec:proof_inv}
\begin{proof}[Proof of Lemma \ref{lemma_class}]
Let
 $(\mathcal K_\alpha)_{\alpha=0,\ldots,m-1}$ 
denote the irreducibility classes of $\pi(.,.)$ labeled in an arbitrary way. We write 
$\alpha\preceq\beta$ if there is a path from  $\mathcal K_\alpha$ to 
$\mathcal K_\beta$;  that is, if
there exists a path $(\alpha=u_0,\ldots,u_k=\beta)$ such that 
for every $l=0\ldots,k-1$, 
there is a possible transition from class  $\mathcal K_{u_l}$  to class  
$\mathcal K_{u_{l+1}}$. 
When $\alpha=\beta$, we consider by convention that both $\alpha\preceq\beta$ 
and $\beta\preceq\alpha$ hold (we consider by extension that there is a 
path of length one leading from  $\mathcal K_\alpha$  to itself).
We write $\alpha\prec\beta$ if $\alpha\preceq\beta$ and $\alpha\neq\beta$. 
The relation $\preceq$ is a total order on 
$\{0,\ldots,m-1\}$.
It follows that the finite set $\{0,\ldots,m-1\}$  is linearly ordered by 
$\preceq$. That is, there exists a (unique) sequence 
$\left(\alpha_0,\ldots,\alpha_{m-1}\right)$
such that
\be\label{linear_order}
\alpha_0\prec\cdots\prec\alpha_{m-1}
%
\ee 
Then, for $k=0,\ldots,m-1$, we set  $\Gamma_k:=\mathcal K_{\alpha_k}$. 
\end{proof}
The proof of Proposition \ref{lemma_phi_super} 
relies on Lemmas \ref{lemma_irred} and \ref{prop_f} below. 
\begin{lemma}\label{lemma_irred}
Let $\alpha=0,\ldots,m-1$, and 
$\rho_.=(\rho_i)_{i\in\mathcal C_\alpha}$ 
be a $[0,1]$-valued family. Then  $\rho_.$ satisfies
\eqref{detailed_q} for all  $i,j\in\mathcal C_\alpha$,
if and only if, either $\rho_.$ is the constant function with value $1$ 
on $\mathcal C_\alpha$, henceforth denoted by 
 $\rho^{\alpha,\infty}_c$; 
or $\rho_.$ is of the form 
\be\label{def_alphac}
\rho^{\alpha,c}_.:=
\frac{c\lambda^\alpha_.}{1+c\lambda^\alpha_.}
\ee
for $c\geq 0$. 
\end{lemma}
\begin{proof}
Assume $\rho_.$ satisfies \eqref{detailed_q} and there exists 
 $i\in\mathcal C_\alpha$  such that $\rho_i=1$. 
For every $j\in\mathcal C_\alpha$, 
by irreducibility, there is a path
$(i_0=i,\ldots,i_{k-1}=j)$ from $i$ to $j$. The relations 
\[
\rho_{i_{k}}\left(
1-\rho_{i_{k+1}}
\right)q\left(i_k,i_{k+1}\right)=
\rho_{i_{k+1}}\left(
1-\rho_{i_k}
\right)q\left(i_{k+1},i_{k}\right)
\]
with  $q\left(i_k,i_{k+1}\right)>0$ and $q\left(i_{k+1},i_{k}\right)>0$ 
imply $\rho_j=1$. Thus $\rho_.$ uniformly equals $1$. Assume now that 
$\rho_.$ never takes value $1$. Then, defining $\lambda_i$ by  
\[
\lambda_i:=\frac{\rho_i}{1-\rho_i},
\]
\eqref{detailed_q} is equivalent to \eqref{reversible_class}. Since $q(.,.)$ 
restricted to $\mathcal C_\alpha$ is irreducible, any reversible
measure on $\mathcal C_\alpha$ is a multiple of  $\lambda^\alpha_.$; 
thus so is $\lambda_.$.
\end{proof}
\begin{lemma}\label{prop_f}
The following are equivalent:\\
(i) The vector $(\rho_0,\ldots,\rho_{n-1})$ satisfies 
\eqref{detailed_q} for every $i=0,\ldots,n-1$;\\ 
(ii) There exists $\alpha\in\{0,\ldots,m-1\}$ such that, 
for every  $\beta<\alpha$, 
$\rho_.$ is identically $0$ on $\mathcal C_\beta$, and for every
 $\beta>\alpha$,  it is identically $1$ on $\mathcal C_\beta$.
For  such  $\alpha$, there exists a unique $c\in[0,+\infty]$
such that  $\rho_.=\rho^{\alpha,c}_.$  on $\mathcal C_\alpha$. \\ 
Besides, if $c\in(0,+\infty)$, $\alpha$ is uniquely determined.
\end{lemma}
\begin{proof}
Let $\alpha\in\{0,\ldots,m-1\}$ and assume $\rho_.$ is neither of the 
constants $0$ or $1$ on $\mathcal C_\alpha$. By Lemma \ref{lemma_irred},
we have $\rho_i\in(0,1)$  for all $i\in\mathcal C_\alpha$.
Assume $\alpha<\beta$, $i\in\mathcal C_\alpha$ 
and $j\in\mathcal C_\beta$. Then there is a path
$(i_0=i,\ldots,i_{k-1}=j)$ such that $q\left(i_l,i_{l+1}\right)>0$ 
for every $l=0,\ldots,k-1$, and $q(i_{l+1},i_l)=0$ 
whenever $i_l$ and $i_{l+1}$ do not belong to the same class. 
\[
\rho_{i_l}\left(1-\rho_{i_{l+1}}\right)q\left(i_{l},i_{l+1}\right)=
\rho_{i_{l+1}}\left(1-\rho_{i_{l}}\right)q\left(i_{l+1},i_{l}\right).
\]
It follows that $\rho_l=1$ for $l\geq L$, where
\[
L=\min\{l=0,\ldots,k\}:\,i_l\not\in \mathcal C_\alpha\}.
\]
Thus $\rho_.$ is identically $1$ on $\mathcal C_\beta$. 
A similar argument shows that 
it is identically $0$ on all classes 
$\mathcal C_\beta$ with $\beta<\alpha$. 
The relation  $\rho_.=\rho^{\alpha,c}$  on $\mathcal C_\alpha$ follows 
from Lemma \ref{lemma_irred}.
If $\rho_.$ is neither of the constants $0$ or $1$ on $\mathcal C_\alpha$, 
$\alpha$ is uniquely determined by the fact that $\rho_.$ 
is identically $0$ or $1$ on all other classes.
\end{proof}
\begin{proof}[Proof of Proposition  \ref{lemma_phi_super}]
\mbox{}\\ \\
{\em Proof of (i).}
Let $\rho\in[0,n]$ and assume $\rho\in[N_{\alpha^*},N_{\alpha^*-1}]$ 
for some  $\alpha^*\in\{0,\ldots,m-1\}$,  where $N_\alpha$ 
was defined in \eqref{lane_numbers} (this $\alpha^*$ may not be unique, 
but as argued below after \eqref{sandwich_alpha}, this creates no ambiguity).  
Note that $\rho^{\alpha^*,c}_i$ 
defined by \eqref{def_alphac} increases from $0$ to $1$ as $c$ 
increases from $0$ to $+\infty$.
Thus there is a unique $c\in[0,+\infty]$ such that
\be\label{complete_density}
\sum_{i\in\mathcal C_{\alpha^*}}\rho^{\alpha^*,c}_i+
N_{\alpha^*}=\rho.
\ee
We then define $(\rho_0,\ldots,\rho_{n-1})$ as follows: 
\be\label{construction}
\rho_i=\left\{
\ba{lll}
 0  & \mbox{if} &  i\in\mathcal C_\beta\mbox{ for }\beta<\alpha^*\\
 1  & \mbox{if} &  i\in\mathcal C_\beta\mbox{ for }\alpha^*<\beta\\
 \rho^{\alpha^*,c}_i & \mbox{if} &  i\in\mathcal C_{\alpha^*}.
\ea
\right.
\ee
By Lemma \ref{prop_f},
we have $(\rho_0,\ldots,\rho_{n-1})\in\mathcal F$
(defined in \eqref{def_set_f}), 
and \eqref{complete_density} implies
\be\label{total_density}
\sum_{i=0}^{n-1}\rho_i=\rho.
\ee
We now prove uniqueness of $(\rho_0,\ldots,\rho_{n-1})\in\mathcal F$ 
satisfying \eqref{total_density}.
 Let $\alpha$ be as in (ii) of Lemma \ref{prop_f}. Then 
\be\label{sandwich_alpha}\rho\in[N_\alpha,N_{\alpha-1}].\ee
\textit{(a)} If $\rho\not\in\{N_{\alpha^*-1},N_{\alpha^*}\}$, 
then $\alpha=\alpha^*$ is uniquely determined. 
Thus the lemma implies that $(\rho_0,\ldots,\rho_{n-1})$ 
is given by \eqref{construction}. \\ 
\textit{(b)} If $\rho=N_{\alpha^*}$, by \eqref{sandwich_alpha}, 
we must have $\alpha=\alpha^*$ or $\alpha=\alpha^*+1$.
In the former case, by Lemma \ref{prop_f}, we recover \eqref{construction}.
 In the latter case, the lemma gives
$\rho_i=0$ if $i\in\mathcal C_\beta$ for  $\beta<\alpha^*+1$ 
and $\rho_i=1$ if $i\in\mathcal C_\beta$ for $\alpha^*+1<\beta$. 
Thus
\be\label{decomp_total}
\rho=N_{\alpha^*}=\sum_{i\in\mathcal C_{\alpha^*+1}}\rho_i+N_{\alpha^*+1}.
\ee
 where $N_\alpha$ was defined in \eqref{lane_numbers}. 
Hence, the sum in \eqref{decomp_total} must be equal 
to $n_{\alpha^*+1}$, so $\rho_i=1$ for every $i\in\mathcal C_{\alpha^*+1}$. 
This gives the same $(\rho_0,\ldots,\rho_{n-1})$ as in \eqref{construction}.\\ 
\textit{(c)} If $\rho=N_{\alpha^*-1}$, a similar argument than in \textit{(b)} 
shows that we recover \eqref{construction}.
\end{proof}
 We now turn to the proof of Proposition \ref{prop:flux_multi}. 
 We shall need the following lemma, which completes the proof of
 Proposition \ref{lemma_phi_super}.
\begin{lemma}\label{lemma:complete}
 For $i=0,\ldots,n-1$, the function $\widetilde{\rho}_i(.)$ 
defined in Proposition \ref{lemma_phi_super} satisfies the following: \\ \\
(i) It is identically $0$ on $[0,N_{\alpha}]$, and identically
$1$ on $[N_{\alpha-1},n]$, where  $\alpha\in\{0,\ldots,m-1\}$ 
 is  such that $i$ belongs to  $\mathcal C_\alpha$. 
 Its restriction to $[N_{\alpha},N_{\alpha-1}]$ depends 
only on the restriction of $q(.,.)$ to 
$\mathcal C_\alpha\times\mathcal C_\alpha$. \\ \\
(ii)   It is increasing 
and continuously differentiable on $[N_{\alpha},N_{\alpha-1}]$, and
\be\label{limit_derivative}
\lim_{\rho\to N_{\alpha-1}-}\widetilde{\rho}_i'(\rho)
=f^\alpha_i>0, \quad\lim_{\rho\to N_{\alpha}+}\widetilde{\rho}_i'(\rho)
=e^{\alpha+1}_i>0
\ee
where $e^\alpha_i$ and $f^{\alpha}_i$ are defined by \eqref{def_ealpha}. 
\end{lemma}
\begin{proof}[Proof of lemma \ref{lemma:complete}]
\mbox{}\\ \\
{\em Proof of (i).} 
The constant values $0$ and $1$ outside $[N_{\alpha},N_{\alpha-1}]$ 
follow from the first two lines of \eqref{construction}.  By the last line of 
\eqref{construction}, the restriction of $\widetilde{\rho}_i$ to 
$[N_\alpha,N_{\alpha-1}]$ depends only on $\lambda^\alpha_.$, 
which depends only on the restriction of $q(.,.)$ to 
$\mathcal C_\alpha\times\mathcal C_\alpha$. \\ \\
{\em Proof of (ii).} 
The regularity on $[N_{\alpha},N_{\alpha-1}]$ follows from 
the third line. Indeed,  $c\mapsto\rho^{\alpha,c}_i$ lies 
in $C^1([0,+\infty))$, and so does
\be\label{total_alpha}
c\mapsto\rho^\alpha(c):=\sum_{i\in\mathcal C_\alpha}\rho^{\alpha,c}_i+N_\alpha.
\ee
 Notice that $\rho^{\alpha,.}_i$ is an increasing function such that 
\be\label{notice_limits}
\rho^{\alpha,0}_i=0,\quad
\lim_{c\to+\infty}\rho^{\alpha,c}_i=1.
\ee
Hence $\rho^{\alpha}$ is an increasing function such that
\begin{eqnarray}
\label{limits_rhoalpha_1}
\rho^\alpha(0)=N_\alpha & , & \lim_{c\to+\infty}\rho^\alpha(c)=N_{\alpha-1}\\
\label{limits_rhoalpha_2}
\lim_{\rho\to N_{\alpha-1}^-}
(\rho^\alpha)^{-1}\left(\rho-N_\alpha\right)=+\infty & , & 
\lim_{\rho\to N_{\alpha}^+}(\rho^\alpha)^{-1}\left(\rho-N_\alpha\right)=0.
\end{eqnarray}
The derivative of $\rho^\alpha$ does not vanish.
Hence its inverse lies in $C^1([0,n_\alpha))$ and can be extended 
into an element of $C^0([0,n_\alpha])$. By \eqref{complete_density} and  the third line of 
\eqref{construction}, the restriction of $\widetilde{\rho}_i$ to 
$[N_\alpha,N_{\alpha-1}]$ is given by
\be\label{givenby}
\widetilde{\rho}_i(\rho)=\rho^{\alpha,c}_i,\quad\mbox{with}\quad c
=(\rho^\alpha)^{-1}\left(\rho-N_\alpha\right)
\ee
and thus lies in 
$C^1([N_\alpha,N_{\alpha-1}))\cap C^0([N_\alpha,N_{\alpha-1}])$. 
 It follows that
\be\label{using_givenby}
\widetilde{\rho}_i'(\rho)=\frac{
\frac{d}{dc}\rho^{\alpha,c}_i
}
{
\frac{d}{dc}\rho^\alpha(c)
}
\ee
with $c$ given by \eqref{givenby}. 
By \eqref{def_alphac} 
\be\label{direct_computation}
\frac{d}{dc}\rho^{\alpha,c}_i  =  \frac{\lambda_i^\alpha}{(1+c\lambda_i^\alpha)^2}\\
\ee
On the other hand, by \eqref{limits_rhoalpha_2}, $\rho\to N_\alpha+$ and 
$\rho\to N_{\alpha-1}-$ respectively correspond in \eqref{givenby} to $c\to 0+$ 
and $c\to+\infty$. Thus,  \eqref{direct_computation} and \eqref{total_alpha} 
yield \eqref{limit_derivative}.
\end{proof}
\begin{proof}[Proof of Proposition \ref{prop:flux_multi}]
\mbox{}\\ \\
\textit{(i)} By \eqref{def_flux_super} and   Lemma \ref{lemma:complete},  
$G$ is continuously differentiable on $[N_\alpha,N_{\alpha-1}]$. 
Besides, for $\rho$ in this interval,
only indices $i\in\mathcal C_\alpha$ may contribute to 
\eqref{def_flux_super},  because 
$\widetilde{\rho}_i(\rho)\in\{0,1\}$ for other indices.\\ \\
\textit{(ii)} 
By \eqref{def_flux_super}--\eqref{flux_tasep},
\[
G'(\rho)=\sum_{i=0}^{n-1}[1-2\widetilde{\rho}_i(\rho)]\widetilde{\rho}'_i(\rho)
\]
Since  by Lemma \ref{lemma:complete},  $\widetilde{\rho}_i$  increases from $0$ to $1$ 
on $[N_{\alpha},N_{\alpha-1}]$, \eqref{one_sided} follows 
 from \eqref{limit_derivative}. \\ \\ 
\textit{(iii)}
For $\rho=\rho^\alpha(c)$  defined in \eqref{total_alpha}, we have
\be\label{flux_c}
G(\rho)=\sum_{i\in\mathcal C_\alpha}
\gamma_i\frac{c\lambda^\alpha_i}{(1+c\lambda^\alpha_i)^2}=:\widetilde{G}^\alpha(c)
\ee
and $G'$ vanishes at $\rho=\rho^\alpha(c)\in [N_{\alpha},N_{\alpha-1}]$ 
if and only if $(\widetilde{G}^\alpha)'$ vanishes at $c$.
The latter derivative is a rational fraction, hence it has finitely many zeroes.
\end{proof}
\section{Proof of Theorem \ref{thm:phase_trans}}
\label{sec_phases}
For the following computations we assume without loss of generality that 
 $\gamma_0+\gamma_1=1$  (which amounts to a time change), and that 
 $p+q=1$ (which allows all the possible
values for the flux $G$, since $G$ depends on  $q/p$), with $0<p<1/2$.
 We rely on the expression for $G$ given in 
 \eqref{eq:Gviaphi_0}--\eqref{eq:psi}, 
where we set $\gamma_0=d$ and $\gamma_1=1-d$.  
We assume that  $d\geq 1/2$. 
Note that $\psi(\rho)\le 1$. 
The following values and equalities are independent of $d$:
\begin{eqnarray}\label{eq:G0-2}
G(0)&=&G(2)=0\\\label{eq:G1}
G(1)&=& \frac{1}{4}-\varphi(1)^2 \\\label{eq:phi2-ro}
\psi(2-\rho)&=& \psi(\rho), \quad\hbox{hence}\quad \varphi(2-\rho)= \varphi(\rho).
\end{eqnarray}
We then compute
\begin{eqnarray*}\label{eq:psi-deriv}
\psi'(\rho)&=& \left(\frac{r-1}{r+1}\right)^2 2(\rho-1)\\\label{eq:phi-deriv}
\varphi'(\rho)&=&-\frac{1}{2}\left(\frac{r-1}{r+1}\right)
\frac{(\rho-1)}{\sqrt{\psi(\rho)}}\\\label{eq:phi-deriv2}
\varphi''(\rho)&=& -2r
\left(\frac{(r-1)}{(r+1)^3}\right)\psi(\rho)^{-3/2}\\\label{eq:phi-deriv3}
\varphi^{(3)}(\rho)&=& 
6r(\rho-1)\left(\frac{(r-1)^3}{(r+1)^5}\right)\psi(\rho)^{-5/2}\\\label{eq:G-deriv}
G'(\rho)&=& \frac{1}{2}(1-\rho) 
+(2d-1)\left[- \varphi(\rho)+(1-\rho)\varphi'(\rho)\right]
-2\varphi(\rho)\varphi'(\rho)\\\label{eq:G-deriv2}
G''(\rho)&=& -\frac{1}{2} +(2d-1) 
\left[- 2\varphi'(\rho)+(1-\rho)\varphi''(\rho)\right]
-2\varphi'(\rho)^2-2\varphi(\rho)\varphi''(\rho)\\\nonumber
G^{(3)}(\rho)&=& (2d-1) \left[- 3\varphi''(\rho)+(1-\rho)\varphi^{(3)}(\rho)\right]
-6\varphi'(\rho)\varphi''(\rho)-2\varphi(\rho)\varphi^{(3)}(\rho)\\\label{eq:G-deriv3}
&=& 6r \frac{(r-1)^2}{(r+1)^4}\psi(\rho)^{-5/2}
\left[(2d-1)\frac{4r}{(r-1)(r+1)}+(1-\rho)\right].
\end{eqnarray*}
We have that $G^{(3)}(\rho)$ changes sign for the value
\begin{equation}\label{eq:tilderho0}
\widetilde\rho_0= \widetilde\rho_0(r,d)=1+ (2d-1)\frac{4r}{(r-1)(r+1)} \geq 1.
\end{equation}
Therefore $G''$ is increasing before $\widetilde\rho_0$ 
and decreasing after. We compute
\begin{eqnarray}\label{eq:G2-0}
G''(0)&=& -\frac{r^2+1}{(r+1)^2}-(2d-1)\frac{(r-1)}{(r+1)}
\left[1+\frac{2r}{(r+1)^2}\right]<0\\
\label{eq:G2-1}
G''(1)&=& -1+ \frac{r+1}{4\sqrt r}
\quad\hbox{(this value is independent of }d)\\\label{eq:psi-tilderho0}
\psi(\widetilde\rho_0)&=&\frac{4r}{(r+1)^2}
\left[(2d-1)^2\frac{4r}{(r+1)^2}+1\right]\\\label{eq:G2-tilderho0}
G''(\widetilde\rho_0)&=& -1+ \frac{1}{\sqrt{\psi(\widetilde\rho_0)}}
\left[-\frac{1}{2}+\frac{(r+1)^2}{4r}\psi(\widetilde\rho_0)\right]\\\label{eq:G2-2}
G''(2)&=& -\frac{r^2+1}{(r+1)^2}
+(2d-1)\frac{(r-1)}{(r+1)}\left[1+\frac{2r}{(r+1)^2}\right].
\end{eqnarray}
Moreover we have that
\begin{eqnarray} \label{ineq:G2-2}
G''(2)<0 &\Leftrightarrow& g(r):= r^3(1-d) + r^2(2-3d)
 + r(3d-1) +d > 0\\  \label{ineq:2-with-rho_0_0}
\widetilde\rho_0<2 &\Leftrightarrow&  (2d-1)4r<r^2-1 \\ \label{ineq:2-with-rho_0_1}
                    &\Leftrightarrow& r>\widetilde r_1 
                    = \widetilde r_1(d): = 2(2d-1)+\sqrt{4(2d-1)^2+1}.  
\end{eqnarray}
Note that $d\mapsto \widetilde r_1(d)$ is an increasing function.
Since  $G''(0)<0$,  to determine the number of inflection points of $G$ 
depending on $r$ for given $d$, it is enough to determine the sign of 
$G''(2)$ (i.e. of $-g(r)$), and that of $G''(\widetilde{\rho}_0(r))$.\\ \\
{\em Sign of $g(r)$ and $G''(2)$}.\\ \\
 {\em First case:} $d=1$. Then, 
\be\label{case_d=1}
g(r)=-r^2+2r+1,\quad g'(r)=-2r+2<0
\ee
thus $g$ is decreasing, with $g(r)=0$ for 
 $\widetilde r_0=1+\sqrt 2<\widetilde{r}_1(1)$
(the other root,  $1-\sqrt 2$, is less than 1). Thus for $r<\widetilde r_0$,  
$G''(2)<0$ and for $r>\widetilde r_0$, 
$G''(2)>0$.  \\ \\
 We now deal with $1/2\leq d<1$. We have
\begin{eqnarray}\label{more-for-g'_0}
g'(r)&=& 3r^2(1-d) + 2r(2-3d) + (3d-1)\\
\label{more-for-g_0}
g(1)=2;\quad && g'(1)=6(1-d)>0. 
\end{eqnarray}
 {\em Second case:}  $d=1/2$. 
We have that $\widetilde\rho_0 =1$ (see \eqref{eq:tilderho0}),
thus  $r> \widetilde r_1(1/2)=1$ by \eqref{ineq:2-with-rho_0_1}.
By \eqref{ineq:G2-2}, we have that  $2g(r)= r^3 + r^2 + r +1 > 0$, hence $G''(2)<0$.
Moreover $G''(\widetilde\rho_0)=G''(1)$ given in \eqref{eq:G2-1}. 
\\ \\
 {\em Third case:}  $1/2< d<1$. 
Going back to \eqref{more-for-g'_0}, the 
sign of $g'(r)$ depends on that of
\be\label{delta_gprime}
\delta=\delta(d):=18d^2-24d+7
\ee
which has a unique root in $[1/2;+\infty)$, namely 
\be
\label{def:d1}
d_1:=\frac{4+\sqrt{2}}{6} >\frac{5}{6}.
\ee 
For $d<d_1$, $\delta(d)\leq 0$, thus $g'(r)$ has constant sign;  
by \eqref{more-for-g_0}, $g'(r)<0$, $g$ is decreasing negative 
and $G''(2)>0$. For $d\geq d_1$, 
there are two values $r_1(d),r_2(d)$ 
such that $g'(r_1)=g'(r_2)=0$:
\begin{eqnarray}\label{eq:r1}
r_1=\frac{3d-2-\sqrt{\delta}}{3(1-d)}
&=&r_0-\frac{\sqrt{\delta}}{3(1-d)}\\\label{eq:r2} r_2
=\frac{3d-2+\sqrt{\delta}}{3(1-d)}
&=&r_0+\frac{\sqrt{\delta}}{3(1-d)}, \quad\mbox{where}\\
\label{more-for-G2-2_0}
r_0 & = & \frac{3d-2}{3(1-d)}.
\end{eqnarray}
Then $g$ is increasing between 1 and $r_1$, 
then decreasing between $r_1$ and $r_2$,
and finally increasing after $r_2$. 
Therefore if $g(r_2)>0$, then $G''(2)< 0$. We have 
\begin{eqnarray}
\label{for-gr2-1}
g(r_2)>0 &\Leftrightarrow & \phi(d)<0, \mbox{ where} \\
\label{for-gr2-0}
\phi(d) & = & \delta\sqrt\delta-3(1-d)+(3d-2)\delta.
\end{eqnarray}
Since $d>\frac{5}{6}$, the function $\phi$ is increasing. We have $\phi(d_1)<0$,
and $\phi(\frac{14}{15})>0$, thus $\phi$ changes sign for some value $\widetilde d_1$ 
between $d_1$ and $\frac{14}{15}$ such that $\phi(\widetilde d_1)=0$, given by:
\begin{equation}\label{eq:tilde-d1}
\widetilde d_1:=\frac{1}{2}+\frac{1}{6}\sqrt{3+2\sqrt{3}}.
\end{equation}
Gathering the above lines, we have that
\begin{eqnarray}\label{ineq:G''2-first}
\hbox{if}\quad d\leq  \widetilde d_1,
\, &&\hbox{then}\quad G''(2)< 0 .
\end{eqnarray}
 For $d\geq \widetilde d_1$, there are two values 
 $r_3(d)<r_2(d)<r_4(d)$ such that 
$g(r_3)=g(r_4)=0$,  $g$
is positive outside $(r_3(d),r_4(d))$ hence $G''(2)<0$, 
and negative in $(r_3(d),r_4(d))$ hence $G''(2)>0$. 
\mbox{}\\ \\
 {\em Fourth case: }$d>1$.
In this case we have $r_2<0<r_1$. Since $1-d<0$, we have $g'(r)>0$ for $0\leq r<r_1$, $g'(r)<0$ for $r<r_2$, and $\lim_{r\to+\infty}g(r)<0$. Since $g(1)=2>0$, there is a unique $r_3(d)>1$ such that $g(r)>0$ for $1\leq r<r_3(d)$ and $g(r)<0$ for $r>r_3(d)$. \\ \\
\noindent 
{\em Sign of $G''(\widetilde\rho_0(r))$}. 
Recall \eqref{eq:G2-tilderho0}. Denote
\begin{equation}
A=A(r)=\frac{4r}{(r+1)^2}
\end{equation}
 {\em First case:} $d=1$. Then,
 \begin{equation}\label{rho_d=1}
\widetilde\rho_0(r)=1+\frac{4r}{r^2-1},\qquad  G''(\widetilde\rho_0(r))=\frac{1}{4}
 \end{equation}
  {\em Second case:} $d>1$. Then, $G''(\widetilde\rho_0(r))>0$. \\ \\
   {\em Third case:} $d<1$.  
Let
\begin{eqnarray}\label{def:bar-d1}
\overline d_1 &=& \frac{2+\sqrt{2\sqrt 3}}{4}\quad\hbox{and}\\ \label{def:A1} 
A_1 = A_1(d)&=&\frac{1}{2(2d-1)^2}\left(-1+\frac{1}{2\sqrt{d(1-d)}}\right)\\ \label{ineq:sign-A1}
A_1(\overline d_1)=1 \quad&\hbox{and}&\quad 
A_1(d) <1 \Leftrightarrow d< \overline d_1
\end{eqnarray}   
 We have that 
   \begin{equation}\label{cond-G2-tilderho0}
  G''(\widetilde\rho_0(r))>0\Leftrightarrow 
  A^2 (2d-1) ^2+A+\frac{1}{16d(d-1)}<0\Leftrightarrow A<A_1
    \end{equation}
   Then 
\begin{eqnarray} \label{ineq:G2-rho_0-1}
\hbox{if}\quad d> \overline d_1,\, &&\hbox{then}\quad G''(\widetilde\rho_0)> 0\\
 \label{def:bar-r1}
\hbox{if}\quad d\leq \overline d_1,&&\hbox{let}\quad
\overline r_1 = \overline r_1(d) = \frac{(1+\sqrt{1-A_1})^2}{A_1}\\
  \label{ineq:G2-rho_0-0}
\hbox{if}\quad d= \overline d_1,\, \hbox{then}\quad 
\overline r_1(d) = 1 &&\quad\hbox{and}\quad 
 G''(\widetilde\rho_0)= \frac{1}{2}\\
\label{ineq:G2-rho_0-2}
\hbox{if}\quad \frac{1}{2}<d< \overline d_1\,
\quad\hbox{and}\quad r> \overline r_1, \,
&& \hbox{then}\quad G''(\widetilde\rho_0)> 0\\
\label{ineq:G2-rho_0-3}
\hbox{if}\quad \frac{1}{2}<d< \overline d_1
\quad\hbox{and}\quad r< \overline r_1, \,
&& \hbox{then}\quad G''(\widetilde\rho_0)< 0\\
\label{ineq:G2-rho_0-4}
\hbox{if}\quad \frac{1}{2}<d< \overline d_1
\quad\hbox{and}\quad r= \overline r_1, \,
&& \hbox{then}\quad G''(\widetilde\rho_0)= 0.
\end{eqnarray}  
Moreover we have that  (see proof of \eqref{one_2_2}--\eqref{one_2_3} below) 
\begin{eqnarray}\label{eq:tilde-d0}
\overline r_1 >\widetilde r_1&\Leftrightarrow& 
d<\widetilde d_0:= \frac{1}{2}+\frac{\sqrt 3}{4}.
\end{eqnarray}
Hence 
\be\label{values:d1-dbar1}
\frac{5}{6}<
\widetilde{d}_1
<
\widetilde{d}_0
<\frac{14}{15}
<\overline d_1 <\frac{39}{40}.
\ee
{\em Proof of 1, (i)}.
 Monotonicity of 
$\overline{r}_1(d)$ follows from its expression 
\eqref{def:A1}--\eqref{def:bar-r1}. Monotonicity properties of 
$r_3(d)$ and $r_4(d)$ hold because for fixed $r>1$, $g$ is a decreasing 
function of $d$.  By implicit function Theorem, $r_3(d)$ and $r_4(d)$ 
are continuously differentiable on $(\widetilde{d}_1;1)$, 
and due to joint $(r,d)$ continuity of $g$, they extend continuously 
to $\widetilde{d}_1$. When $d\to 1^-$, \eqref{lim_r4} holds because 
for $r>1$, $g(r)$ tends to $-\infty$ as $d\to 1-$, whereas 
$r_3(d)\to\widetilde{r}_0=1+\sqrt{2}$, the unique zero of $g(r)$
found above when $d=1$. The unique zero $r_3(d)$ of $g(r)$ for 
$d>1$ tends to $\widetilde{r}_0$ as $d\to 1+$, and on $(1;+\infty)$, 
$r_3(d)$ is continuously differentiable by the implicit function theorem. \\ \\
{\em Proof of 1, (ii)}. For the purpose of 2., we prove 
the following larger set of properties:
\begin{eqnarray}
d=\widetilde{d}_0 & \Rightarrow & 
\widetilde{r}_1(d)=
\overline{r}_1(d)=r_3(d) 
\label{one_2_1}\\ 
d<\widetilde{d}_0 & \Rightarrow & \widetilde{r}_1(d)
<\overline{r}_1(d)<r_3(d)\label{one_2_2}\\
d>\widetilde{d}_0 & \Rightarrow & r_3(d)<\widetilde{r}_1(d)<r_4(d)\label{one_2_3}\\
\widetilde{d}_0<d\leq\overline{d}_1 & \Rightarrow & \overline{r}_1(d)<r_3(d).
\label{one_2_4}
\end{eqnarray} \\
Proof of \eqref{one_2_1}: For $d=\widetilde{d}_0$, 
and $r=\widetilde{r}_1(d)=\overline{r}_1(d)$, 
by definition of these quantities, we have 
$\widetilde{\rho}_0=2$ and $G''(\widetilde{\rho}_0)=0$, 
hence $G''(2)=0$, implying $r\in\{r_3(d);r_4(d)\}$; 
that $r=r_3(d)$, i.e. \eqref{one_2_2}, follows from \eqref{one_2_3}, 
proven right below. \\ \\
Proof of \eqref{one_2_2}
 and \eqref{one_2_4}: 
Monotonicities proved in 1, \textit{(i)}, and 
\eqref{one_2_1}, imply $\widetilde{r}_1(d)<\overline{r}_1(d)$.  
For the second inequality, assume $\overline{r}_1(d)>r_3(d)$; 
then for $r_3(d)<r<\inf(\overline{r}_1(d),r_4(d))$, we have
 $G''(2)>0$ and $G''(\widetilde{\rho}_0)\leq 0$,  hence a 
contradiction.  Assume now $\overline{r}_1(d)=r_3(d)$. Then for 
$r=\overline{r}_1(d)=r_3(d)$, we have $G''(2)=G''(\widetilde{\rho}_0)=0$, 
while $\widetilde{r}_1(d)<\overline{r}_1(d)$ implies $\widetilde{\rho_0}<2$ 
for $r=\overline{r}_1(d)$. This cannot occur since $G''$ is strictly decreasing 
on $[\widetilde{\rho}_0,2]$. 
 \\ \\
Proof of \eqref{one_2_3}:  Observe that by its expression 
\eqref{ineq:2-with-rho_0_1}, $\widetilde{r}_1$ is increasing in $d$. 
By  \eqref{one_2_1} and monotonicities proved in 1, (i), this 
implies $r_3(d)<\widetilde{r}_1(d)$ and $\overline{r}_1(d)<\widetilde{r}_1(d)$.
The inequality $\widetilde{r}_1(d)<r_4(d)$ is only to be proven for $d<1$, 
where $r_4$ is defined. For  
$r=\widetilde{r}_1(d)$, by definition, $\widetilde{\rho}_0=2$.
Since $\overline{r}_1(d)<\widetilde{r}_1(d)$, 
\eqref{ineq:G2-rho_0-0}--\eqref{ineq:G2-rho_0-3} imply 
$G''(\widetilde{\rho}_0)< 0$, hence $G''(2)< 0$. By the above study of $g$, 
we must have $r< r_4(d)$. \\ \\
 Proof of \eqref{one_2_4}:
Monotonicities proved in 1, \textit{(i)}, and 
\eqref{one_2_1}, imply $\widetilde{r}_1(d)>\overline{r}_1(d)$, thus 
$\widetilde{\rho}_0>2$ for  $r=\overline{r}_1(d)$. The proof of the 
inequality $\overline{r}_1(d)<r_3(d)$ is as in \eqref{one_2_4}, 
except that now, assuming $r=\overline{r}_1(d)=r_3(d)$, $G''$ is strictly 
increasing on $[2,\widetilde{\rho}_0]$, and $G''(2)=G''(\widetilde{\rho}_0)=0$, 
whence a contradiction.
\\ \\
{\em Proof of 2}. The following holds regardless of the value of $d$. 
First, if $r\leq \overline{r}_1(d)$,  since $G''(\widetilde{\rho}_0)\leq 0$ 
and $\widetilde{\rho}_0$ is a global strict maximizer of $G''$ on 
$[0;+\infty)$, $G''\leq 0$ and vanishes at most once, hence $G$ 
is strictly concave. 
Next, if $r_3(d)<r<r_4(d)$, $G''(\widetilde{\rho}_0)>G''(2)>0$. 
Thus $G$ has a single inflection point (note that this does not 
depend on the position of $\widetilde{\rho}_0$ with respect to $2$).
Finally, if $r>r_4(d)$, $G''(2)<0<G''(\widetilde{\rho})$. 
Since $r_4(d)>\widetilde{r}_1(d)$, $\widetilde{\rho}_0<2$, 
hence $G$ has two inflection points. \\ \\
For other positions of $r$, the conclusion  depends on $d$:\\ 
Proof of \textit{(i)}: 
If $r>\overline{r}_1(d)$,
by \eqref{one_2_2},  $\overline{r}_1(d)>\widetilde{r}_1(d)$, 
thus $G''(\widetilde{\rho}_0)>0$ and 
$\widetilde{\rho}_0<2$. Since $d<\widetilde{d}_1$, we also have 
$G''(2)<0$. Recalling $G''(0)<0$ and variations of $G''$, 
we conclude that $G''$ vanishes once on either side of $\widetilde{\rho}_0$.\\ 
Proof of \textit{(ii)}: 
If $\overline{r}_1(d)<r<r_3(d)$, then $G''(\widetilde{\rho}_0)>0$, 
$G''(2)<0$, 
 and by \eqref{one_2_2},  $r>\widetilde{r}_1(d)$, 
thus $\widetilde{\rho}<2$, implying two points of inflection. 
\\ 
Proof of \textit{(iii)--(v)}: If $r<r_3(d)$, then $G''(2)<0$. 
Since $r_3(d)<\widetilde{r}_1(d)$, we also have $\widetilde{\rho}_0>2$, 
thus $G''<0$ on $[0;2]$ and $G$
is strictly concave. \\ 
The sign of $G$ in \textit{(v)} follows from item \ref{sign_prop} of 
Proposition \ref{prop:mlt1},
the value $\frac{d}{d-1}$ arising from \eqref{cond_vanish} 
with $\gamma_0=d$, $\gamma_1=1-d$ and $q>p$. Note that $r_3(d)<\frac{d}{d-1}$, 
because if $G$ is strictly concave with $G(0)=G(2)=0$, it cannot vanish 
on $(0;2)$. \\ \\
As regards local extrema of $G$ in \textit{(v)}, denoting the location of 
the zero by $\rho_0$, since $G(0)=G(\rho_0)=G(2)=0$, there is at least 
one extremum on either side of $\rho_0$, and there cannot be more because 
there is a single inflection point; the extremum on the left of $\rho_0$ 
is a maximum because $G''(0)<0$. \\ \\
{\em Complement.} We can locate inflection points with respect to $1$. 
Recall  that  $\widetilde{\rho}_0>1$ for $d>1/2$ and that 
$\widetilde{\rho}_0=1$ for $d=1/2$,  cf. \eqref{eq:tilderho0}. From the 
above discussion,  
whenever there are two inflection points,  they lie on either side of  
$\widetilde{\rho}_0$, hence the larger one lies to the right of $1$; 
whereas a single inflection point lies to the left of 
$\min(\widetilde{\rho}_0;2)$. 
Further, by \eqref{eq:G2-1}, 
\be
\label{ineq:G2-1}
G''(1)>0 \Leftrightarrow  r> (2+\sqrt 3)^2
=:{\overline{\overline r}}_1.
\ee
Note that
\be\label{comp_rtilde1-rbarbar1}\widetilde r_1<{\overline{\overline r}}_1.\ee 
Let 
\[
\widetilde{d}_0<\overline{\overline{d}}_1
:=\frac{1}{2}\left(1+\frac{14\sqrt{3}}{27}\right)<\overline{d}_1.
\]
Using \eqref{eq:G2-2}, we have that for $r={\overline{\overline r}}_1$, 
$G''(2)>0$ if and only if $d>\overline{\overline{d}}_1$, and $G''(2)<0$ 
if and only if $d<\overline{\overline{d}}_1$. Since $r_3(d)$ is decreasing 
and $r_4(d)$ increasing,
this implies 
\begin{eqnarray}
\label{claim_barbar_1}
d < \overline{\overline{d}}_1 & \Rightarrow 
& {\overline{\overline r}}_1>r_4(d)\\ 
d > \overline{\overline{d}}_1 & \Rightarrow  
& r_3(d)<{\overline{\overline r}}_1<r_4(d) .
\end{eqnarray}
We can deduce that: \\ 
1) If $d<\overline{\overline{d}}_1$: 
\textit{a)} whenever $G$ has two inflection points, if $r>{\overline{\overline r}}_1$,
the smaller inflection point lies to the left of $1$; if 
$r>{\overline{\overline r}}_1$, it lies to the right of $1$.
\textit{b)} Whenever $G$ has a single inflection point, it lies 
between $1$ and $\min(\widetilde{\rho}_0;2)$. \\ 
2) If $d>\overline{\overline{d}}_1$: 
\textit{a)} whenever $G$ has two inflection points, 
the smaller one lies to the left of $1$.
\textit{b)} Whenever $G$ has a single inflection point, if 
$r>{\overline{\overline r}}_1$, if lies to the left of $1$; 
if $r<{\overline{\overline r}}_1$, it lies between $1$ and 
$\min(\widetilde{\rho}_0;2)$.
\section{Proof of hydrodynamics: \thmref{local_equilibrium}}\label{sec:proof_hydro}
 We follow the approach of 
\cite{Bahadoran2002}. The principle of this approach is 
to reduce the hydrodynamic limit for the Cauchy problem 
to that for the Riemann problem 
\eqref{first_riemann},
\eqref{eq:hydrodynamic equation}--\eqref{eq:init-riemann} 
thanks to an approximation scheme. \\ \\
 The Harris construction  (see section \ref{subsec:setup}) 
allows to couple the evolutions from
different initial configurations through \textit{basic coupling}, that is, by
using the same Poisson processes for them.  In other words, particles 
of these coupled evolutions 
make the same jumps, from the same departure site to the same arrival site, 
when it is possible, 
that is, when the departure site is full and the arrival site is empty 
for both configurations; otherwise
only one particle jumps and the other one does not move.  \\ \\
To derive Cauchy hydrodynamics from Riemann hydrodynamics, we need two properties.
The first one is the  {\em finite propagation} property 
(\cite[Lemma 3.1, Lemma 3.2]{Bahadoran2002}), whose proof
carries over here with minor modifications, and the similar property  
for the hydrodynamic equation.  In our context, the statement reads as follows. 
\begin{proposition}\label{prop:prop} 
There exist  constants $\speed,C>0$ such that the following holds:\\ \\
(i)
Assume  $(\eta^N_t)_{t\geq 0}$ and $(\xi^N_t)_{t\geq 0}$ are two 
coupled multilane SEP's with 
 generator $L^{\theta(N)/N}$, cf. \eqref{eq:generator of the Exclusion}, 
such that
$\eta^N_0(z,i)=\xi^N_0(z,i)$  for all $z\in[a,b]$ and 
 $i\in W$ (cf. \eqref{def_W}),  where $a,b\in\Z$ and $a<b$. 
Then outside probability  $e^{-CNt}$,  $\eta^N_{Nt}(z,i)=\xi^N_{Nt}(z,i)$ 
for all  $t<(b-a)/(2\speed)$,  $z\in[a+\speed Nt,b-\speed Nt]$  and  $i\in W$. \\ \\
(ii) 
Let $u_0$, $v_0$ be initial data for \eqref{eq:hydrodynamic equation}.
Then the associated entropy solutions  $u(t,.)$ and $v(t,.)$  at time 
$t<(b-a)/(2\speed)$ satisfy
\be\label{contraction_entropy}
\int_{a+\speed t}^{b-\speed t}[ u(t,x)-v(t,x)]^\pm dx
\leq \int_{a}^{b}[u_0(x)-v_0(x)]^\pm dx .
\ee
In particular, if $u_0(.)$ and $v_0(.)$ coincide a.e. on $[a,b]$, 
then  $u(t,.)$ and $v(t,.)$  coincide a.e. on $[a+\speed t,b-\speed t]$.
\end{proposition}
The second property,  proved in Subsection \ref{proof:cauchy}, 
is \emph{macroscopic stability}  
(\cite[Definition 3.1]{Bahadoran2002}). It states that if 
two coupled configurations
are initially macroscopically close, they remain so at later times. 
 This will be a consequence of the following property, which is a 
substantial multilane refinement of  
\cite[Lemma 3.1]{BMM}. The refinement involves condition 
\eqref{assumption_relax} to control the possibly slower transverse dynamics. 
\begin{proposition}\label{prop:macrostab} 
Let $\eta_0^N,\xi_0^N\in\mathcal{X}$ be such that
\be\label{eq:macroscopic stability-4}
 \eta ^N_0(z,i)=\xi^N_0(z,i)=0,\quad\forall (z,i)\in (\Z\setminus[-aN;aN])\times W 
 \ee
	for some constant $a>0$. 
	Define, for $t\geq 0$, the function  
	$\phi_t^{ N}:\mathbb{Z}\rightarrow \Z\cap
	[-n,n]$ 
	by
	\begin{align}
	\phi_t^{ N}(z)=\sum_{i\in W}\left[\eta^N_t(z,i)-\xi^N_t(z,i)\right]
	\end{align}
  where $(\eta^N_t)_{t\geq 0}$ and $(\xi^N_t)_{t\geq 0}$ denote 
  the processes  with generator $L^{\theta(N)/N}$ 
   starting  respectively from $\eta^N_0$ and $\xi^N_0$. 
	 Then, under condition \eqref{assumption_relax},  
 for every  $\gamma>0$, 
	\begin{equation}
	\lim_{N\to+\infty}\mathbb{P}
	\left( \sup_{z\in \Z,\,t\geq 0}
	N^{-1}\left|\sum_{u\in\Z:\,u\geq z}\phi^N_{ t}(u)\right|
	>\sup_{z\in \Z}N^{-1}\left|\sum_{u\in \Z:\,u\geq z}\phi^N_0(u)\right|
	+\gamma\right)=0.\label{eq:macroscopic stability-1}
	\end{equation}
	\end{proposition}
Following \cite[Theorem 3.2]{Bahadoran2002}, 
in order to prove statement \textit{(i)} of \thmref{local_equilibrium}, 
it is enough to prove it for Riemann initial profiles 
\eqref{first_riemann}, \eqref{eq:init-riemann}, and to verify that 
our model satisfies Propositions \ref{prop:prop}--\ref{prop:macrostab}. 
\begin{remark}\label{rk:kl}
\textit{Since local equilibrium implies hydrodynamics
(see \cite[Proposition 0.4]{kipnis2013scaling}), 
statement \textit{(i)} of \thmref{local_equilibrium} for Riemann profiles 
 is implied by statement \textit{(iii)}  of \thmref{local_equilibrium}. 
Thus  the latter will be the actual purpose of Subsection \ref{proof:riemann}.}
\end{remark}
For the Riemann problem, the hydrodynamic limit was first addressed in
\cite{Andjel1987, AV2} in the case  of a strictly concave flux function, 
for which the entropy solution is very explicit. For more general flux functions, 
the Riemann solution is not so explicit, but the hydrodynamic limit was addressed 
in \cite{Bahadoran2002} thanks to a variational representation. \\
In the present case, the microscopic derivation of this variational formula 
cannot be carried out as in \cite{Bahadoran2002} because the vertical part 
of the generator has a different scaling than the horizontal part. 
The microscopic derivation  here is carried out in 
Subsection \ref{proof:riemann} 
and  relies on a two-block estimate which we derive in 
Subsection \ref{subsec:prop_relax} from an {\em approximate interface property}. 
This property is a suitable generalization of the known {\em exact} 
interface property (\cite{Liggett1976}) for the single-lane asymmetric 
exclusion process (which no longer holds exactly for multilane exclusion). 
The two-block estimate will also allow us in Subsection \ref{subsec:proof_loc_eq} 
to derive the weak local equilibrium property for the Cauchy problem. 
The Riemann strong local equilibrium in statement \textit{(iii)} 
of \thmref{local_equilibrium} shall be derived in Subsection 
\ref{proof:riemann} together with Riemann hydrodynamics, see Remark 
\ref{rk:kl}.\\ \\  
To prove Propositions \ref{prop:prop} and \ref{prop:macrostab}  
and conclude the proof of \thmref{local_equilibrium}, we need 
coupling tools recalled hereafter. \\ \\
 \noindent{\bf Coupling, ordering and discrepancies.}  
There is a natural partial order on $\mathcal{X}$, 
namely, for $\eta,\xi\in\mathcal{X}$,
\be\label{eq:orderconf}
\eta\leq\xi\quad\mbox{ if and only if }
\quad\forall x\in V,\, \eta\left(x\right)\leq\xi\left(x\right).
\ee 
 Basic  coupling shows that
the simple exclusion process is \textit{attractive}, that is, the partial order
\eqref{eq:orderconf} is conserved by the dynamics. In other words, 
\be\label{eq:attra}
\forall \eta_0,\xi_0\in\mathcal{X}, \, \eta_0\le\xi_0 \Rightarrow \forall t\ge 0,\,
\eta_t\le\xi_t \,\mbox{ a.s.}
\ee 
The order \eqref{eq:orderconf} endows an order on the set  $\mathcal{M}_{1}$  
of probability measures
on $\mathcal{X}$ in the following way. A function $f$ on $\mathcal{X}$
is said to be increasing if and only if $\eta\leq\xi$ implies 
$f\left(\eta\right)\leq f\left(\xi\right)$.
For two probability measures $\mu_{0},\mu_{1}$ on $\mathcal{X}$,
we write $\mu_{0}\leq\mu_{1}$ if and only if for every increasing
function $f$ on $\mathcal{X}$ we have 
$\int f(\eta)\mu_{0}\left(d\eta\right)\leq\int f(\eta)\mu_{1}\left(d\eta\right)$.
 We shall write $\mu_1<\mu_2$ if $\mu_1\leq\mu_2$ and $\mu_1\neq\mu_2$. \\ \\
 Thus \eqref{eq:attra} implies, for two probability measures $\mu,\nu$ 
 on $\mathcal X$,
\be\label{attractive}
\mu\leq\nu\Rightarrow \mu S_t\leq\nu S_t
\ee
where $(S_t)_{t\geq 0}$ denotes the semigroup of the process $(\eta_t)_{t\geq 0}$. 
In such a coupling, we say that at $x$ there is an $\eta$ \textit{discrepancy} 
(or blue particle) if $\eta(x)>\xi(x)$, a $\xi$ discrepancy (or red particle) 
if $\eta(x)<\xi(x)$, a coupled particle (or black particle) if $\eta(x)=\xi(x)=1$, 
a hole (or white particle)  if $\eta(x)=\xi(x)=0$. Blue and red are 
called opposite (type) particles 
(or opposite (type) discrepancies). The evolution of the coupled process can be 
formulated as follows. At a time $t\in\mathcal N_{(x,y)}$, a blue, red or 
black particle at $x$ exchanges with a hole at $y$; a black particle at $x$ 
exchanges with a blue or red particle at $y$; if there is a pair of 
opposite (type) particles at $x$ and $y$, they are replaced by a hole at $x$ 
and a black particle at $y$. We call this a \textit{coalescence}. This shows that 
no new discrepancy can ever be created.   
 \subsection{The Riemann problem:  
 Proof of \thmref{local_equilibrium}, part \textit{(iii)}}\label{proof:riemann}
 For the Riemann problem \eqref{eq:hydrodynamic equation}--\eqref{eq:init-riemann}, 
the following result can be found in
\cite{Bahadoran2002}.
\begin{proposition}
\label{lemma:riemann} The
entropy weak solution of  \eqref{eq:hydrodynamic equation}--\eqref{eq:init-riemann} 
is the self-similar  monotone  function given  by  
\be\label{entropy_envelope}
u\left(t,x\pm 0\right)=u\left(1,\frac{x}{t}\pm 0\right)
=\left(G^*_{\alpha,\beta}\right)'\left(\frac{x}{t}\pm 0\right)
\ee
where
\be\label{dual_G}
G^*_{\alpha,\beta}(v):=\left\{
\ba{lll}
\inf_{\rho\in[\beta;\alpha]}[v\rho-G(\rho)] & \mbox{if} & \alpha\geq\beta\\
\sup_{ \rho\in[\alpha;\beta]}[v\rho-G(\rho)] & \mbox{if} & \alpha\leq\beta.
\ea
\right.
\ee
 If $\alpha\leq\beta$ (resp. $\alpha\geq\beta$), $u(1,v-)$ 
 is the smallest (resp. largest) and $u(1,v+)$ the largest 
 (resp. smallest)
optimizer in \eqref{dual_G}. In particular,  $u(1,.)$ 
is continuous at $v$ if and only if the optimizer is unique, 
and is then equal to this optimizer.  
\end{proposition}
The proof  of \thmref{local_equilibrium}, part \textit{(iii)} 
partly follows the scheme of \cite{Andjel1987, Bahadoran2002}, 
but significant differences are involved due to the fact that the horizontal 
and vertical parts of the generator have different scalings. In particular, 
we need  a block estimate (Proposition \ref{prop:twoblock} below) 
for which we first introduce relevant notation.  The one block part, see
  \eqref {eq:relax_oneblock} below, will be used here, while the two-block 
  part \eqref{eq:relax} will be necessary for Cauchy hydrodynamics. 
Let $f$ be a local function on $\mathcal X$. 
 For $z\in\Z$, $\xi\in\mathcal X$, $l\in\N$, 
 we define the local block average of $f$ by
\be\label{def_block_average}
M_{z,l}f(\eta):=\frac{1}{2l+1}\sum_{u\in \Z:\,|u-z|\leq l}\tau_u f(\eta).
\ee
By a slight abuse of notation, we write $M_{x,l}\eta^i$, 
resp. $M_{x,l}\overline{\eta}$, to denote $M_{x,l}f(\eta)$ 
for $f$ defined by $f(\eta)=\eta^i(0)$, resp. $f(\eta)=\overline{\eta}(0)$,  where $\overline{\eta}$ was defined in \eqref {config_total}.
\begin{proposition}\label{prop:twoblock}
Let 
$u(.)\in C^0_K(\R)$ such that for all $c\in[0,n]$,  $u(.)-c$ 
changes sign finitely many times.
Assume 
$\eta^N_0$ is a local Gibbs state  with global profile $u(.)$, 
defined by \eqref{gibbs_lane} with
$u^N_x:=u(x/N)$.
Then, for   every local function $f:\mathcal X\to\R$,  
$A<B$  and $s>0$,
\be\label{eq:relax_oneblock}
\lim_{l\to+\infty}\limsup_{N\to+\infty}
\Exp\left\{
N^{-1}\sum_{x\in\Z\cap[ NA,NB]}\Delta_{x,l}(\eta^N_{Ns})
\right\}
=0
\ee
\be\label{eq:relax}
\lim_{\varepsilon\to 0}\limsup_{N\to+\infty}
\Exp\left\{
N^{-1}\sum_{x\in\Z\cap[ NA,NB]}\Delta_{x,N\varepsilon}(\eta^N_{Ns})
\right\}
=0
\ee
where
\be\label{def_relax}
\Delta_{x,l}(\eta):=\left|
M_{x,l}f(\eta)-\overline{f}(M_{x,l}\overline{\eta})
\right|
\ee
 with $\overline{f}$ defined by \eqref{eq_average}  and, as in \thmref{local_equilibrium}, $(\eta^N_s)_{s\geq 0}$ 
 denotes a process with initial state  $\eta^N_0$ 
 and generator  $L^{\theta(N)/N}$.   
\end{proposition}
Proposition \ref{prop:twoblock} is proved in Subsection \ref{subsec:prop_relax}. 
\begin{proof}[Proof of \thmref{local_equilibrium}, part \textit{(iii)}]
 We assume $\beta\leq\alpha$, the case 
	$\alpha\leq\beta$ being similar.  In the sequel, 
	we denote by $(S_t^N)_{t\geq 0}$ the semigroup generated by 
	   $L^{\theta(N)/N}=N^{-1}\mathcal L^N$,  and $(S_{Nt}^N)_{t\geq 0}$ the semigroup generated by 
	   $\mathcal L^N$ defined in \eqref{gen_relax}.  
	    Statement (iii) of \thmref{local_equilibrium} follows from Proposition \ref{lemma:riemann} above and Lemmas \ref{lemma:step_one}--\ref{lemma:step_three} below, that will be proved hereafter. 
	   \end{proof}
	\begin{lemma}\label{lemma:step_one}
	Let $t>0$.   Then
	%
	for  every sequence  $N_p\to+\infty$ as $p\to+\infty$, there exists 
	a subsequence $N_{p_{k}}$ and a countable  
	subset $D$ of $\R$ 
	 with the following property: for $N=N_{p_k}\to+\infty$,  simultaneously for all  
	 $v\in\R\setminus D$,  the weak limit
	\begin{equation}\label{eq: local equilibrium-9}
\lim_{N\to+\infty}\frac{1}{t}\int_0^t\mu_{\alpha,\beta}\tau_{\left[vNt\right]}
S_{Ns}^Nds=\mu_{v}, 
	\end{equation} 
	hold for a   family of measures $\mu_v$,  
	nonincreasing  with respect to $v$, of the form
	\be\label{form_mu_v}
	\mu_v= \int_{[0;n]}\nu_\rho\lambda_v(d\rho)
	= \int_{[\beta,\alpha]}\nu_\rho\lambda_v(d\rho) 
	\ee
	where $\lambda_v$ is a probability measure 
	  on  $[\beta,\alpha]$,  
	 and $\nu_\rho$ is the measure
	 defined by \eqref{def_inv_rho} and \eqref{def_inv_gen}.
	 \end{lemma}
	 \begin{remark}\label{rk:test_meas}
\textit{(i) The left-hand side of \eqref{eq: local equilibrium-9} 
is the probability measure on $\mathcal X$
whose action on test functions $\varphi:\mathcal X\to\R$ is given by
\be\label{rk_action}
\left\langle
\frac{1}{t}\int_0^t\mu_{\alpha,\beta}\tau_{\left[vNt\right]}
S_{Ns}^Nds,\varphi(.)
\right\rangle:= \frac{1}{t}\int_0^t
\left\langle
\mu_{\alpha,\beta}
S_{Ns}^N,\varphi(\tau_{\left[vNt\right]}.)
\right\rangle
ds
\ee
(ii) Although we shall not use this fact as such, it is worth 
noting that $\mu_v$ defined by \eqref{form_mu_v} is 
translation invariant because $\nu_\rho$ is.}
\end{remark}
	 \begin{lemma}\label{lemma:step_three}
	The measure  $\lambda_v(d\rho)$    in \eqref{form_mu_v} satisfies 
	\[
	 \int_{[\beta,\alpha]}[v\rho-G(\rho)]\lambda_v(d\rho)
	=\inf_{\beta\leq\theta\leq\alpha}\left\{ v\theta-G\left(\theta\right)\right\}
	\]
	\end{lemma}
	 The proof of Lemma \ref{lemma:step_three} relies on the following lemma. 
	\begin{lemma}
	\label{lemma:step_2}
	for $N=N_{p_k}\to+\infty$,  and $(u,v)\in (\R\setminus D)^2$, 
	\begin{equation}\label{eq: local equilibrium-2}
\lim_{N\rightarrow\infty}\mu_{\alpha,\beta}
S_{Nt}^N\left(
\frac{1}{Nt}
%
\sum_{x\in V:\,\left\lfloor uNt\right\rfloor\leq x\left(0\right)
\leq\left\lfloor vNt\right\rfloor}
\overline{\eta}\left(x(0)\right)
\right)
=F\left(v\right)-F\left(u\right),
	\end{equation}
where
\begin{align}
\label{Feq}
	F\left(w\right) & :=\int [w\rho-G(\rho)]\lambda_{w}\left(d\rho\right)
	\end{align}	
	\end{lemma}
	For the proof of Lemma \ref{lemma:step_one}, we enumerate some 
	relevant statements on probability-valued monotone functions 
	(in the sense of stochastic order), analogous to classical properties 
	of real-valued monotone functions. We state them for nonincreasing 
	functions, but the same statements hold with nondecreasing instead.
	\begin{lemma}\label{lemma:tech}
	Let $\mathcal P(\mathcal X)$ denote the set of probability measures 
	on $\mathcal X$. Then:
	\\ \\
	(i) A nonincreasing function  $m(.):\R\to\mathcal P(\mathcal X)$ 
	has left and right weak limits $m(v\pm)$ at all $v\in \R$ in the sense 
	of weak convergence, such that $m(v-)\leq m(v)\leq m(v+)$, 
	and is weakly continuous, i.e. $m(v-)=m(v+)$, outside (at most)  
	countably many $v\in\R$.\\ \\
	(ii) A nonincreasing function $m(.):D\to\mathcal P(\mathcal X)$, 
	where $D$ is a dense subset of $\R$,
	extends into a nonincreasing function $\widetilde{m}(.)$ on $\R$, 
	and $\widetilde{m}(v\pm)$ is uniquely determined for all $v\in\R$. \\ \\
	(iii) Under previous notation, if a sequence of nonincreasing 
	functions $m_n(.):D\to\mathcal P(\mathcal X)$ converges to a 
	(nonincreasing) function 
	$m(.)$, $\widetilde{m}_n(.\pm)$ converges to $\widetilde{m}(.\pm)$ on $\R$.
	\end{lemma}
	\begin{proof}[Proof of Lemma \ref{lemma:tech}]
	Testing against nonincreasing local functions on $\mathcal X$,
	this follows from similar properties of real-valued sequences or functions. 
	One uses the fact that any local function on $\mathcal X$ is a linear 
	combination of nondecreasing local functions.
	\end{proof}
	\begin{proof}[Proof of Lemma \ref{lemma:step_one}]
	  Define 
\begin{eqnarray}
	\label{def_time_average}
	\mu_{v}(N) & := & 
	\frac{1}{t}\int_0^t\mu_{\alpha,\beta}\tau_{\left[Nvt\right]}S^N_{Ns}ds,  
	\end{eqnarray}
	Since $\beta\leq\alpha$,  $\tau\mu_{\alpha,\beta}\leq\mu_{\alpha,\beta}$.  
	Thus by attractiveness  and invariance of $\nu_\alpha$ and $\nu_\beta$, 
	for $v\leq w$, 
	\begin{eqnarray}
	\label{sandwich_measures}
\nu_\alpha\geq \mu_v(N)\geq\mu_w(N)\geq\nu_\beta
	\end{eqnarray}
Since the measures $\mu_v(N)$ are supported on the compact set $\mathcal X$, 
by diagonal extraction, we can find a 
a subsequence $(N_{p_k})_k$ of $(N_p)_p$ such that, for each  $v\in\mathbb{Q}$,
the  weak limits  
\begin{eqnarray}
\label{weak_lim}
\mu_v & := & \lim_{N\to+\infty}\mu_v(N) 
\end{eqnarray}
exist for $N=N_{p_k}\to+\infty$.
 By \eqref{sandwich_measures} and Lemma \ref{lemma:tech}, $\mu_v$ 
 is a nonincreasing function of $v$ and extends into a nonincreasing 
 function (denoted the same) on $\R$. In addition, we have
\be\label{continuous_mu_genuine}
\mu_v=\mu_{v+}=\mu_{v-}
\ee
for $v\in\R\setminus D$, where $D$ denotes the countable set of 
discontinuity points of the extended $\mu_.$, and the limit 
\eqref{weak_lim} still holds for $v\in\R\setminus D$, 
which we henceforth assume.  \\ \\
It remains to show that 
the measure $\mu_v$ is
of the form \eqref{form_mu_v}. 
 To this end, we define
\begin{eqnarray}
\nonumber
\mu_{v,\varepsilon}^+(N) & := & 
\frac{1}{\varepsilon}\int_v^{v+\varepsilon}\mu_w(N) dw\\
	& = & \frac{1}{t}\int_0^t
	\frac{1}{\lfloor N\varepsilon\rfloor+1}
\sum_{z=\lfloor Nvs\rfloor}^{\lfloor Nvs\rfloor+\lfloor N\varepsilon\rfloor}\mu_{\alpha,\beta}\tau_{z}S^N_{Ns}+O\left(\frac{1}{N}\right)
	\label{def_space_average_right}
	\end{eqnarray}
where $O(1/N)$ holds in the sense of total variation norm, 
and the second equality follows from \eqref{def_time_average}. 	
	The integral in \eqref{def_space_average_right} 
	is meant against test functions, i.e.,
	\be\label{test_integral}
	\left\langle
	\int_v^{v+\varepsilon}\mu_w(N) dw,f
	\right\rangle:=\int_v^{v+\varepsilon}\langle
	\mu_w(N),f
	\rangle dw
	\ee
	for bounded continuous test functions $f$ on $\mathcal X$.
	Using \eqref{weak_lim}, we let $N\to+\infty$ in  
	\eqref{def_space_average_right} and apply dominated convergence 
	to the right-hand side of \eqref{test_integral}. 
	This yields the weak  limit
\be
 \lim_{N\to+\infty}\mu^+_{v,\varepsilon}(N)=\frac{1}{\varepsilon}
 \int_v^{v+\varepsilon}\mu_w dw=:\mu_{v,\varepsilon}^+   \label{weak_lim_plus}
\ee
Letting $\varepsilon\to 0^+$ in \eqref{weak_lim_plus} yields 
(since $v\in\R\setminus D)$,
\be\label{continuous_mu}
\lim_{\varepsilon\to 0^+}\mu^+_{v,\varepsilon}=\mu_{v}
\ee
Given \eqref{weak_lim_plus}--\eqref{continuous_mu}, 
 we will prove that  \eqref{form_mu_v} holds already 
 for $\mu^+_{v,\varepsilon}$, cf. \eqref{form_mu_v_eps} below, 
 and let $\varepsilon\to 0^+$. 
Let $f$ be a local function on $\mathcal X$.  By
\eqref{def_time_average} and \eqref{def_space_average_right},
\be\label{rewrite_average}
\int_{\mathcal X}f(\eta)d\mu_{v,\varepsilon}^+(N)(\eta)=
\frac{1}{t}\int_0^t
\Exp\left\{
 \frac{1}{\lfloor N\varepsilon\rfloor +1}
\sum_{z=\lfloor Nvs\rfloor}^{\lfloor Nvs\rfloor+\lfloor N\varepsilon\rfloor}
f\left(\tau_{z}\eta^N_{Ns}\right)
ds
\right\}+O\left(\frac{1}{N}\right)
\ee
 For each $s\in[0;t]$, we can apply \eqref{eq:relax_oneblock} 
 in Proposition \ref{prop:twoblock} to the above expectation
 (for $A=\lfloor Nvs\rfloor,B=\lfloor Nvs\rfloor+\lfloor Nv\varepsilon\rfloor$), 
 and then use dominated convergence (since $f$ is bounded).
Hence, we can rewrite the right-hand side of \eqref{rewrite_average} as
\be\label{rewrite_average_2}
\frac{1}{t}\int_0^t\Exp\left\{
\frac{1}{\lfloor N\varepsilon\rfloor +1}
\sum_{z=\lfloor Nvs\rfloor}^{\lfloor Nvs\rfloor+\lfloor N\varepsilon\rfloor}\overline{f}\left(
M_{z,l}\,\overline{\eta}^N_{Ns}
\right)ds
\right\}+\delta_{N,l,\varepsilon}
%
\ee
where 
\be\label{vanish_delta}
\lim_{l\to+\infty}\limsup_{N\to+\infty}\delta_{N,l,\varepsilon}
=0
\ee
The expectation in \eqref{rewrite_average_2} is of the form
\be\label{rewrite_average_3}
\int_{[0,n]}\overline{f}(\rho) d\lambda_{v,\varepsilon}(N,l)(\rho)
=\int_{\mathcal X}f(\eta) d\mu_{v,\varepsilon}(N,l)(\eta)
\ee
for some probability measure  $\lambda_{v,\varepsilon}(N,l)$  on $[0,n]$
(not depending on $f$),
 where 
\be\label{rewrite_average_3_bis}
\mu_{v,\varepsilon}(N,l):=\int_{[0,n]}\nu_\rho  d\lambda_{v,\varepsilon}(N,l)(\rho)
\ee
 Gathering \eqref{rewrite_average}--\eqref{rewrite_average_3_bis}, we have
 \be\label{rewrite_average_4}
\int_{\mathcal X}f(\eta)
d\mu^+_{v,\varepsilon}(N)(\eta) 
=\int_{\mathcal X}f(\eta)
\int_{[0,n]}\nu_\rho  d\lambda_{v,\varepsilon}(N,l)(\rho)+\delta_{N,l,\varepsilon}
\ee
Letting  $N=N_{p_k}\to+\infty$
followed by $l\to +\infty$ in \eqref{rewrite_average_4}, 
we obtain a family of probability measures 
$\lambda^+_{v,\varepsilon}$ such that
\be\label{form_mu_v_eps}
\mu^+_{v,\varepsilon}=\int_{[0;n]}\nu_\rho\lambda^+_{v,\varepsilon}(d\rho)=
\int_{[\beta;\alpha]}\nu_\rho\lambda^+_{v,\varepsilon}(d\rho)
\ee
where the second equality follows from the extreme bounds 
in \eqref{sandwich_measures}. By letting $\varepsilon\to 0$ 
in \eqref{form_mu_v_eps} and using  \eqref{continuous_mu}, 
we obtain a family of measures $(\lambda_v)_{v\in \R\setminus D}$ 
satisfying \eqref{form_mu_v}.
\end{proof}
\begin{proof}[Proof of Lemma \ref{lemma:step_2}]
	We prove that
	 for $N=N_{\ell_k}\to+\infty$,  and $(u,v)\in D^2$, 
	\begin{equation}\label{eq: local equilibrium-2}
\lim_{N\rightarrow\infty}\mu_{\alpha,\beta}S_{Nt}^N
\frac{1}{Nt} f^{u,v}_N (Nt)
=F\left(v\right)-F\left(u\right),
	\end{equation}
where
\begin{align}
\label{def_fuv}
 f^{u,v}_N(t) &  := 
\sum_{x\in V:\,\left\lfloor ut\right\rfloor \leq x\left(0\right)
\leq \left\lfloor vt\right\rfloor}
\overline{\eta}\left(x(0)\right)
\\
\label{Feq}
	F\left(w\right) & :=\int [w\rho-G(\rho)]\lambda_{w}\left(d\rho\right)
	\end{align}	
	Define the function
	\begin{align}\label{def_g1}
	g^{u,v}_N(s,z) 
	& :=\sum_{ \lfloor  us\rfloor+1\leq x\left(0\right)
	\leq\lfloor vs\rfloor}\mu_{\alpha,\beta} S^N_{s}
	\left(\overline{\eta}\left(x+z\right)\right)
	\\
	&+\left(\lfloor us\rfloor+1-us\right)
	 \mu_{\alpha,\beta} S^N_{s}\left(\overline{\eta}\left(\lfloor
	 us\rfloor+z\right)\right)\notag
	 \\
	    & +  \left(vs-\lfloor vs\rfloor\right)
	\mu_{\alpha,\beta} S^N_{s}\left(\overline{\eta}
	\left(\lfloor vs\rfloor+1+z\right)\right)
	\label{def_g2}
	\end{align}
	for $(s,z)\in[0;+\infty)\times\Z$, and set
	\be
	\label{function_g_eps}
	g_N^{u,v,\varepsilon}(s)  
	:=  \frac{1}{1+\lfloor\varepsilon N\rfloor}
	\sum_{z=0}^{\lfloor\varepsilon N\rfloor}g^{u,v}_N(s,z)
	\ee
	We then have,  for every $t\geq 0$, 
	\be\label{compare_fg}
	\left\vert
	f^{u,v}_N(t)-g_N^{u,v,\varepsilon} (t) 
	\right\vert\leq 2n\left(
	1+N\varepsilon
	\right)
	\ee
The function	$g^{u,v}_N$ is absolutely continuous  
with respect to its first variable,  and 
	for  $N\in\N^*$,  $u,v\in D$ and $s$ such that $sNv,sNu\notin\mathbb{Z}$, 
	we have
	\begin{align}
	 \frac{1}{N}\frac{d}{ds}g^{u,v}_N(Ns,z) &=
	 \sum_{\lfloor uNs\rfloor+1\leq x(0)\leq\lfloor vNs\rfloor}
	 \mu_{\alpha,\beta}S^N_{  Ns}\left(\mathcal L^N
	 \overline{\eta}\left(x+z\right)\right)\notag\\
	 & +v\mu_{\alpha,\beta}S^N_{  Ns}\left(\overline{\eta}
	\left(\lfloor vNs\rfloor+1+z\right)\right)
	-uN\mu_{\alpha,\beta}S^N_{ Ns}\left(\overline{\eta}
	\left(\lfloor uNs\rfloor+z\right)\right)\notag\\
	 & +\left(\lfloor uNs\rfloor+1-uNs\right)
	 \mu_{\alpha,\beta}S^N_{ N s}\left(\mathcal L^N
	 \overline{\eta}\left(\lfloor uNs\rfloor+z\right)\right)\notag\\
	 &+\left(vNs-\lfloor vNs\rfloor\right)
	 \mu_{\alpha,\beta}S^N_{  Ns}\left((\mathcal L^N
	 \overline{\eta}\left(\lfloor vNs\rfloor+1+z\right)\right)\notag\\
	 & =\mu_{\alpha,\beta}S^N_{Ns}
	 \left(\tau_{\lfloor uNs\rfloor+z}j(\eta)\right) 
	 -\mu_{\alpha,\beta}S^N_{Ns}\left(\tau_{\lfloor vs\rfloor+z}j(\eta)\right)\\
	 & +vN\mu_{\alpha,\beta}S^N_{ Ns}\left(\overline{\eta}
	 \left(\lfloor vNs\rfloor+1+z\right)\right)
	 -uN\mu_{\alpha,\beta}S^N_{ N s}\left(\overline{\eta}
	 \left(\lfloor uNs\rfloor+z\right)\right)\notag\\
	 &+\left(\lfloor uNs\rfloor+1-uNs\right)
	 \mu_{\alpha,\beta}S^N_{Ns}\left(\tau_{\lfloor uNs\rfloor -1+z}j(\eta)-\tau_{\lfloor uNs\rfloor+z}j(\eta)\right)\notag\\
	 & +\left(vNs-\lfloor vNs\rfloor\right)
	 \mu_{\alpha,\beta}S^N_{Ns}\left(\tau_{\lfloor vNs\rfloor +z}j(\eta)-\tau_{\lfloor uNv\rfloor+1+z}j(\eta)\right)\label{derivative_g}\\
	\end{align}
 where we have used  \eqref{micro_gradient} in the second equality.  
 It follows that
\be
\frac{1}{N}\frac{d}{ds}g_N^{u,v,\varepsilon}(Ns) =
h_N^{u,v,\varepsilon,1}(Ns)+h_N^{u,v,\varepsilon,2}(Ns)
\ee
where 
\begin{align}
h_N^{u,v,\varepsilon,1}(s) 
&=\frac{1}{1+\lfloor N\varepsilon\rfloor}
\sum_{z=0}^{\lfloor N\varepsilon\rfloor}
\tau_{\lfloor us\rfloor+z}\mu_{\alpha,\beta}S^N_{s}
[j(\eta)-u\eta(0)]
\notag\\
&- \frac{1}{1+\lfloor N\varepsilon\rfloor}
\sum_{z=0}^{\lfloor N\varepsilon\rfloor}
\tau_{\lfloor vs\rfloor+z}\mu_{\alpha,\beta}S^N_{s}
[j(\eta)-v\eta(1)]
\label{g1}\\
h_N^{u,v,\varepsilon,2}(s) &=
\left(\lfloor us\rfloor+1-us\right)
	 \mu_{\alpha,\beta}S^N_{s}\left(\
	 \tau_{\lfloor us\rfloor -1}j(\eta)
	 -\tau_{\lfloor us\rfloor+\lfloor\varepsilon s\rfloor}j(\eta)
	 \right)\notag
	 \\
	 &+\left(vs-\lfloor vs\rfloor\right)
	 \mu_{\alpha,\beta}S^N_{s}\left(\
	 \tau_{\lfloor vs\rfloor }j(\eta)
	 -\tau_{\lfloor uv\rfloor+1+\lfloor\varepsilon s\rfloor}j(\eta)
	 \right)
\label{g2}
\end{align}
We now integrate \eqref{derivative_g} over $s\in[0;t]$. Recalling \eqref{def_space_average_right}, and using 
\eqref{def_microflux_ladder}--\eqref{jh},
\be\label{uniform_bound_g2}
\vert h_N^{u,v,\varepsilon,2}(s)\vert\leq M:=\sum_{i\in W}\left\{
d_i+l_i+\sum_{j\in W}q(i,j)
\right\}
\ee
\begin{align*}
\frac{ g^{u,v,\varepsilon}_{N}\left(Nt\right)}{Nt}
&=
%
%
\left\langle
\mu^+_{u,\varepsilon}(N),j(\eta)-u\eta(0)
\right\rangle
-\left\langle
\mu^+_{v,\varepsilon}(N),j(\eta)-v\eta(1)
\right\rangle+\frac{M}{\varepsilon}O\left(\frac{1}{N}\right)
\end{align*}
where the $O(.)$ is uniform. 
	Letting $N=N_{l_k}\to+\infty$, followed by $\varepsilon\to 0$ in $E_0$, 
	and  using  \eqref{weak_lim}, \eqref{weak_lim_plus}--\eqref{continuous_mu} 
	and \eqref{form_mu_v},
	we obtain that for  $(u,v)\in D^2$,
	\be\label{almost_there}
	\lim_{\varepsilon\to 0}\lim_{N\to+\infty}
	\frac{1}{Nt}g_N^{u,v,\varepsilon}(t)=F(v)-F(u)
	\ee
	Eventually, \eqref{eq: local equilibrium-2} follows 
	from \eqref{almost_there} and \eqref{compare_fg}.  
	\end{proof}
	\begin{proof}[Proof of Lemma \ref{lemma:step_three}]
	The proof is exactly as in \cite[Theorem 2.1]{Bahadoran2002}, 
	using  attractiveness, \eqref{eq: local equilibrium-2}--\eqref{Feq}, 
	and processes starting 
	from $\mu_{\beta,\theta}$  and 
	$\mu_\theta$ for $\theta\in[\beta,\alpha]$. We thus omit further details.
\end{proof}
\begin{remark}\label{remark_av}
\textit{ Lemmas \ref {lemma:step_one} and \ref {lemma:step_2} 
are respective counterparts of Lemmas 3.1 and 3.2 in \cite{Andjel1987}.  
The former, due to the double scaling, requires here a completely 
different proof, based on the block estimates in Proposition \ref{prop:twoblock}.
When $\theta(N)=N$, that is $\mathcal L^N=N(L_h+L_v)$, 
the dynamics follows a time rescaling of a {\em fixed} generator. 
In this case, 
 the proof of Lemma \ref {lemma:step_one} can be simplified by using  
the original argument of \cite[Lemma 3.1.]{Andjel1987}. That is, the limit 
\eqref{eq: local equilibrium-9}--\eqref{form_mu_v}
can be obtained by observing that  Cesaro limits of 
the process distribution are contained in the subset 
$\mathcal I\cap\mathcal S$ of invariant measures 
of this fixed generator, and using Theorem 
\ref{thm:characterization_lemma_gen}. Note, 
cf. Remark \ref {remark_inv_meas}, that Proposition \ref{prop:twoblock} 
enables us to avoid Theorem 
\ref{thm:characterization_lemma_gen}. Finally,  the proof of  
Lemma \ref{lemma:step_2} essentially follows  that of 
\cite[Lemma 3.3.]{Andjel1987}, with some more details provided here.}
	\end{remark}
\subsection{ Local equilibrium:  
proof of \thmref{local_equilibrium}, part \textit{(ii)}}
\label{subsec:proof_loc_eq}
In the sequel, we use the following realization of a  
local Gibbs state $(\xi^N)$ with distribution \eqref{local_gibbs_1}.
Take an i.i.d. sequence $(U_x)_{x\in V}$ of $\mathcal U(0,1)$ 
random variables, and define, for $x=(z,i)\in V$,
\be\label{couple_initial}
\xi^N(z,i)={\bf 1}_{\{U_x\leq  u^{N,i}_z\}}.
\ee
The construction \eqref{couple_initial} has the property that 
if we construct two Gibbs states $(\xi^N)$ and $(\zeta^N)$ 
with respective profiles $u(.)$ and $v(.)$ with the same 
$\mathcal U(0,1)$ family, then
\be\label{couple_gibbs}
\Exp\left[
\left(\xi^N(z,i)-\zeta^N(z,i)\right)^\pm
\right]=\left(
  u^{N,i}_z-v^{N,i}_z
\right)^\pm.
\ee
In particular, $\xi^N(z,i)$ and $\zeta^N(z,i)$ are ordered a.s. 
like $u^N_x$ and $v^N_x$.
\begin{proof}[Proof of \thmref{local_equilibrium}, part \textit{(ii)}]
\mbox{}\\ \\
{\em Step one.} 
First assume that $u_0(.)$ has compact support and satisfies 
 the assumption of  Proposition \ref{prop:twoblock}. 
The triangle inequality  and Proposition \ref{prop:twoblock} imply, 
for $\varphi\in C^0_K(\R)$, 
\be\label{hydro3}
\lim_{\varepsilon\to 0}\limsup_{N\to+\infty}\Exp\left\{
\left|
N^{-1}\sum_{x\in\Z}\varphi\left(\frac{x}{N}\right)\tau_x f(\eta^N_{Nt})-
N^{-1}\sum_{x\in\Z}\varphi\left(\frac{x}{N}\right)\overline{f}
\left( M_{x,N\varepsilon}\overline{\eta}^N_{Nt}\right)
\right|dx
\right\}
=0.
\ee
 Indeed, in the sum in \eqref{hydro3} we may replace $\tau_x f(\eta^N_{Nt})$ 
 by the spatial average of $\tau_y f(\eta^N_{Nt})$ over $y\in\Z$ 
 such that $|y-x|\leq N\varepsilon$. The replacement error vanishes as 
 $N\to+\infty$ followed  by $\varepsilon\to 0$, because by an exchange 
 of summation, this is equivalent to replacing $\varphi(x/N)$ with its 
 spatial average. We can then apply Proposition \ref{prop:twoblock} 
 to replace the spatial average of $\tau_y f(\eta^N_{Nt})$ by 
 $\overline{f}(M_{x,l}\overline{\eta}_{Nt})$. \\ \\
  Next, by part \textit{(i)} of \thmref{local_equilibrium}, 
  for every $A>0$, we have
\be\label{hydro2}
\lim_{\varepsilon\to 0}\limsup_{N\to+\infty}\Exp\left\{
\int_{-A}^A\left|
M_{\lfloor Nx\rfloor,N\varepsilon}\overline{\eta}_{Nt}- u(t,x)
\right|dx
\right\}
=0.
\ee
Since $\overline{f}$ is continuous, $\overline{f}$ allows us 
to replace the second sum in \eqref{hydro3} by the last 
integral in \eqref{wle}.  \\ \\
{\em Step two.} Next assume $u_0(.)$ is a Borel function supported 
in $[a,b]\subset\R$. It can be approximated in $L^1(\R)$ by
a sequence  $(u_{k,0}(.))_{k\in\N}$  of continuous 
functions supported in $[a,b]$ and satisfying the assumption of 
 Proposition \ref{prop:twoblock}. 
For each  $k\in\N$, let  $(\eta^{N}_{k,0})_{N\in\N^*}$ 
denote a local Gibbs state   with profile $u_{k,0}$. 
Since the coupling cannot create discrepancies, 
using \eqref{couple_gibbs}, we have 
\be\label{lonet}
\limsup_{N\to+\infty}\Exp\left\{
N^{-1}\sum_{(x,i)\in V}\left|
\eta^{N}_{k,Nt}(x,i)-\eta^N_{Nt}(x,i)
\right|
\right\}
\leq \sum_{i\in W} \int\left|
u_{k,0}^i(x)-u_0^i(x)
\right|dx.
\ee
Denoting  (for any $\eta,\xi\in\mathcal  X$) 
\begin{eqnarray*}
A^N_{\varphi,f}(\eta) & := & N^{-1}\sum_{x\in\Z}\varphi\left(\frac{x}{N}\right)
\tau_x f(\eta) \\
B^N(\eta;\xi) &:= & N^{-1}\sum_{(x,i)\in V}\left|
\eta(x,i)-\xi(x,i)
\right|. 
\end{eqnarray*}
we have that 
\be\label{control_wle}
\left|
A^N_{\varphi,f}(\eta^{N}_{k,Nt})-A^N_{\varphi,f}(\eta^N_{Nt})
\right|\leq\vert|\varphi\vert|_\infty \vert|f\vert|_\infty 
B^N(\eta ^N_{k,Nt};\eta^N_{Nt}).
\ee  
For each  $k\in\N$, $\eta^N_{k,Nt}$   satisfies w.l.e.
with profiles $ u^i_{k}(t,.):=\widetilde{\rho}_i[u_k(t,.)]$, 
 and by contraction  property \eqref{contraction_entropy} for the 
 hydrodynamic equation  \eqref{eq:hydrodynamic equation}, 
  $u^i_k(t,.)\to u^i(t,.)$   in $L^1_{\rm loc}(\R)$. 
Then \eqref{lonet}--\eqref{control_wle} implies weak 
local equilibrium for $(\eta^{N}_{Nt})$. \\ \\
{\em Step three. } Assume $u_0(.)$ is a Borel function on $\R$. For  
 $k\in\N$, define $u_{k,0}$  by truncating $u_0$ to $0$  outside $[-k,k]$. 
Denote by  $u_k(.,.)$  the corresponding entropy solution. The associated 
l.g.s.  $\eta^N_{k,0}$  can be coupled to $\eta^N_0$ so that 
 $\eta^N_{k,0}(x,i)=\eta^N_0(x,i)$ for all $x\in[-kN,kN]$  
  and $i\in W$. Choosing  $k$  large enough, using w.l.e. for 
  $(\eta^N_{k,Nt})$ and finite propagation property
 (Proposition \ref{prop:prop}) yields the conclusion.
\end{proof}
\subsection{Macroscopic stability: 
proof of Proposition \ref{prop:macrostab}}\label{proof:cauchy}
\subsubsection{ Preliminary definitions}
Recall the definition of discrepancies after equation \eqref{attractive}. 
For the proof of Proposition \ref{prop:macrostab}, we introduce 
a labeling of discrepancies. 
 We initially label blue particles ($\eta$  discrepancies)  
  using successive integers.
   We denote by $X_t^k\in V$ the position of the blue particle with label 
    $k$ (in short, we shall simply say: the position of blue label $k$) 
    at time $t$.
	This position is defined until the particle possibly coalesces 
	with a red one,   at which time the blue label ceases to exist,  
	so that labels are no longer successive integers. 
	The labeling 
	is chosen initially so that for any pair of labels $u,v$, 
	$u<v\Rightarrow X_0^u(0)\leq X_0^v(0)$. We define the motion of labeled 
	particles as follows from the Harris construction so that at time $t>0$, 
	we still have,  for any pair of labels $u,v$,
	\be\label{ordered_labels}
	u<v\Rightarrow X_t^u(0)\leq X_t^v(0).
	\ee
	 Note that  though still ordered at all times through 
	 \eqref{ordered_labels}, blue labels cease to be successive 
	 (i.e. some numbers are missing) as soon as a coalescence occurs. 
Let now $k$ be a blue label. 
	If $X^k_{t^-}=x=(z,i)$, $t\in{\mathcal N}_{(x,y)}$, 
	and there is a hole  at $y$ at time $t-$, 
	then  label $k$ first  jumps to $y$, so that $X_t^k=y$.  
	This may in some cases break the ordering relation \eqref{ordered_labels}. 
	If so, an exchange of labels is  performed after the jump to prevent 
	this without modifying the particle configurations. Namely:\\ \\
	\textit{(i)} if $y=(z',i)$ with $z'>z$,  label $k$ is exchanged with the biggest 
	label $l$ such that  $X^{l}_{t-}(0)=z$,  if  $l\neq k$; so that 
	in the end we have $X_t^k=x$ and $X_t^{l}=y$. \\ \\ 
	\textit{(ii)} If $y=(z',i)$ with $z'<z$, label $k$ is exchanged with the smallest 
	label $l$ such that $X^{l}_{t-}(0)=z$,  if $l\neq k$; 
	so that in the end we have $X_t^k=x$ and $X_t^{l}=y$.  \\ \\
 Similarly, 
	if  $X^k_{t^-}=y=(z',j)$,  
	$t\in{\mathcal N}_{(x,y)}$ and there is a black particle at time $t-$ at $x$, 
	the black particle exchanges with label $k$, but as above, the latter 
	exchanges with 
	another label $l$ initially at  $(z',j)$ for some $j\neq i$, 
	 if necessary to maintain \eqref{ordered_labels}. \\ \\
 Finally, if a coalescence occurs following a jump from a site $(z,i)$ to $(z,j)$ 
 with $z\in\Z$ and $i\neq j$, we exchange the label $k$ of the newly coalesced 
 blue particle with a label $l$ chosen randomly (independently of anything else) 
 among labels of all blue particles currently on $\{z\}\times W$, so that after 
 the jump, $l$ will be coalesced everafter, and $k$ still uncoalesced if 
 $l\neq k.$
 It is important to notice that this redistribution of labels does not affect 
 the spatial ordering of blue labels.   
 \begin{remark}\label{rk:coalescence}
\textit{Contrary to above, when a coalescence occurs upon a jump of the type 
$(z,i)\to(z\pm 1,i)$, no exchange is involved in the definition 
of the labels' dynamics: as mentioned above, the two coalescing 
labels are simply destroyed.}
\end{remark}
	We also label red particles according to similar rules and 
	denote by $Y_t^k$ the position 
	at time $t$ of the red particle with label $k$. \\ \\
 As will appear below, the quantities in \eqref{eq:macroscopic stability-1} 
 are closely related 
 to crossings of blue particles by red particles as we now define.
 \begin{definition}\label{def:crossings_0}
({\em Crossings.})    Let $k$ and $l$ be two labels (of the same or 
different colours) and $Z^k_t$, $Z^l_t$ their positions at time $t$ 
(e.g. $Z^k_t=X^k_t$ if $k$ is blue, or $Z^k_t=Y^k_t$ if it is red). 
We shall say that $l$ has crossed  $k$ to the right between times $s$ and $t$ 
(where $s<t$ or $s=t-$) if  $Z^l_{s}(0)<Z^k_s(0)$
and $Z^l_{t}(0)\geq Z^k_t(0)$; and that  $l$ has crossed $k$ to the left 
if $Z^l_{s}(0)\geq Z^k_s(0)$
and $Z^l_{t}(0)< Z^k_t(0)$. 
\end{definition}
Observe  that, due to \eqref{ordered_labels} and Remark \ref {rk:coalescence}, 
a crossing
is impossible unless the two labels either were on different lanes at time $s$,  
or (inclusively) at least one of them has changed lanes at least once between 
times $s$ and $t$. In particular, no crossing ever occurs in a single lane 
nearest-neighbour ASEP.
On the other hand, 
due to the non-strict inequality in the above definition,  
when there are are at least two lanes, 
 it is indeed possible for two  discrepancies of the same colour  to ``cross'' 
in spite of  \eqref{ordered_labels}. \\ \\
 We introduce the following  definitions of relevant crossings: 
 \begin{definition}\label{def:crossings_relevant}
Let $\mathcal V=\mathcal V_0$ denote the set of  initial blue labels 
(by \eqref{eq:macroscopic stability-4}, these labels necessarily in $[-aN,aN]$).  
Let $\mathcal V_t$ denote the set of blue labels $v$  
(not necessarily in $[-aN,aN]$)  that are uncoalesced at time $t$.
 We similarly denote by $\mathcal W_t$ the set of red labels still 
 uncoalesced at time $t$, and $\mathcal W:=\mathcal W_0$ the set of 
 initial red labels. For $v\in\mathcal V$ and $t\geq 0$, we set  
\begin{eqnarray}
L^v_t & := & {\bf 1}_{\mathcal V_t}(v)
\left\{
\sum_{w\in\mathcal W}{\bf 1}_{\{Y^w_0(0)<X^v_0(0)\}}
{\bf 1}_{\{Y^w_t(0)\geq X^v_t(0)\}}
\right\},\label{def_crossings_0}\\
R^v_t & := & {\bf 1}_{\mathcal V_t}(v)
\left\{
\sum_{w\in\mathcal W}{\bf 1}_{\{Y^w_0(0)\geq X^v_0(0)\}}
{\bf 1}_{\{Y^w_t(0)< X^v_t(0)\}}
\right\}.\label{def_crossings}
\end{eqnarray}
\end{definition}  
That is, $L^v_t$ is the number of red particles having crossed $v$ 
to the right up to time $t$, and $R_t^v$ the number of red particles 
having crossed  $v$  to the left,  on the time interval $(0;t]$,  
in the following sense:  \eqref{def_crossings_0}--\eqref{def_crossings} 
does not count multiple crossings of $v$ by the same red particle, 
but only one crossing, counted as leftward or rightward according 
to initial and final relative positions of the red particle with respect to $v$. 
Note that the indicators reset these numbers forever to zero as soon as 
the blue or red particle coalesce with an opposite discrepancy.  Remark that
 $L^v_t$ and $R^v_t$ are zero in the single lane case.  \\ \\
The proof of Proposition \ref{prop:macrostab} follows the ideas of 
\cite[Lemma 3.1]{BMM}. However in our setting, to take into account 
the possibly slower vertical scaling, a significant refinement involving 
condition \eqref{assumption_relax} is necessary, see \eqref{lb} and 
Lemma \ref{lemma_coalescence} below. 
\subsubsection{ Proof of Proposition \ref{prop:macrostab}}
 The core of the argument  is contained in the following two results.  
 Lemma \ref{lemma:crossing_1}
says that it is unlikely for a {\em given} label to be crossed by many 
labels of opposite colour, while Corollary \ref{lemma:crossing_2} says 
it is simultaneously unlikely for {\em all} relevant blue labels.
We point out that {\em these results do not in themselves 
require \eqref{eq:macroscopic stability-4}},
but their application to Proposition \ref{prop:macrostab} does.
However,  
they will be used to prove the quasi-interface property, 
cf. Lemma \ref{prop:interface} below, under a set of assumptions 
that does not in general imply \eqref{eq:macroscopic stability-4}. 
\begin{lemma}\label{lemma:crossing_1}
Under notations of Proposition \ref{prop:macrostab}, the following holds 
{\em without requiring
\eqref{eq:macroscopic stability-4}}: for any $v\in\mathcal V$ and 
$\gamma>0$, 
\begin{align}
 \label{prob_crossing_2}
		\Prob\Bigg(
  \sup_{t\geq 0}C^v_{t}>\gamma N\Bigg)<  \max\left[
  e^{-CN}
  ,\exp\left\{
  -C\frac{\theta(N)^{m^*}}{N^{m^*-1}}
  \right\}
  \right]
\end{align}
for some constant $C=C(\gamma)>0$,  where $C^v_t:=L^v_t+R^v_t$. 
\end{lemma}
 \begin{remark}\label{rk:vanish}
\textit{(i)  The random variable  $\sup_{t\geq 0}C^v_t$
in \eqref{prob_crossing_2} counts the number of times label $v$ has been crossed by red particles in its lifetime (be it finite or infinite).  \\ 
 (ii)  Observe that the right-hand side of 
\eqref{prob_crossing_2}
vanishes  as $N\to+\infty$  if and only if \eqref{assumption_relax} holds.}
\end{remark}
\begin{corollary}\label{lemma:crossing_2}
Under assumptions of Lemma \ref{lemma:crossing_1},
we have
\be\label{tsh_2}
\lim_{N\to+\infty}\Prob\left(
\sup_{v\in\mathcal V,\,t>0} 
C^v_{t} >\gamma N
\right)=0.
\ee
\end{corollary}
\begin{remark}\label{rk:blue_red}
\textit{Lemma \ref{lemma:crossing_1} and Corollary \ref{lemma:crossing_2} 
equally hold for red particles, since exchanging the roles of 
$\eta$ and $\xi$ exchanges the colours of blue and red particles.}
\end{remark}
We first conclude the
proof of Proposition \ref{prop:macrostab}, then prove 
Lemma \ref{lemma:crossing_1} and Corollary \ref{lemma:crossing_2}. 
	\begin{proof}[Proof of Proposition \ref{prop:macrostab}]
{\em \noindent Notational remark.} In the sequel,  $\phi^N_t$
  will sometimes be denoted simply by $\phi_t$. Similarly, 
  dependence on $N$ will sometimes be implicit for other quantities such 
  as $\eta_0$, $\xi_0$, $X_t^k$, $Y_t^l$, 
   $\mathcal V_t$ defined above.\\ \\
 It is enough to show that for  $\gamma>0$, 	
	\begin{align}\label{tsh}
	\lim_{N\to+\infty}\mathbb{P}\Bigg(\sup_{z\in \Z,\,t\geq 0}
	\sum_{u\in \Z:\,u\geq z}\phi^N_{t}(u)
	-\sup_{z\in \Z}\sum_{u\in \Z:\,u\geq z}\phi^N_0(u)>\gamma N\Bigg)=0.
	\end{align}
 Then, \eqref{eq:macroscopic stability-1} follows by exchanging 
 the roles of $(\eta^N_t)_{t\geq 0}$ and $(\xi^N_t)_{t\geq 0}$. \\ \\
  We now derive \eqref{tsh} from \eqref{tsh_2}. 
 First we claim that, for every $t\geq 0$, 
\be\label{first_claim}
\sup_{z\in \Z}\sum_{u\in \Z:\,u\geq z}\phi_t(u)
-\sup_{z\in \Z}\sum_{u\in \Z:\,u\geq z}\phi_0(u)\leq 
\sup_{v\in \mathcal V}
\Delta^v_t
\ee
where, for $v\in\mathcal V$, we set
\begin{eqnarray}
\Delta^v_t  &   :=  &  
\left[
\sum_{u\in \Z:\,u\geq X_t^v(0)}\phi_t(u)- 
\sum_{u\in \Z:\,u\geq X^v_0(0)}\phi_0(u)
\right] {\bf 1}_{{\mathcal V}_t}(v)\label{def_delta_v}
\end{eqnarray}
 Indeed, assume first that $\mathcal V_t\neq\emptyset$. Then  
 there exists a label $v\in\mathcal V_t$ such that
\be\label{max_label}
\sup_{z\in \Z}\sum_{u\in \Z:\,u\geq z}\phi_t(u)
=\sum_{u\in\Z:\,u\geq X_t^v(0)}\phi_t(u)
\ee
It follows that 
\[
\sup_{z\in \Z}\sum_{u\in \Z:\,u\geq z}\phi_t(u)
-\sup_{z\in \Z}\sum_{u\in \Z:\,u\geq z}\phi_0(u)\leq 
\sum_{u\in\Z:\,u\geq X_0^v(0)}\phi_t(u)-\sum_{u\in\Z:\,u\geq X_0^v(0)}\phi_t(u)
\]
whence \eqref{first_claim}. On the other hand, assuming $\mathcal V_t=\emptyset$, 
we have $\eta_t\leq\xi_t$, thus the first supremum in \eqref{first_claim} 
is nonpositive; while the second supremum is zero by 
\eqref{eq:macroscopic stability-4}, and the right-hand side of 
\eqref{first_claim} is zero by \eqref{def_delta_v}, so that 
\eqref{first_claim} still holds.  \\ \\
Now define 
\begin{eqnarray}
\nonumber
B^v_t & : = &  {\bf 1}_{
\{
v\in\mathcal V_t
\}
}\left\{
\sum_{z\geq X_t^v(0)}\sum_{i\in W}\left[
\eta_t^N(z,i)-\xi_t^N(z,i)
\right]^+-
\sum_{z\geq X_0^v(0)}\sum_{i\in W}
\left[
\eta_0^N(z,i)-\xi_0^N(z,i)
\right]^+
\right\}
\label{def_crossings_2}
\end{eqnarray}
In words, $B^v_t$ is the number of blue particles that crossed $v$ 
to the right minus the number that crossed $v$ to the left. 
It follows from definitions \eqref{def_crossings_0}--\eqref{def_crossings} 
and \eqref{def_crossings_2} that 
\be\label{decomp_delta}
\Delta_t^v=R_t^v-L_t^v+B_t^v
\ee
Moreover, 
\begin{eqnarray}
\nonumber
B^v_t   & = &  (v_{ t}^*-V_t+1)-(v_{ 0}^*-V_0+1)\\
& =& (v-V_t)-(v-V_0)+(v^*_t-v^*_0)\nonumber\\
& \leq & n-1
\label{crossing_bound}
\end{eqnarray}
where  $v_{t}^*:=\max\mathcal V_t$  denotes the highest of all blue labels 
 at time $t$,  and $V_t$ the lowest among blue labels lying at time $t$ 
 at the same location as $v$. In the first equality in \eqref{crossing_bound}, 
 each of the quantities in parenthesis is equal to the corresponding sum in 
 the  definition \eqref{def_crossings_2}  of $B^v_t$, measuring the number of 
 blue particles to the right of $v$ (including $v$ itself), hence the equality. 
 The first two terms on the second line of \eqref{crossing_bound} are each 
 positive by definition, and they are bounded from above by $n-1$ 
 due to \eqref{ordered_labels}.
The last term on the second line is nonpositive because blue labels cannot 
be created,
hence the final inequality.   Finally, \eqref{tsh} follows from 
Corollary \ref{lemma:crossing_2}, \eqref{decomp_delta} and \eqref{crossing_bound}.
\end{proof}
\subsubsection{ Proofs of Lemma \ref{lemma:crossing_1} and Corollary \ref{lemma:crossing_2}}
 The proof of Lemma \ref{lemma:crossing_1} will use the following result, that will be proved in Subsection \ref{subsec:proof_coalesce}. Note that this statement involves no time scaling.
\begin{lemma}\label{lemma_coalescence}
Consider the simple exclusion process on $W=\{0,\ldots,n-1\}$ with weakly 
irreducible jump kernel $q(.,.)$. Let $(\zeta^1_t,\zeta^2_t)$ be a coupling 
of two such processes via a common Harris system. Assume $\zeta^1_0$ and 
$\zeta^2_0$ have at least two opposite discrepancies, and denote by $T^*$ 
the first coalescence time of two opposite discrepancies.
Then, under condition \eqref{assumption_relax}, there exist  $p^*=p(n,q(.,.))$,  
and $q^*=q^*(n,q(.,.))>0$ (not depending on the initial configurations), 
such that, for every $\theta>0$ and $m^*$ given by \eqref{number_steps},
\be\label{eq:coalescence}
\Exp\left(
e^{-\theta T^*}
\right)\geq p^*\left(\frac{q^*}{q^*+\theta}\right)^{m^*}.
\ee
\end{lemma}
 Lemma \ref{lemma_coalescence} yields a lower  bound (in stochastic sense) on the coalescence time for an exclusion process on a finite set, in terms of the diameter of the random walk graph. 
\begin{example}\label{example_coalescence}
When  $n=2$, $m^*$ and $q^*$ are explicit, as 
a coalescence occurs after the first jump.  
 Indeed, looking into the proof below shows that 
  the statement holds with $m^*=1$ and $q^*=q(0;1)+q(1;0)$.
More generally under \eqref{onestep}, we have $m^*=1$ and 
 $q^*$ given by \eqref{onestep}.
\end{example}
\begin{proof}[Proof of Lemma \ref{lemma:crossing_1}]
\mbox{}\\ \\
 {\em Step one.}  
For each  blue label $v$, define the stopping times
	\begin{align}
		S_0^v=0 \leq T_1^v \leq S_1^v \leq T_2^v ... \label{stoppings}
	\end{align}
 (hereafter denoted simply by $T_i,S_i$), 
	by 
	\begin{eqnarray}
		T_i&:=&\inf\left\{t\geq S_{i-1}:\text{there is a red particle at } 
		(X^v_t(0),j)\mbox{ for some }j\neq X_t^v(1)\right.\nonumber\\
		& &\left.\mbox{ or }  v 
		\mbox{ has coalesced by time }t\right\}\label{stopping_t}\\
		S_i&:=& \min(S_{i,1},S_{i,2}),
		\label{stopping_s}
		\end{eqnarray}
		where
		\begin{eqnarray}
		S_{i,1} &:= & \inf\left\{t\ > T_{i}:\,t\in\bigcup_{(x,y)\in V_i}\mathcal N_{(x,y)}\right\}\label{stopping_s1}
		\\
  	S_{i,2} &:= & \inf\left\{ t\ \geq T_{i}:\, 
  v\mbox{ has coalesced by time }t\right\}\label{stopping_s2}
   \end{eqnarray}
   with
   \begin{eqnarray}
  V_i & := & \left\{
 \left((z,j),(z',j)\right)\in V^2:\,j\in W,\, 
 z\neq z',\, \{z,z'\}\subset\{X^v_{T_i}(0),X^v_{T_i}(0)\pm 1\}
  \right\} 
  \label{stopping_v}
	\end{eqnarray}
	In words, the first part of the event defining $T_i$ says that 
	at time $T_i$, a red label finds itself at the same spatial position 
	as $X_t^v$ in $\Z$ but on another lane (this only makes sense as long 
	as $v$ has not coalesced, hence the second part of the event).  
  The event in the definition \eqref{stopping_s2} of $S_{i,2}$  
  says that a Poisson clock rings for one of the $2n$ 
 horizontal edges connecting the position of $X^v_.$  to a neighbouring 
 position on one of the lanes. 
	Let $(\mathcal{F}_t)_{t\geq 0}$ be the filtration generated by 
	 $(\eta^N_s,\xi^N_s)_{0\leq s \leq t}$.  
 Note that $T_i=S_{i-1}$ is possible 
 only if and only if the first event in $S_{i-1}$ 
had brought a red particle  or left an existing one at the same location as $v$, or 
$v$ was already coalesced at time $T_i$, whereas $S_i=T_i$ is possible 
if and only if $v$ was already coalesced at time $T_i$. In particular, 
the sequence \eqref{stoppings} becomes stationary as soon as $v$ coalesces. \\ \\ 
An important fact is that
on the time interval $(S_i,S_{i+1}]$, at most  $n$  red/blue 
crossings may occur, i.e.
\be\label{note_crossings}
C_{S_{i+1}}-C_{S_i}=L_{S_{i+1}}-L_{S_i}+R_{S_{i+1}}-R_{S_i}\leq  n
\ee
Indeed on $(S_i,T_{i+1})$ there is no red particle 
at the same spatial location as $v$, hence no crossing may occur. 
One crossing  possibly occurs at time $T_{i+1}$ if the red particle was 
on the left of $v$ at $T_{i+1}^-$. On the time interval $(T_{i+1},S_{i+1})$,  
$v$ cannot move and no red particle can move from or to $z=X_{T_{i+1}^-}$, 
so no crossing may occur. Finally at time $S_{i+1}$ a red particle may move 
from or to $z$, which generates one crossing; or $v$ may jump, which may 
generate up to $(n-1)$ crossings. \\ \\
 We now proceed along the following idea, developed in steps two and three below. 
Let us call each interval 
$[S_{i-1},S_i]$, $i\geq 1$, a {\em cycle}. The number of crossings of $v$ 
in each cycle, cf. \eqref{note_crossings}, is at most $n$. 
Thus, to achieve the event in \eqref {prob_crossing_2}, label $v$ 
must have gone through at least 
$\gamma N/n$ cycles before coalescing (if it ever does). 
Steps two establishes a lower bound for the probability of $v$ coalescing 
within each cycle, and step three derives \eqref {prob_crossing_2} from this bound.
 \\ \\
 {\em Step two.} 
  We now show  that  
	\begin{align}\label{lb}
	\mathbb{P}(
	 v\in\mathcal V_{T_i}\,\text{and}\,
	 \text{label $v$ coalesces in $(T_i,S_{i}]$}|\mathcal{F}_{T_i})
	\geq\frac{p^*}{n-1}\left(\frac{q^{*}}{q^{*}
	+\frac{N}{\theta(N)}}\right)^{m^*}
	\end{align}
	 where $m^*$ is given by \eqref{number_steps} and  $p^*,q^*>0$, 
	  introduced in Lemma \ref{lemma_coalescence}; these quantities  
	  are independent of $i$, $v$ and $N$.  \\ \\
To establish \eqref{lb},  first note that on the time interval $[T_i,S_i)$, 
in the coupled process 
$(\eta_t,\xi_t)_{t\geq 0}$, $v$ does not move and the vertical layer 
$\{X^v_{T_i}\}\times W$ is isolated.  Consider the process 
$(\widetilde{\eta}^i_.,\widetilde{\xi}^i_.)$ defined by
\be\label{layer}
(\widetilde{\eta}^i_t(j), \widetilde{\xi}^i_t(j)):=\left(
{\eta}^i_{T_i+\frac{N}{\theta(N)}t}\left(
 X^v_{T_i}(0),j
\right),{\xi}^i_{T_i+\frac{N}{\theta(N)}t}\left(
 X^v_{T_i}(0),j
\right)
\right),\quad j\in W
\ee
By strong Markov property, conditioned on 
 $\mathcal F_{T_i}$,  
this process  evolves following  the simple exclusion 
 process on  $W$
 with kernel  $q(.,.)$ 
 with initial condition given by the restriction of $(\eta_{T_i},\xi_{T_i})$  
 to the layer $\{ X^v_{T_i}(0)\}\times W$. Recall indeed that in 
 Proposition \ref{prop:macrostab}, $(\eta_t,\xi_t)$ has generator 
 $L^{\theta(N)/N}$, cf. \eqref{eq:generator of the Exclusion}. Hence 
  its transverse jumps  have time scaling  $\theta(N)/N$, and
   the inverse scaling in \eqref{layer} yields a rate $1$ exclusion process. \\ \\ 
 Let us denote by 
  $U^*_i$ the first coalescence time of two discrepancies in 
  $ (\widetilde{\eta}^i_.,\widetilde{\xi}^i_.)$.  
 If  
 $$\frac{N}{\theta(N)}U^*_i<S_{i,1}-T_i,$$
  then   a coalescence occurs at time
 \be\label{coalescence_time}
 T_i+\frac{N}{\theta(N)}U^*_i\leq S_{i,2}
 \ee
 on $\{X^v_{T_i}\}\times W$. The inequality in \eqref{coalescence_time} 
 is due to the fact that this coalescence may or not involve label $v$. 
 If it does, $S_{i,2}=S_{i}$ coincides with the left-hand side 
 of \eqref{coalescence_time}.  
Further,  by strong Markov property,  conditioned on 
$\mathcal F_{T_i}$,  $S_{i,1}-T_i=:W_i^*$ is exponentially distributed with parameter 
\[
\lambda^*:=\sum_{i\in W}(d_i+l_i)
\]
For simplicity and without loss of generality we may assume $\lambda^*=1$.
Thus, 
the probability that  a coalescence occurs before $S_{i,1}$ on 
$\{X^v_{T_i}\}\times W$
is bounded from below by 
\be\label{bound_gamma}
\Prob\left(
U^*_i<\frac{\theta(N)}{N}W^*_i
\right)=\Exp\left(
e^{
-\frac{N}{\theta(N)}U^*_i
}\right)\geq p^*\left(\frac{q^*}{q^*+\frac{N}{\theta(N)}}\right)^{m^*}
\ee 
where the equality follows from conditioning on $U^*_i$, and the inequality
from Lemma \ref{lemma_coalescence}.
The proof of \eqref{lb} is concluded by observing that there is a probability 
at least $1/(n-1)$ that the coalesced label selected after the vertical 
redistribution of labels defined above is eventually $v$. \\ \\
 {\em Step three.} We bound the right-hand side of \eqref{lb} 
 by that of \eqref{prob_crossing_2} and conclude the proof of the lemma. 
 To this end we note that  
\be\label{bound_gamma_2}
\frac{q^*}{q^*+\frac{N}{\theta(N)}}\geq \min\left(
a,b\frac{\theta(N)}{N}
\right)
\ee
for constants $a,b$ independent of $i,v,N$. 
 Recall from explanations at the end of step one, that the event 
 in \eqref{prob_crossing} implies $v$ having gone through at least 
 $\gamma N/n$ cycles.  Hence, using strong Markov property 
 and \eqref{note_crossings}--\eqref{bound_gamma_2}, 
	 we conclude that for every $\gamma>0$ ,  there are constants $A,C$ 
	  independent of $i$, $v$ and $N$, such that
\begin{eqnarray}
 \nonumber
		\mathbb{P}\Bigg(
    \sup_{t\geq 0} C^v_t>\gamma N\Bigg)& < &
  \left[
  1-\frac{p^*}{n-1}\left(\frac{q^*}{q^*+\frac{N}{\theta(N)}}\right)^{m^*}
  \right]^{\frac{\gamma N}{ n}}\\
  & < &  \max \left[
  e^{-CN},\exp\left\{
  -C\frac{\theta(N)^{m^*}}{N^{m^*-1}}
  \right\}
  \right]
\label{prob_crossing}
\end{eqnarray}
for some constant $C>0$,  where we used the inequality $1-x\leq e^{-x}$. 
\end{proof} 
\begin{proof}[Proof of Corollary \ref{lemma:crossing_2}]
First observe that a  union bound from Lemma \ref{lemma:crossing_1} 
over the $O(N)$ relevant labels $v$ (as done in \cite[Lemma 3.1]{BMM}) 
 leads to a slightly worse condition than \eqref{assumption_relax} 
 to get a vanishing bound, namely
\be\label{assumption_relax_worse}
\lim_{N\to+\infty}\frac{\theta(N)}{N^{1-\frac{1}{m^*}}(\log N)^{\frac{1}{m^*}}}
=+\infty.
\ee
Note that when $m^*=1$, \eqref{assumption_relax} reduces to 
$\theta(N)\to+\infty$, whereas \eqref{assumption_relax_worse} 
introduces a superfluous growth condition. \\ \\
 To prove \eqref{tsh_2}, we will show that  with negligible error, 
  we can restrict the supremum  to a set of labels  whose size is bounded 
  uniformly with respect to $N$. To this end,
we prove below that for every $\varepsilon>0$, there exists a finite subset 
$\mathcal V'=\mathcal V'(\eta_0,\xi_0,\mathcal V)$ such that 
\begin{eqnarray}\label{finite_max}
\vert\mathcal V'\vert\leq\frac{2a n}{\varepsilon},  &&
\sup_{v\in \mathcal V}
{\bf 1}_{{\mathcal V}_t}(v)L^v_t\leq
2(N\varepsilon+n)+\sup_{v\in \mathcal V'}
{\bf 1}_{{\mathcal V}_t}(v)L^v_t\\
&&
\sup_{v\in \mathcal V}
{\bf 1}_{{\mathcal V}_t}(v)R^v_t\leq
2(N\varepsilon+n)+\sup_{v\in \mathcal V'}
{\bf 1}_{{\mathcal V}_t}(v)R^v_t.\label{finite_max_2}
\end{eqnarray}
 Since $C^v_t=L^v_t+R^v_t$,
\eqref{finite_max} and \eqref{finite_max_2} allow us to bound the probability in 
\eqref{tsh_2} as follows:
\begin{eqnarray}\nonumber
\Prob\left(
\sup_{v\in\mathcal V,\,t>0} 
C^v_{t} > \gamma N
\right) & \leq & 
\Prob\left(
\sup_{v\in\mathcal V',\,t>0} 
C^v_{t} > N(\gamma-\varepsilon)+n
\right)\\
& \leq & \frac{2an}{\varepsilon}\sup_{v\in\mathcal V} \Prob\left(
\sup_{t>0} 
C^v_{t} > N(\gamma-\varepsilon)+n
\right)
\label{replace_sup} 
\end{eqnarray}
where the last inequality follows from a union bound. The probabilityt in \eqref{replace_sup} vanishes as $N\to+\infty$ by Lemma \ref {lemma:crossing_1}. This and (ii) of Remark \ref{rk:vanish} imply \eqref{tsh_2}.
 \\ \\
In order to prove \eqref{finite_max}--\eqref{finite_max_2}, we define 
a subdivision of $[-aN;aN]\cap\Z$, denoted 
\[
\mathcal E=\mathcal E(\eta_0,\xi_0)
=\{z_k:\,k=1,\ldots,m\},\quad m\leq\frac{2a n}{\varepsilon},
\]
as follows. 
First, we define $z_1$ as the spatial location  at time $0$  
of the leftmost blue particle.
Assume now $k\geq 1$ and $z_k$ has been defined. If there are at least 
$N\varepsilon$ red particles to the right of $z_k$  (i.e. in 
$(z_k,+\infty)\times W$),   we define $z_{k+1}$  as the position 
of the first blue particle to the right of  $z_k$ such that there are 
at least $N\varepsilon$ red particles  in $(z_k,z_{k+1}]\times W$.  
Otherwise, we set: either $m=k$ if $z_k$ is the position of the 
rightmost blue particle; or $m=k+1$ and $z_{k+1}$ to be the position 
of the next blue particle to the right of $z_k$. 
For $k\geq 1$,
we denote by $v_k$ the label of  a  blue particle initially at $z_k$, i.e.
\[
X^{v_k}_0(0)=z_k,\quad k=1,\ldots, m.
\]
Let $v\in\mathcal V\setminus\{v_k:\,k=1,\ldots,m\}$. Then we have 
\[
z_k< X^v_0(0)< z_{k+1}
\]
for a unique  $k=k(v)\in\{1,\ldots,m\}$. 
Observe that a red particle that crossed $v$ from right 
to left in the time interval $(0;t]$, and did not lie initially in  
$[z_k;z_{k+1}]\times W$ 
must have crossed $v_{k+1}$ in this time interval. Similarly, a red particle 
that crossed $v_{k+1}$ from left to right and did not start in  
$[z_k;z_{k+1}]\times W$ 
must have crossed $v$. Note also that  since there are initially at most 
$n$ particles in $\{z_k\}\times W$,  by definition of $z_k$, there are initially 
at most $N\varepsilon+n$ red particles in  $[z_k;z_{k+1}]\times W$.  
It follows from this 
and a similar comparison with $v_k$ that 
\[
R^v_t  \leq  R^{v_{k+1}}_t+N\varepsilon+n, \quad
%
%
%
L^v_t\leq L^{v_{k}}_t+N\varepsilon+n
\]
Thus \eqref{finite_max}--\eqref{finite_max_2} hold for  
$\mathcal V'=\{v_k:\,k\in\{1,\ldots,m\}\}$.
\end{proof}
 \subsubsection{ Proof of Lemma  \ref{lemma_coalescence}}\label{subsec:proof_coalesce}
\begin{proof}[Proof of Lemma \ref{lemma_coalescence}]
By Lemma \ref{lemma:number_steps} below, there exists 
$j\leq m^*(n,q(.,.))$ and a sequence $(x_0,\ldots,x_{j})\in W^{j+1}$, 
such that  for $k=0,\ldots,j-1$ and $i\in\{1;2\}$, 
\begin{eqnarray}
\zeta^{i,0}& := & \zeta^i_0\label{initial_steps}\\
\zeta^{i,k+1}& := & (\zeta^{i,k})^{x_k,x_{k+1}}
\label{coalescence_steps}
\end{eqnarray}
\be\label{admissible}
q(x_k,x_{k+1})\max\left\{
\zeta^{i,k}(x_k)[1-\zeta^{i,k}(x_{k+1})],\,
\quad i\in\{1;2\}
\right\} >0
\ee
and that the jump from $x_{j-1}$ to $x_{j}$ for 
$(\zeta^{1,j-1},\zeta^{2,j-1})$
generates a coalescence, that is,
\begin{eqnarray}\nonumber
&&\left\vert
\zeta^{1,j}(x_{j-1})-\zeta^{2,j}(x_{j-1})
\right\vert+
\left\vert
\zeta^{1,j}(x_{j})-\zeta^{2, j}(x_{j})
\right\vert\\
& < &
\left\vert
\zeta^{1,j-1}(x_{j-1})-\zeta^{2,j-1}(x_{j-1})
\right\vert+\left\vert
\zeta^{1,j-1}(x_{j})-\zeta^{2,j-1}(x_{j})
\right\vert .
\label{new_coalescence}
\end{eqnarray}
Let  $(T_k)_{k\geq 1}$  denote 
the successive jump times  of the Markov process 
$(\zeta^1_t,\zeta^2_t)_{t\geq 0}$,  set $T_0:=0$, and let
$(\widetilde{\zeta}^1_k,\widetilde{\zeta}^2_k)_{k\geq 0}
:=(\xi^1_{T_k},\xi^2_{T_k})_{k\geq 0}$ denote  the
discrete-time skeleton of this process .
Let   
\be\label{discrete_coalescence}
T^s:=\inf\{
k\geq 1:\,\mbox{ a coalescence occurs for $(\zeta^1_t,\zeta^2_t)$ at time $t=T_k$}
\}
\ee
denote the discrete coalescence time of this process (we do not know 
a priori whether $T^s<+\infty$ and do not need this information, 
but this follows from \eqref{positive_prob} below and strong Markov property). 
Conditioned on  $(\widetilde{\zeta}^1_k,\widetilde{\zeta}^2_k)_{k\geq 0}$,  
the  intervals $\Delta_k:=T_{k+1}-T_{k-1}$, $k\geq 0$, between transition times, 
are independent exponential random variables with  parameters 
$\lambda_k:=\lambda(\widetilde{\zeta}^{1}_k,\widetilde{\zeta}^{2}_k)$.
 These parameters are 
bounded from below by a constant  $q^*=q(n,q(.,.))$  independent of 
$(\zeta^1_0,\zeta^2_0)$;
indeed if we start with at least two opposite discrepancies,  before $T^s$, 
 we can never  reach a blocked configuration with all white or all  black  particles. 
By Lemma \ref{lemma:number_steps},  there is a constant $p^*=p^*(n,;q(.,.))$ 
such that,
uniformly over initial configurations,
\be\label{positive_prob}
\Prob\left( T^s\leq m^*\right)>0
\geq p^*(n,;q(.,.))>0
\ee
Hence, using the moment generating function of $\Gamma(m^*,q^*)$ distribution, 
 \begin{eqnarray*}
\Exp\left(
e^{-\theta T^*}
\right)
& = & \Exp\left(
e^{-\theta\sum_{k=1}^{T^s}\Delta_k}\right)
\\
& = & \sum_{j=1}^{+\infty} \Exp\left(
e^{-\theta\sum_{k=1}^j\Delta_k}\right)\Prob(T^s=j)\\
 &  \geq  & \sum_{j=1}^{+\infty} \left(\frac{q^*}{\theta+q^*}\right)^j\Prob(T^s=j)\\
& \geq & \sum_{j=1}^{ m^*} \left(\frac{q^*}{\theta+q^*}\right)^j\Prob(T^s=j)\\
& \geq & p^*\left(\frac{q^*}{\theta+q^*}\right)^{ m^*}.
\end{eqnarray*}  
\end{proof}
\begin{lemma}\label{lemma:number_steps}
Under assumptions of Lemma \ref{lemma_coalescence}, starting from any 
two coupled configurations $(\zeta^1_0,\zeta^2_0)$ with at least 
two opposite discrepancies,  there exists $ l\leq m^*(n,q(.,.))$ 
(with $m^*$ defined in \eqref{number_steps})  and a sequence 
$(x_0,\ldots,x_{ l })\in W^{ l+1}$
such that  for $k=0,\ldots, l -1$ and $i\in\{1;2\}$, 
\eqref{initial_steps}--\eqref{coalescence_steps} holds,
and that the last jump from $x_{ l-1}$ to $x_{ l}$ for 
$(\zeta^{1, l-1},\zeta^{2, l-1})$
generates a coalescence.
\end{lemma}
\begin{proof}[Proof of Lemma \ref{lemma:number_steps}]
Let $i,j$ denote locations of two opposite discrepancies 
(say a blue one at $i$, i.e. a $\zeta^1_0$ discrepancy, and a red one at $j$, 
i.e. a $\zeta^2_0$ discrepancy) between $\zeta^1_0$ and $\zeta^2_0$. 
In short, to describe the transition 
\eqref{coalescence_steps}--\eqref{admissible}, we will say that 
has jumped from $x_k$ to 
$x_{k+1}$ a  blue  particle if $\zeta^{1,k}(x_k)=1,\,\zeta^{2,k}(x_{k})=0$; 
a red particle if $\zeta^{1,k}(x_k)=0,\,\zeta^{2,k}(x_{k})=1$; 
a black particle if $\zeta^{1,k}(x_k)=\zeta^{2,k}(x_{k})=1$. \\ \\
Let  $(x_0=i,\cdots,x_{D}=j)$ be a path connecting $i$ to $j$, 
 with $D\leq n^*$,  either for $q(.,.)$ or $\check{q}(.,.)$. 
 Without loss of generality, we assume the first case. 
 First we define a sequence of admissible jumps leading from 
 $(\zeta^1_0,\zeta^2_0)$ to 
 $\left((\zeta^1_0)^{x_0,x_1},(\zeta^2_0)^{ x_0,x_1}\right)$. 
If $x_1$ is occupied by a red particle, the jump is directly feasible 
and leads to a coalescence, hence $l=1$. If $x_1$ is occupied 
by a white particle, i.e. $\zeta^1_0(x_1)=\zeta^2_0(x_1)=0$, 
the jump is also feasible, but no coalescence occurs.  Finally, 
if there is a black particle at  $x_1$, i.e. 
$\zeta^1_0(x_1)=\zeta^2_0(x_1)=1$, the blue particle at $x_0=i$ 
may not jump directly to $x_1$ and we need intermediate jumps. 
Let $x_{k_1},\ldots,x_{k_p}$ denote positions of black particles 
along the path $(x_0,\ldots,x_D)$, where $p\leq  D-1$ is the number 
of such particles, and $k_1<\cdots<k_p$. 
The black particle at $x_{k_p}$ can be moved to $x_{D-1}$ in 
$D-1-k_p\leq n^*-1-p$ steps along the path, since there are 
no more black particles there. 
At the end of this the black particle is at $x_{D-1}$.  Next, 
the black particle at $x_{k_{p-1}}$ can be similarly moved to 
$x_{D-2}$ in at most $n^*-p-1$ steps. At the end of these procedures, 
all black particles have been moved to $x_{D-1},\ldots,x_{D-p}$ 
in at most $p(n^*-1-p)$ steps. After this, the blue particle at $x_0$ 
can either be moved to $x_{D-p-1}$ (if $p= D-1$, none of the 
black particle moves has actually occurred because the whole path between $x_1$
and $x_{D-1}$ was occupied by black particles)  or meet 
a red particle along this path before, 
in which case the coalescence is achieved. 
In the former case, in $p$ steps, we can successively let 
each of the black particles jump one step further along the path, 
exchanging with the red particle, until the latter finds itself at 
$x_{ D-p}$.
Finally, the blue particle at $x_0$ can be moved  in at most
$D-p$ steps to the first red particle it encounters along the path 
(which is either at $x_{ D-p}$, or possibly earlier on the path. 
The total number of 
steps performed to achieve coalescence was at most $p(n^*-1-p)+n^*$.
This quantity achieves maximum value given by \eqref{number_steps}  for
$p=\left\lfloor \frac{n^*}{2}\right\rfloor$.
\end{proof}
\subsection{ Two-block estimate:  proof of Proposition 
\ref{prop:twoblock}}\label{subsec:prop_relax}
 The proof of  Proposition \ref{prop:twoblock} relies on the  
 {\em quasi-interface} property  stated in Lemma \ref{prop:interface} below. 
 The latter is an extension to multilane SEP of the {\em exact} interface 
 property for single-lane SEP, see \cite{Liggett1976}, which states that 
 the number of sign changes of discrepancies between two coupled processes 
 cannot increase in time. Such a property does not hold in the multilane case, 
 but we will show that it holds approximately with high probability.  
  The new difficulty here with respect to the single-lane setting is 
  that discrepancies of opposite type may cross each other 
  (in the sense of Definition \ref{def:crossings_0})  without 
  coalescing by using different lanes. 
 \begin{definition}\label{def_admissible}
 We call {\em admissible} a coupled configuration $(\eta,\xi)$ such that:  
 \textit{(i) }
 For every $z\in\Z$, the vertical layer $\{z\}\times W$ does not contain 
 two opposite discrepancies. \\ 
\textit{(ii)} 
There exists $A\in\N$ such that $\eta$ and $\xi$ 
 coincide  on 
$([A,+\infty)\cap\Z)\times W$ and on $((-\infty,-A]\cap\Z)\times W)$. 
\end{definition}
The above admissibility property will be used as an assumption 
for Lemma \ref{prop:interface}. Before proceeding further, we give 
an important example of admissible configurations for the sequel.
\begin{lemma}\label{lemma:ex_admissible}
 The coupled Gibbs state $(\eta^N,\xi^N)$ defined by 
 \eqref{couple_initial} is admissible if   
there exists $a>0$ such that $u_0(.)$ and 
 $v_0(.)$ 
 coincide  on $(-\infty;a]$ and on $[a;+\infty)$.
\end{lemma}
\begin{proof}[Proof of Lemma \ref{lemma:ex_admissible}.]
Let $(z,i)\in\Z\times W$. Since $\widetilde{\rho}_i(.)$ is increasing 
(cf. Proposition \ref{lemma_phi_super}), \eqref{couple_gibbs} and 
\eqref{gibbs_lane} imply that the  ordering between $\eta^N(z,i)$ and 
$\xi^N(z,i)$ is the same as between $u^N_z$
and $v^N_z$, hence independent of $i$. This implies condition 
\textit{(i)} of Definition \ref{def_admissible}. Moreover, 
 the assumption of the lemma  implies condition \textit{(ii)} 
of Definition \ref{def_admissible}. 
\end{proof}
 Given an admissible configuration $(\eta,\xi)\in\mathcal X^2$, 
 we say a subinterval $I$ of $\Z$ is blue  with respect to $(\eta,\xi)$, if:  
 (a) $I\times W$ contains at least a blue discrepancy and no red discrepancy; 
 (b) the endpoints $z^-\leq z^+$ of $I$ are such that $\{z^-\}\times W$ and 
 $\{z^+\}\times W$ both contain at least a blue discrepancy; 
 (c) $I$ is not contained in another interval satisfying (a), (b). 
 A red subinterval is defined similarly by exchagning the roles 
 of blue and red discrepancies. Connected components of the set of 
 points not belonging to a blue or red intervals are called 
 black and white intervals; black and white intervals contain 
 neither blue nor red discrepancies. 
The family $\mathcal C=\mathcal C(\eta,\xi)$ of  
blue, red and black and white intervals form a partition of $\Z$, 
that we call the {\em separating family} of $(\eta,\xi)$. 
 By (ii) of Definition \ref{def_admissible}, 
If $\eta=\xi$, $\mathcal C$ consists of the single interval $\Z$, 
otherwise,  $\mathcal C$
contains at least two semi-infinite intervals, and possibly 
finite intervals inbetween. Two successive intervals of $\mathcal C$ 
necessarily have different colours (whether blue, red or black and white). 
We point out that although we speak of an {\em interval} $I\subset\Z$, 
we inspect all lanes simultaneously. Thus the above notions are {\em not} 
lane by lane but two-dimensional notions.
 \\ \\
 The extinction time $t^*(I)$ of a blue or red interval $I\in\mathcal C$ 
 is the first time when all discrepancies  
initially in $I\times W$ have coalesced (if this does not occur, 
we set $t^*(I):=+\infty$).
We say $I$ has {\em survived} until time $t^*(I)$, and we set 
$I_t=\emptyset$ for $t\geq t^*(I)$. For $0\leq t< t^*(I-)$, let 
\begin{eqnarray}\label{def_ext_dis_1}
v_t(I) & := & \inf\left\{
v\in\mathcal V_t:\, Z_0^v\in I\times W
\right\}\\
\label{def_ext_dis_2}
w_t(I) & := & \sup\left\{
v\in\mathcal V_t:\, Z_0^v\in I\times W
\right\}
\end{eqnarray}
where $Z_t^v=X_t^v$ if $v$ is a blue label, or $Z_t^v=Y_t^v$ if $v$ 
is a red label. In words, $v_t(I)$ and $w_t(I)$ are the smallest and 
largest  label of a discrepancy initially in $I\times W$ and not 
coalesced by time $t$. These discrepancies have the same colour 
by definition of $\mathcal C$. We call them {\em extremal} labels of $I$ 
at time $t$, or extremal labels of $I_t$. 
We then define
\be\label{def_evol_int}
I_t:=\left[
X_t^{v_t(I)}(0), X_t^{w_t(I)}(0)
\right]
\ee
if $I$ is a blue interval, or
\be\label{def_evol_int_red}
I_t:=\left[
Y_t^{v_t(I)}(0), Y_t^{w_t(I)}(0)
\right]
\ee
if $I$ is a red interval. 
In particular, $I_0=I$. 
\\ \\
Note that if there are at least two lanes, due to possible crossings 
of opposite discrepancies,  
the property of being an admissible coupled configuration is not 
conserved by the dynamics at times $t>0$. In particular: 
(i) if $I$ is blue or red, 
$I_t\times W$ may contain  discrepancies opposite to 
the colour of $I$, even 
though the latter did not. 
However we still call it blue or red according to 
its initial colour (since, as established in the forthcoming 
Lemma \ref{prop:interface}, the interval will not contain many particles 
of the opposite colour).
(ii) A blue and a red interval may intersect at $t>0$. 
(iii) Since the second inequality in \eqref{ordered_labels} is not 
necessarily strict, two intervals of the same colour may intersect 
at time $t>0$, though only at their endpoints.
(iv) By definition of blue/red intervals and (iii), a blue or red label 
not coalesced at time $t$ necessarily belongs to one or two intervals 
of its colour, and possibly to one or two intervals of the opposite colour. \\ \\
Consider the set of points at time $t$ that belong to no blue or 
red interval. Connected components of this set are called 
{\em black and white} intervals at time $t$. By (iv) of the above paragraph, 
black and white intervals never contain any blue or red particle. Dynamically, 
this means that a blue or red particle willing to jump into a black and 
white interval actually pushes its boundary by one unit.  
We denote by $\mathcal C_t$ the collection of blue, red and black and 
white intervals at time $t$.
Since at times $t>0$, blue and red intervals are not necessarily disjoint,
$\mathcal C_t$ is 
not necessarily a partition of $\Z$. However, by definition, $\mathcal C_t$ 
is always a covering of $\Z$.
A black and white interval may disappear when it is surrounded by two blue 
or red intervals that become contiguous.
It can also be created when new space appears between two blue or red intervals 
that used to be contiguous to each other. However, the number of black 
and white intervals at any time cannot exceed its initial number. Though 
it is possible to give a formal definition of the individual evolution of 
each black and white interval, this is not necessary for our purpose.
 The above definition and properties are illustrated on Figure \ref{interface}. 
\\ \\
It appears from the above that the multilane case is quite different 
from the single-lane situation, where blue and red particles never cross, 
implying that blue and red intervals never intersect nor contain particles 
of the opposite colour, thereby defining clear interfaces between colours. 
In the multilane case, there can be particles of different colours 
on different lanes. Thus
even at a given spatial location, the colour is not clearly defined.
These  differences make the definition and control of interfaces more challenging.
However, the following lemma shows that
with several lanes, there remains but little invasion of blue or red intervals 
by the opposite colour. 
\begin{lemma}\label{prop:interface}
 Assume  that $(\eta^N,\xi^N)_{N\in\N}$ is a sequence of admissible 
 coupled configurations in the sense of Definition \ref{def_admissible}. 
 Assume further that, for some $a>0$ independent of $N$, we have   $A=aN$ 
  in condition \textit{(iii)}.  Denote by $\mathcal C^N$ the separating 
  family of $(\eta^N,\xi^N)$, and by $(I_t^N)_{t\geq 0}$ the evolution of 
  a blue or red interval $I\in\mathcal C^N$. 
For $m\in\N^*$, let   $\mathcal E^N_m$  be the event that at all times 
$t\geq 0$, the following holds:  for every  $I\in\mathcal C^N$  
such that $I^N_t\neq\emptyset$, $I^N_t\times W$  contains at most  $m$  
discrepancies of 
 type opposite to the ones  
initially in  $I\times W$. 
 Then:  \\ 
 (o) The number of black and white intervals at any time is at most  $M_N-1$, 
 where $M_N$  is the number of intervals in the initial separating family.\\ 
(i)  for every $h>0$, 
\be\label{prob_interface}
\lim_{N\to+\infty}\Prob( \mathcal E^N_{Nh})=1.
\ee
\end{lemma}
 \begin{remark}\label{rk:interface}
\textit{Lemma \ref{prop:interface} does not by itself define (approximate) 
{\em interfaces} between colours. This is actually not needed for our 
purpose.  However, we may recover at given time a partition of $\Z$ into 
$O(M_N)$ new intervals as follows, see Figure \ref{intervals}. 
(i) We merge intervals of the same colour intersecting through their endpoint.  
(ii) All intersections of red and blue intervals are declared black and white, 
and the final black and white intervals are connected components of this 
augmented black and white domain. (iii) Subtracting the previous intersections 
from blue or red intervals defines the final blue or red intervals. 
Black and white 
intervals now contain no more than $Nh$ blue or red particles with 
probability tending to $1$, whereas blue, resp. red intervals contain 
{\em almost surely no} particle of the opposite (red, resp. blue) colour. }
\end{remark}
\begin{figure}[htpb]
\centering
%
\includegraphics[width=15cm , height=10cm]
{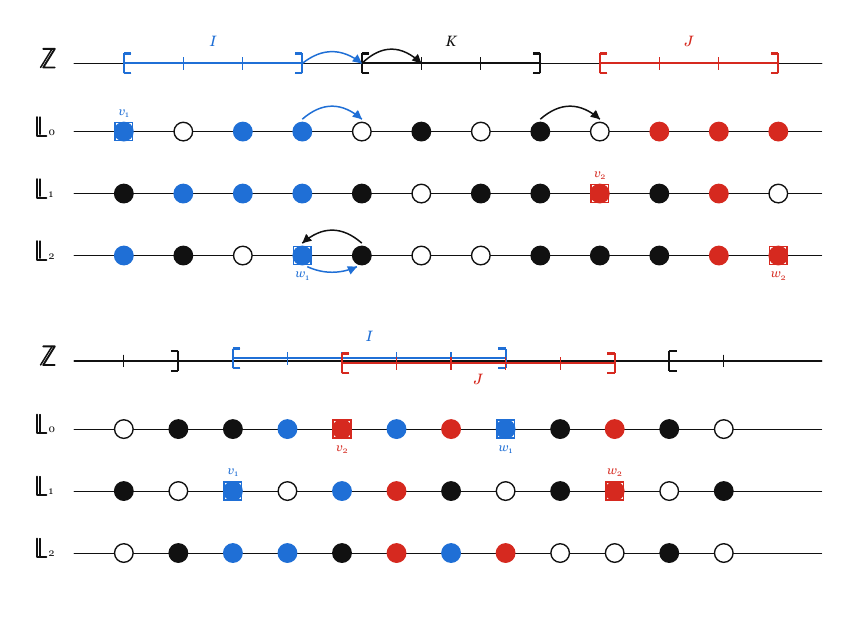}
\caption{
{\bf Top picture:} an initial admissible coupled three-lane configuration. 
The associated separating family of intervals is shown on the $\mathbb{Z}-$ 
axis on top of the picture. Squared blue and labels are the extremal ones 
inside the interval of the corresponding colour. 
Arrows illustrate some possible jumps and their effect on intervals. The double 
arrow is a black particle jumping over a blue one and exchanging due to 
priority. Jumps between $I$ and $K$ result in the blue interval pushing 
the black and white interval. The jump between $K$ and $J$  has no effect 
on these intervals.
{\bf Bottom picture:} a possible evolution of the initial extremal particles 
and associated intervals. Blue and red intervals are no longer disjoint. 
}
\label{interface}
%
\end{figure}
\begin{figure}[htpb] 
\centering
%
\includegraphics[width=15cm , height=8cm]
{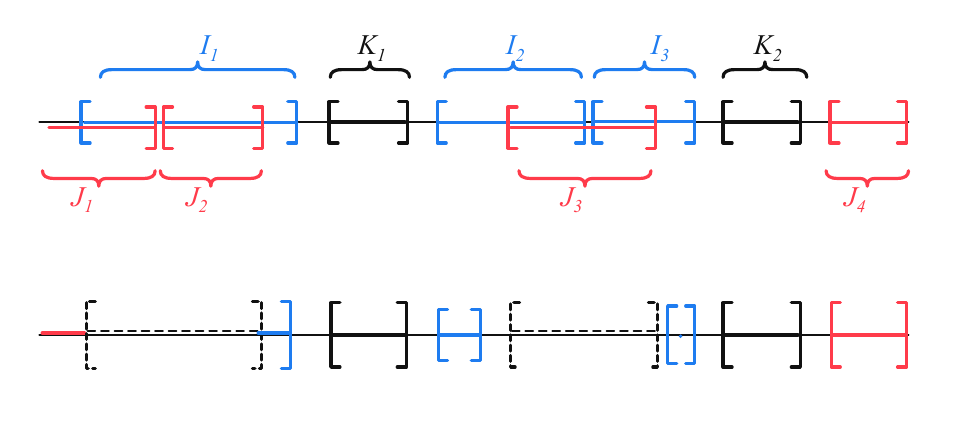}
\caption{
{\bf Top picture:} a possible evolution of intervals. Red intervals $J_1$ 
and $J_3$, like blue intervals $I_1$ and $I_2$, intersect at their endpoints.  
{\bf Bottom picture:} recomposed intervals as in Remark \ref{rk:interface}. 
Dashed intervals  are new black and white intervals obtained as intersections 
of former blue and red intervals.
} 
\label{intervals}
%
\end{figure} 
\begin{proof}[Proof of  Lemma \ref{prop:interface} ]
\mbox{}\\ \\
 {\em Proof of \textit{(o)}}. The number of black and white intervals 
 is highest when there is no interlap between blue and red intervals, 
 in which case we have to count the number of intermediate intervals  
 between $M^N$ consecutive non-overlapping intervals. \\ \\
 {\em Proof of \textit{(i)}.}  
 For notational convenience, we shall drop the superscript $N$. 
Consider a given $I\in\mathcal C$, for instance 
a blue interval.  
 We claim that the number of red particles   
in $I_t\times W$ at time $t$ is bounded from above by 
\be\label{blue_red_flux}
L^{v_t}_t+R^{w_t}_t\leq C^{v_t}_t+C^{w_t}_t
\leq 2\sup_{v\in\mathcal V}{\bf 1}_{\mathcal V_t}(v)C^v_t.
\ee
 where $v_t=v_t(I)$ and $w_t=w_t(I)$, see 
 \eqref{def_ext_dis_1}--\eqref{def_ext_dis_2},  while $L^v_t$ and $R^v_t$ 
 were defined in \eqref{def_crossings_0}--\eqref{def_crossings}. 
Indeed,  a red label $u$ lies in $I_t\times W$ at time $t$ if and only if 
it satisfies either \eqref{enter_right} or 
\eqref{enter_left} below:
\begin{eqnarray}
\label{enter_right}
Y^{u}_0(0) > X^{w_0}_0(0) & \quad\mbox{and}\quad &
Y^{u}_t(0) \leq X^{w_t}_t(0)\\
\label{enter_left}
Y^{u}_0(0) <  X^{v_0}_0(0) & \mbox{and} &
Y^{u}_t(0) \geq X^{v_t}_t(0)
\end{eqnarray}
In case \eqref{enter_right}, by \eqref{def_ext_dis_1},  we have 
$w_t\leq w_0$, hence \eqref{ordered_labels} implies
$X^{w_t}_0(0)\leq X^{w_0}(0)$.
Thus, \eqref{enter_right} implies
\be\label{enter_right_implies}
Y^{u}_0(0) > X^{w_t}_0(0)  \quad\mbox{and}\quad 
Y^{u}_t(0) \leq X^{w_t}_t(0)
\ee
Comparing with 
\eqref{def_crossings_0} shows that the number of red labels satisfying 
\eqref {enter_right_implies} is exactly $R_t^{w_t}$. 
A similar argument in case \eqref{enter_left},  comparing with 
\eqref{def_crossings},  shows that the number of red labels 
satisfying \eqref{enter_left} is bounded above by $L_t^{v_t}$. \\ \\
For a red interval,  the analogue of 
\eqref{blue_red_flux} holds, where  we replace $L^{v_t}_t$ and $R^{v_t}_t$   
 by   $L^{'v_t}_t$  and $R^{'v_t}_t$, and similarly for $w_t$, where $L^{'v}_t$
and $R^{'v}_t$ are  defined analogously to \eqref{def_crossings} for crossings 
of a red particle labelled $v$ by blue particles (exchanging the roles of 
$X$ and $Y$).  \\ \\
 Given the upper bound \eqref{blue_red_flux} and its  counterpart for 
 blue particles in a red interval,   the result follows  
 from Corollary \ref{lemma:crossing_2}. 
\end{proof}
\begin{proof}[Proof of Proposition \ref{prop:twoblock}]
Every local function on $\mathcal X$ is  the sum of a constant 
and a linear combination of  nonnegative  nonconstant nondecreasing functions.
 Thus it is enough to consider.  such functions $f$. For these
 we will show that,  for every $A>0$, $s>0$,  $r\in[0,n]$  and $\delta>0$, 
\begin{eqnarray} 
\label{relax2_0}
\lim_{l\to +\infty}\limsup_{N\to+\infty}
\Exp\left\{
N^{-1}\sum_{x\in\Z:\,|x|\leq NA}\Delta^{\delta\pm}_{x,l}(\eta^N_{Ns};r)
\right\}
& = & 0   \\
\label{relax2}
\lim_{\varepsilon\to 0}\limsup_{N\to+\infty}
\Exp\left\{
N^{-1}\sum_{x\in\Z:\,|x|\leq NA}\Delta^{\delta\pm}_{x,N\varepsilon}(\eta^N_{Ns};r)
\right\}
& = & 0
\end{eqnarray}
where
\begin{eqnarray*}
\Delta^{\delta,\pm}_{x,l}(\eta;r):=\psi^{\delta,\pm}\left[
M_{x,l}f(\eta),\overline{f}(
M_{x,l}\overline{\eta});r
)
\right]
\end{eqnarray*}
and the functions $\psi^{\delta,\pm}$ are defined by
\[
\psi^{\delta,+}(u,v;r)
={\bf 1}_{\{u>\overline{f}(r)+\delta,v<\overline{f}(r)-\delta\}},\quad
\psi^{\delta,-}(u,v;r)={\bf 1}_{\{u<\overline{f}(r)-\delta,v>\overline{f}(r)+\delta\}}.
\]
%
 We will then show at the end of this proof that there exists a constant 
 $C>0$ (depending only on $f$) such that
\be\label{bound_deltapm}
 \int_0^n  \left[\psi^{\delta,+}(u,v;r)+\psi^{\delta,-}(u,v;r)\right]dr
\geq C(|u-v|-2\delta)
\ee
for every $u,v\in[\min f,\max f]$,
so that \eqref{eq:relax} follows from \eqref{relax2} by integration 
with respect to $r$ and dominated convergence. The inequality 
\eqref{bound_deltapm} says that if $u$ and $v$ lie almost on the same side 
of $\overline{f}(r)$ for all $r$, then they must be close to each other. \\ \\
To establish \eqref{relax2},  
We consider the coupled process 
$(\eta^N_{Nt},\xi^N_{Nt})_{t\geq 0}$,
where  $\xi^N_0=\xi_0\sim\nu_r$.  Since $\nu_r$ is invariant, we have 
$\xi_t\sim\nu_r$ for all $t>0$.  We couple $\eta^N_0$ and
$\xi^N_0$ via \eqref{couple_initial}--\eqref{couple_gibbs}, the latter being 
a  local 
Gibbs state with uniform profile $r$.  Lemma \ref{lemma:ex_admissible} 
 ensures that $(\eta^N_0,\xi^N_0)$ is admissible in the sense 
 of Definition \ref{def_admissible}.  \\ 
By assumption on $u_0(.)$, for $N$ large enough, there is a (deterministic)  
separating family $\mathcal C^N$ of intervals for $(\eta^N_0,\xi_0)$ 
with a fixed number $M$ of intervals depending only on $\rho_0(.)$ and $r$ 
(namely the number of sign changes between $\rho_0(.)$ and $r$).  
We  
apply Lemma \ref{prop:interface} 
and denote simply by  ${\mathcal E}^N=\mathcal E^N_{N h}$ 
the event in the lemma, and by $\mathcal C^N_t$ the evolution 
of $\mathcal C^N$ at time $t$. 
By \textit{(o)} 
of Lemma \ref{prop:interface}, 
\be\label{number_int}
\vert\mathcal C^N_t\vert 
\leq 2M-1
\ee  
 We decompose the interval $\Z\cap[-NA,NA]$ into the following three random sets (that are not necessarily a partition). 
We denote by $B^{N,l}_{Ns}$, resp. $R^{N,l}_{Ns}$, $W^{N,l}_{Ns}$, 
 the set of $x\in[-NA,NA]\cap\Z$ 
such that $[x-l,x+l]\cap\Z$ lies inside a blue, resp.  red, 
resp. black and white interval. We denote by $E^{N,l}_{Ns}$ the set of 
$x\in[-NA,NA]\cap\Z$ 
 such that $[x-l,x+l]$ is contained in  none of the sets $B^{N,l}_{Ns}$, resp. 
$R^{N,l}_{Ns}$, $W^{N,l}_{Ns}$.    \\ \\
 The idea to obtain \eqref{relax2} is that  most $x\in [-NA,NA]$ lie in 
 one of these sets, so $\eta^N_{Ns}$ and 
$\xi_{Ns}$ are almost ordered on  $[x-l,x+l]\cap\Z$. 
This implies $M_{x,l}f(\eta^N_{Ns})$ and 
$M_{x,l}f(\xi_{Ns})\simeq \overline{f}(r)$ are ordered like 
$\overline{f}(M_{x,l}\overline{\eta}^N_{Ns})$ and 
$\overline{f}(M_{x,l}\overline{\xi}_{Ns})\simeq\overline{f}(r)$, 
where the  approximations follow from the law of large numbers 
for the stationary process;  
we have also used that $\overline{f}$ is continuous and nondecreasing, 
see remark after \eqref{eq_average}. This makes 
$\Delta^{\delta,\pm}(\eta^N_{Ns};r)$ small. \\ \\
Let us now make this precise. With the above definitions, we have 
\[
|E^{N,l}_{Ns}|\leq  2 (2M-1) l
%
%
\]
because for any subinterval $J$ of $\Z$, there are at most 
$2l$ points $x\in\Z$ 
for which $[x-l,x+l]$ crosses the boundary of  $J$.  It follows that
\be\label{nocolour}
\Exp\left\{
{\bf 1}_{\mathcal E^N}N^{-1}\sum_{x\in\Z:\,\vert x\vert\leq NA}
\Delta^{\delta,\pm}_{x, l}
(\eta^N_{Ns};r){\bf 1}_{E_{Ns}^{N, l}}(x)
\right\}\leq 4A(2M-1)\frac{l}{N}.
\ee
 We will next show that
\begin{eqnarray}
\label{blue_part_mic}
\lim_{l\to+\infty}\limsup_{N\to+\infty} \Exp\left\{
N^{-1}\sum_{x\in\Z:\,|x|\leq NA}
{\bf 1}_{\mathcal E^N}\Delta^{\delta,\pm}_{x,l}(\eta^N_{Ns};r)
{\bf 1}_{B_{Ns}^{N, l}}(x)
\right\}=0,\\
\label{blue_part}
\lim_{\varepsilon\to 0}\limsup_{N\to+\infty}
\Exp\left\{
N^{-1}\sum_{x\in\Z:\,|x|\leq NA}
{\bf 1}_{\mathcal E^N}\Delta^{\delta,\pm}_{x,N\varepsilon}(\eta^N_{Ns};r)
{\bf 1}_{B_{Ns}^{N, N\varepsilon}}(x)
\right\}
=0
\end{eqnarray}
and the same limits when replacing   $B^{N,l}_{Ns}$ with $R^{N,l}_{Ns}$ 
or $W^{N,l}_{Ns}$.   Since $\Delta^{\delta,\pm}$ is bounded, 
summing these three limits 
with \eqref{nocolour}, letting  $N\to+\infty$ followed by either 
$l\to+\infty$, or  $\varepsilon\to 0$ with $l=N\varepsilon$, 
 and using Lemma \ref{prop:interface}, we obtain \eqref{relax2}.
 \\  \\
Let  $n_0$ denote the smallest integer such that $f$ 
only depends on sites in $[-n_0,n_0]\times W$.  Then for all $\eta,\xi\in\mathcal X$,
\be\label{lip_f}
|f(\eta)-f(\xi)|\leq \Vert f\Vert_\infty\sum_{x\in V:\,|x(0)|\leq n_0}|\eta(x)-\xi(x)|
\ee
for $x\in\Z\cap[-NA,NA]$, we denote by $n_t^l(x)$  the number of
 wrong discrepancies in $[x-l-n_0,x+l+n_0]$.
By ``wrong'' discrepancies we mean red particles in a blue interval, 
blue particles in a red interval, or blue and red particles 
in a black and white interval. Hence
\be\label{def_wrong}
n_t^l(x)  = \left\{
\begin{array}{lll}
\dsp \sum_{z\in\Z\cap[x-l-n_0,x+l+n_0]}\sum_{i\in W}(\eta_t(z,i)-\xi_t(z,i))^- 
& \mbox{if} & x\in B^{N,l}_t,\\
\dsp \sum_{z\in\Z\cap [x-l-n_0,x+l+n_0]}\sum_{i\in W}(\eta_t(z,i)-\xi_t(z,i))^+
& \mbox{if} & x\in R^{N,l}_t,\\
\dsp \sum_{z\in\Z\cap [x-l-n_0,x+l+n_0]}\sum_{i\in W}|\eta_t(z,i)-\xi_t(z,i)| 
& \mbox{if} & x\in W^{N,l}_t.
\end{array}
\right.
\ee
By triangle inequality, 
\be\label{triangle_wrong}
\left(
M_{x,l}\overline{\eta}^{N}_{Ns} -  M_{x,l}\overline{\xi}_{Ns}
\right)^\pm\leq (2l+1)^{-1}\sum_{z\in\Z\cap[x-l,x+l]}
\sum_{i\in W}(\eta(z,i)-\xi(z,i))^\pm .
\ee
Thus by \eqref{lip_f}--\eqref{triangle_wrong} , 
 for all  $x\in B^N_{Ns}$,  we have  
\begin{eqnarray}\label{blue_1} 
M_{x,l}\overline{\eta}^{N}_{Ns}
 & \geq & M_{x,l}\overline{\xi}_{Ns}-\frac{ n^l_{Ns}(x)}{ 2l+1}
\geq r-\frac{ n^l_{Ns}(x)}{2l+1}-\varepsilon^{N,l,1}_{Ns}(x)\\
\label{blue_11}
\overline{f}\left(M_{x,l}\overline{\eta}^{N}_{Ns}\right) 
& \geq & \overline{f}(r)-\Vert\overline{f}'\Vert_\infty\left(
\frac{ n^l_{Ns}(x)}{2l+1}
+\varepsilon^{N,l,1}_{Ns}(x)
\right)\\
\label{blue_2}
M_{x,l}f(\eta^N_{Ns}) & \geq & M_{x,l}f(\xi_{Ns})
-||f||_\infty\frac{ n^l_{Ns}(x)}{ 2l+1}
 \\
 \nonumber & \geq &  
 \overline{f}(r)-||f||_\infty\frac{ n^l_{Ns}(x)}{ 2l+1}-\varepsilon^{N,l,2}_{Ns}(x)
\end{eqnarray}
where 
\be\label{def_eps}
\varepsilon^{N,l,1}_{Ns}:=M_{x,l}\overline{\xi}_{Ns}-r,\quad
\varepsilon^{N,l,2}_{Ns}:=M_{x,l}f(\xi_{Ns})-f(r).
\ee
 Similarly, for $x\in R^N_{Ns}$,
\begin{eqnarray}
\label{red_1}
M_{x, l}\overline{\eta}^{N}_{Ns}
 & \leq & M_{x,l}\overline{\xi}_{Ns}+\frac{ n^l_{Ns}(x)}{ 2l+1}
\leq r+ \frac{ n^l_{Ns}(x)}{ 2l+1}+\varepsilon^{N,l,1}_{Ns}(x)\\
\label{red_11}
\overline{f}\left( M_{x,l}\overline{\eta}^{N}_{Ns}\right)
 & \leq & \overline{f}(r)+\Vert\overline{f}'\Vert_\infty\left(
\frac{ n^l_{Ns}(x)}{2l+1}
+\varepsilon^{N,l,1}_{Ns}(x)
\right)\\
\label{red_2}
M_{x,l}f(\eta^N_{Ns}) & \leq & 
M_{x,l}f(\xi_{Ns})+||f||_\infty\frac{ n^l_{Ns}(x)}{ 2l+1}
 \\
 \nonumber & \leq &  
 \overline{f}(r)+||f||_\infty\frac{ n^l_{Ns}(x)}{ 2l+1}+\varepsilon^{N,l,2}_{Ns}(x).
\end{eqnarray}
For 
 $[x-l,x+l]\subset W^N_{Ns}$, both \eqref{blue_1}--\eqref{blue_2}
  and \eqref{red_1}--\eqref{red_2} hold. 
 It follows from \eqref{blue_1}--\eqref{blue_2} that
 \begin{eqnarray}\nonumber
 &&
 N^{-1}\sum_{x\in\Z:\,|x|\leq NA}
 {\bf 1}_{\mathcal E^N}\Delta^{\delta,\pm}_{x,l}(\eta^N_{Ns};r)
 {\bf 1}_{B^N_{Ns}}
 (x)\\
 & \leq & 
 N^{-1}\sum_{x\in\Z:\,|x|\leq NA}
 {\bf 1}_{\mathcal E^N}
 {\bf 1}_{B^N_{Ns}}(x)
 {\bf 1}_{
 \{
  \vert\vert f\vert\vert_\infty
  \frac{ n^l_{Ns}(x)}{ 2l+1}+\vert\varepsilon^{N,l,1}_{Ns}(x)\vert >\delta
\}
 }\nonumber\\
 & \leq & \delta^{-1}N^{-1}\sum_{x\in\Z:\,|x|\leq NA}
 {\bf 1}_{\mathcal E^N}
 {\bf 1}_{B^N_{Ns}}(x)\left(
 \vert\vert f\vert\vert_\infty
 \frac{ n^l_{Ns}(x)}{ 2l+1}+\vert\varepsilon^{N,l,1}_{Ns}(x)\vert
\right)\label{bound_deltaplus}
\end{eqnarray}
and similar relations hold with  $R^{N,l}_{Ns}$ and $W^{N,l}_{Ns}$. 
By Lemma \ref{prop:interface} 
 \be\label{total_wrong}
 {\bf 1}_{{\mathcal E}^N}
 \sum_{x\in \Z\cap[-NA,NA]}
 {\bf 1}_{
 B^N_{Ns}\cup R^N_{Ns}\cup W^N_{Ns}
 }(x) n^l_{Ns}(x)\leq (2l+1) (2M-1) Nh,
 \ee 
 where the factor $2M-1$ follows from \eqref{number_int}. On the other hand, 

\begin{eqnarray}
\nonumber
&&\Exp\left\{
 N^{-1}\sum_{x\in\Z\cap[-NA,NA]}
\left(
\vert\varepsilon^{N,l,1}(x)\vert+\vert\varepsilon^{N,l,2}(x)\vert
\right)
\right\}\\
& = & 2A\Exp\left(
\left\vert M_{x,l}{\xi}_{Ns}-r\right\vert+
\left\vert
M_{x,l}f(\xi_{Ns})-f(r)
\right\vert
\right)\stackrel{l\to+\infty}{\longrightarrow} 0.\label{lln}
\end{eqnarray}
 where the limit follows from the law of large numbers in $L^1$ 
 for the stationary process $\xi_.$. 
Gathering  \eqref{bound_deltaplus}--\eqref{lln}, 
 letting $N\to+\infty$, then $l\to+\infty$, or $l=N\varepsilon$ 
 with $\varepsilon\to 0$, and finally $h\to 0$,we obtain
\eqref{blue_part_mic}--\eqref{blue_part}. 
Similar limits are obtained replacing $B_{Ns}$ by $R_{Ns}$ or $W_{Ns}$. 
Combining these with \eqref{nocolour} yields \eqref{relax2_0}--\eqref{relax2}.
\\ \\
 We finally prove the claim \eqref{bound_deltapm}. Set 
\[
\Psi^{\delta,+}(u,v;\rho)={\bf 1}_{\{u>\rho+\delta,v<\rho-\delta\}},\quad
\Psi^{\delta,-}(u,v;\rho)={\bf 1}_{\{u<\rho-\delta,v>\rho+\delta\}}
\]
so that $\psi^{\delta,\pm}(u,v,r)=\Psi^{\delta,\pm}(u,v,\overline{f}(r))$. 
Since $f$ is nondecreasing, $\rho\mapsto\nu_\rho$ is 
stochastically nondecreasing, 
and $\nu_0$ and $\nu_{ n}$ are supported respectively  
on the empty and full configuration,  
the minimum $m$ and maximum $M$ of $\overline{f}$ coincide 
with those of $f$, and are given 
respectively by the value of $f$ at the empty and full configuration.
Setting
\be\label{bound_derivative}
C:=\frac{1}{\sup_{r\in[0,2]}\overline{f}'(r)},
\ee
we have
\begin{eqnarray*}
\int_0^2 \psi^{\delta,+}(u,v;r)dr 
& \geq & C\int_0^2\Psi^{\delta,+}(u,v;\overline{f}(r))\overline{f}'(r)dr\\
& = & C\int_m^M\Psi^{\delta,+}(u,v;\gamma)d\gamma=c(u-v-2\delta)^+ .
\end{eqnarray*}
Similarly, we obtain the lower bound $c(v-u-2\delta)^{ -}$ 
when replacing $\psi^{\delta,+}$ by $\psi^{\delta,-}$. The bound 
\eqref{bound_deltapm} follows by summing these two bounds.
\end{proof}
\section{Proof of Theorem \ref{th:relax_limit}}\label{sec:proof_relax}
We begin with an outline of the proof. Some useful properties 
of weakly coupled systems \eqref{general_system} are next gathered 
in Subsection \ref{subsec:material}, and the complete proof follows 
in Subsection \ref{subsec:proof_pde}.\\ \\
As mentioned in Subsection \ref{subsec:relax}, we must close 
the conservation law obtained after adding the equations 
in \eqref{general_relax_system}, namely
\be\label{total_conservation}
\partial_t R_\varepsilon(t,x)+\partial_x\left(
 \sum_{i=0}^{n-1}f_i(\rho_{i,\varepsilon}(t,x)) 
\right)=0 .
\ee
To this end, we must use  the balance terms $c_i(\vec{\rho})$ in \eqref {general_relax_system}  to
show  that  the relaxation to local equilibrium \eqref{lane_density}
 holds approximately for $\rho_\varepsilon$ as $\varepsilon\to 0$, 
 thereby allowing to approximate the flux in \eqref{total_conservation} 
 as follows (with $f$ given by \eqref{relax_flux}):
\be\label{approx_flux}
\rho_{i,\varepsilon}(t,x)\simeq\widetilde{\rho}_i[R_{i,\varepsilon}(t,x)],\quad
\sum_{i=0}^{n-1}f_i(\rho_{i,\varepsilon}(t,x)) \simeq{f}(R_\varepsilon(t,x)) .
\ee
We must also establish  entropy inequalities \eqref{entropy_namely}  for the limiting conservation law.
Following a usual scheme in relaxation theory (see e.g. \cite{chen,nat}),  
the  proof of   \eqref{entropy_namely}  involves 
constructing a suitable {\em dissipative} entropy, here of the form 
\be\label{big_entropy}
H(\vec{\rho}) =\sum_{i=0}^{n-1}h_i(\rho_i)
\ee
for the balance system \eqref{general_relax_system}, 
where  $\vec{\rho}=(\rho_0,\ldots,\rho_{n-1})$.  By dissipative, 
we mean the following:
\begin{eqnarray}
\label{dissip_1}
\left\langle
\overrightarrow{\nabla} H(\vec{\rho}),\vec{c}(\vec{\rho})\right\rangle  &\leq &  0 \\
\left\langle
 \overrightarrow{\nabla} H(\vec{\rho}),\vec{c}(\vec{\rho})\right\rangle=0
 &\Rightarrow&    \vec{c} (\vec{\rho})=\vec{0}, 
\label{dissip_2}
\end{eqnarray}
where  $\vec{c}(\vec{\rho})=(c_0(\rho),\ldots,c_{n-1}(\rho))$ 
The left-hand side of \eqref{dissip_1} is the entropy dissipation 
due to the balance term   $\vec{c}(\vec{\rho})$, 
and \eqref{dissip_2} expresses that zero dissipation implies local equilibrium. 
Regarding the latter, we will show that
(with $\mathcal F$ defined in \eqref{def_set_f})
\be\label{equiv_eq}
 \vec{c}(\vec{\rho}) =\vec{0}\Leftrightarrow\vec{\rho}\in\mathcal F
\ee
which implies \eqref{lane_density}.
Note that  the right to left implication in \eqref{equiv_eq} follows directly 
from \eqref{form_relax_1}--\eqref{form_relax_2} and \eqref{detailed_q}.
The corresponding entropy inequality writes
(cf. \eqref{entropy_cond})
\be\label{total_entropy}
\partial_t\left[
\sum_{i=0}^{n-1}h_i(\rho_{i,\varepsilon})
\right]+\partial_x\left[
\sum_{i=0}^{n-1}g_i(\rho_{i,\varepsilon})
\right] \leq
\varepsilon^{-1}\left\langle
\overrightarrow{\nabla} 
H(\vec{\rho}_\varepsilon),\vec{c}(\rho_\varepsilon)\right\rangle\leq 0
\ee
where $g_i$ is the entropy flux of $h_i$ for
\eqref{general_relax_system}. By space integration of \eqref{total_entropy} 
and using \eqref{dissip_1}--\eqref{dissip_2} and \eqref{equiv_eq}, 
we will arrive at \eqref{approx_flux}. This approximation 
can also be applied to entropies, hence in the limit, we will obtain 
the approximate entropy inequality
\be\label{entropy_ineq_eq}
\partial_t h[R_\varepsilon(t,x)]+\partial_x g[R_\varepsilon(t,x)]\lesssim 0
\ee
where the equilibrium entropy-flux pair is given by
\be\label{entropy_eq}
h(R):=\sum_{i=0}^{n-1}h_i[\widetilde{\rho}_i(R)],\quad
 g(R):=\sum_{i=0}^{n-1}g_i[\widetilde{\rho}_i(R)] .
\ee
Then we have to find a large enough family of entropies 
of the form \eqref{big_entropy} 
so that \eqref{entropy_eq} selects the unique entropy 
solution to \eqref{relax_law}. \\ \\
It turns out that the special structure 
\eqref{form_relax_1}--\eqref{form_relax_2}  of 
the balance term  $\vec{c}(\vec{\rho})$  allows us to use 
for the above purposes a combination of
Kru\v{z}kov entropies whereby the solution is coupled to 
stationary solutions supported 
on the equilibrum manifold.  Namely, we consider, cf. \eqref{kruzkov_plus}, 
\be\label{good_entropies}
 h_{r_i+}(\rho)=(\rho-r_i)^+,\quad
 g_{r_i+}(\rho)={\bf 1}_{\{\rho>r_i\}}[f_i(\rho)-f_i(r_i)],
\quad  \vec{r}=(r_0,\ldots,r_{n-1})\in\mathcal F .
\ee
\subsection{Preliminary material}\label{subsec:material}
We recall here some properties of entropy solutions of 
\eqref{general_system} established in \cite{hanat}.  In the following, 
$\vert\vert.\vert\vert$  denotes the $l^1$ norm on $\R^n$, i.e. 
\be\label{def_l1_vec}
\Vert\vec{u}\Vert=\sum_{i=0}^{n-1}\vert u_i\vert,\quad\vec{u}\in\R^n
\ee
Let $L^\infty (\R;\R^n)$ denote the set of Borel functions 
$\vec{f}= (f_0,\ldots,f_{n-1}):\R\to\R^n$ such that 
$f_i\in L^\infty(\R;\R)$ for every $i\in\{0,\ldots,n-1\}$. 
For $f\in L^\infty (\R;\R^n)$, we set
\be\label{def_l1_func}
\vert\vert\vec{f}\vert\vert_1:= 
\int_\R\Vert \vec{f}(x)\Vert dx
=\int_\R\sum_{i=0}^{n-1}\vert f_i(x)\vert dx\in[0;+\infty]
\ee
Then, \eqref{def_l1_func} defines a norm on the set, denoted $L^1(\R;\R^n)$,  
of functions $\vec{f}(.)$ for which it has finite value.
For notational convenience, 
we denote (though this does  not   lead to a norm)
\be\label{not_norm}
\Vert \vec{f}\Vert_{1,+}:=\int_\R \sum_{i=0}^{n-1} f_i(x)^+dx\in[0;+\infty]
\ee
\begin{theorem}\label{th:entropy_sol}\mbox{}
 There exists a constant $C>0$, depending only on the flux functions $f_i$ 
 and the balance terms $c_i$ in \eqref{general_system}, such that the following holds. 
Let $\vec{\rho}=(\rho_0,\ldots,\rho_{n-1})$ and 
$\vec{r}=(r_0,\ldots,r_{n-1})$ be entropy solutions 
to \eqref{general_system}, with respective initial data 
$\vec{\rho}^{\,0}=(\rho_0^0,\ldots,\rho_{n-1}^0)$ 
and $\vec{r}^{\,0}=(r_0^0,\ldots,r_{n-1}^0)$  in $(L^\infty(\R;[0;1]))^n$ . Then
\begin{enumerate}
\item 
For every $t\geq 0$, 
\eqref{form_relax_1}--\eqref{form_relax_2}.
\be\label{contraction_entropy_1}
\Vert\vec{\rho}(t,.)-\vec{r}(t,.)\Vert_1  
\leq  e^{Ct}\Vert\vec{\rho}^{\,0}(.)-\vec{r}^{\,0}(.)\Vert_1 .
\ee
\item 
(Finite propagation and uniqueness). Let
\be\label{speed_system}
V:=\max_{i=0,\ldots,n-1}\Vert f'_i\Vert_{\infty}
\ee
If  $\vec{\rho}^{\,0}(.)$ and $\vec{r}^{\,0}(.)$  coincide 
on the space interval $[a,b]$
where $-\infty\leq a<b\leq+\infty$, then for every $0\leq t\leq (b-a)/(2V)$, 
 $\vec{\rho}(t,.)$ and $\vec{r}(t,.)$  coincide on the space interval 
 $[a+Vt,b-Vt]$.\\ \\ 
In particular, 
there exists a unique entropy solution to \eqref{general_system} with given 
initial datum $\vec{\rho}^{\,0}(.):=\left(\rho^0_0(.),\ldots,\rho^0_{n-1}(.)\right)$.
\item  
Assume that for each $i=0,\ldots,n-1$, 
\be\label{quasi_monotone}
\forall j\neq i,\,
(\rho_0,\ldots,\rho_{n-1})\mapsto c_i(\rho_0,\ldots,\rho_{n-1})
\mbox{ is nondecreasing w.r.t. $\rho_j$} 
\ee
Then 
for every $t\geq 0$,
\be\label{contraction_entropy_2}
\Vert\vec{\rho}(t,.)-\vec{r}(t,.)\Vert_{1,+} 
\leq  e^{Ct}\Vert\vec{\rho}^{\,0}(.)-\vec{r}^{\,0}(.)\Vert_{1,+}
\ee
In particular the solution semigroup for \eqref{general_system} is monotone:
if $\rho_i^0(.)\leq r_i^0(.)$ for every $i=0,\ldots,n-1$, then 
$\rho_i(t,.)\leq r_i(t,.)$ for every $i=0,\ldots,n-1$. 
\end{enumerate}
\end{theorem}
In particular, looking at \eqref{form_relax_1}--\eqref{form_relax_2}, we have
\begin{lemma}\label{lemma:relax}
Assumption \eqref{quasi_monotone} is satisfied by the relaxation terms 
\eqref{form_relax_1}--\eqref{form_relax_2}.
\end{lemma}
The following corollary   considers solutions of  \eqref{general_system}  
with locally bounded variation (see e.g. \cite{amb}). To this end, 
we recall some facts about local variation. A function $f:\R^p\to \R$
is said to have locally bounded variation on $\R^p$ with respect to its 
$i$-th variable, $i\in\{0,\ldots,p-1\}$, if its partial derivative 
$\partial_i f$ in distribution sense
on $\mathcal D'(\R^p)$ is a locally finite signed measure on 
$(\R^p,\mathcal B(\R^p))$. We say $f$
has locally bounded variation on $\R^p$ if it has locally bounded variation 
on $\R^p$ with respect to all variables. Then the variation of $f$ 
on $A\in\mathcal B(\R^p)$ is defined by
\be\label{def_var_func}
{\rm TV}_A(f):=\sum_{i=0}^{p-1}\vert\partial_i f\vert(A)
\ee
where, for a signed measure $\mu$ on $(\R^p,\mathcal B(\R^p))$, 
$\vert\mu\vert$ is its variation measure defined by
\be\label{def_var_meas}
\vert\mu\vert(A)=\sup\left\{
\int_A \varphi d\mu,\,\mu\in C^{\infty}(\R^p)
\right\},
\quad A\in\mathcal B(\R^p)
\ee
 From classical BV theory (see e.g. \cite{amb}), we have 
\be\label{def_var_shift}
\vert\partial_i f\vert(A) = \lim_{\varepsilon\to 0}\varepsilon^{-1}
\int_{A}\vert
{f}(\vec{x}+\varepsilon\vec{e}_i)-{f}(\vec{x})
\vert d\vec{x}
\ee
where $(\vec{e}_i;\,i\in\{1,\ldots,p\})$ is the canonical basis of $\R^p$. 
Finally, the variation of a function $\vec{f}:\R^p\to\R^n$ written 
$\vec{f}=(f_0,\ldots,f_{n-1})$, is the sum of  variations of its 
coordinate functions, i.e., for $A\in\mathcal B(\R)$ and $i\in\{0,\ldots,p-1\}$,
\begin{eqnarray}
\label{def_var_vec_1}\vert\partial_i \vec{f}\vert &  
:= & \sum_{j=0}^{n-1}\vert\partial_if_j\vert\\
\label{def_var_vec_2}
{\rm TV}_A(\vec{f})& := &\sum_{j=0}^{n-1}\sum_{i=0}^{p-1}\vert\partial_i f_j\vert(A)
\end{eqnarray}
(note that in 
\eqref {def_var_vec_1}, $\partial_i \vec{f}$ is a vector measure).
\begin{corollary}\label{cor:entropy_sol}
Assume $\vec{\rho}^{\,0}(.)$ has locally bounded space-time variation. 
Then $\vec{\rho}(.,.)$ has locally bounded space-time variation, and for every $t>0$,
 $\rho(t,.)$  has locally bounded space variation.
Besides, for every $a,b\in\R$ with $a<b$,
and every $t>0$,
\be\label{space_var}
{\rm TV}_{[a;b]}[\vec{\rho}(t,.)]\leq
e^{Ct}{\rm TV}_{[a-Vt;b+Vt]}[\vec{\rho}^{\,0}(.)]
\ee
and
\be\label{space_time_var}
{\rm TV}_{[0;t]\times [a;b]}[\vec{\rho}(.,.)]\leq
(1+V)C^{-1}(e^{Ct}-1){\rm TV}_{[a-Vt;b+Vt]}[\vec{\rho}^{\,0}(.)]+
t\max_{i=0,\ldots,n-1}
\Vert c_i\Vert_{\infty}
\ee
\end{corollary}
\begin{proof}[Proof of Corollary \ref{cor:entropy_sol}]
\mbox{}\\ \\
Applying \eqref{contraction_entropy_1} to  $\vec{\rho}^{\,0}(.)$ and 
$\vec{\rho}^{\,0}(.+\varepsilon)$  yields 
\begin{eqnarray}\nonumber
\int_a^b\vert\vert
\vec{\rho}(t,x+\varepsilon)-\vec{\rho}(t,x)
\vert \vert  dx & \leq  & e^{Ct}\int_{a-Vt}^{b+Vt} \vert\vert
\vec{\rho}^{\,0}(x+\varepsilon)-\vec{\rho}^{\,0}(x)\vert\vert
dx\\ \label{apply_yields}
& \leq & e^{Ct}\varepsilon{\rm TV}_{[a-Vt;b+Vt]}[\vec{\rho}^{\,0}(.)]
\end{eqnarray}
 This and \eqref{def_var_shift} yield
\eqref{space_var}.
By \eqref{space_var} and slicing of BV functions (see e.g. \cite{amb}),
$\partial_x \vec{\rho}$ is a locally finite vector measure, and
\begin{eqnarray}\nonumber
\vert
\partial_x\vec{\rho} \vert([0;t]\times[a;b]) & = & \int_0^t
{\rm TV}_{[a;b]}[\vec{\rho}(s,.)]ds\\
\nonumber
 & \leq & \int_0^t
e^{Cs}{\rm TV}_{[a-Vt,b+Vt]}[\vec{\rho}^{\,0}(.)]ds\\
\label{slicing}
 &\leq & C^{-1}\left(e^{Ct}-1\right){\rm TV}_{[a-Vt;b+Vt]}[\vec{\rho}^{\,0}(.)]
\end{eqnarray}
 where the first inequality follows from \eqref{space_var} and $s\leq t$.
By \eqref{general_system} and \eqref{speed_system}, $\partial_s\rho(s,x)$ 
is also a locally finite measure, and,  for every $i\in\{0,\ldots,n-1\}$, 
\[
\vert
\partial_s\rho_i(s,x)
\vert \leq V\vert\partial_x\rho_i(s,x)\vert+\vert c_i[\rho(s,x)]\vert
\]
This and \eqref{slicing} imply \eqref{space_time_var}.
\end{proof}
\subsection{Proof of Theorem \ref{th:relax_limit}}\label{subsec:proof_pde}
 The proof is decomposed into four  steps. \\ \\
{\em Step one:  entropy inequality and relaxation.} 
Let $\vec{r}=(r_i)_{\{i=0,\ldots,n-1\}}\in\mathcal F$.
We write entropy inequality \eqref{entropy_cond} for the relaxation system 
\eqref{general_relax_system} with the Kru\v{z}kov entropy-flux pair 
$(h_{r_i+},g_{r_i+})$, 
 cf.  \eqref{good_entropies}, against a nonnegative test function 
 $\varphi\in C^0_K((0;+\infty)\times\R)$, and sum over $i=0,\ldots,n-1$ 
 to obtain \eqref{total_entropy}:
\begin{eqnarray}
\nonumber
&&\sum_{i=0}^{n-1}\int_0^{+\infty}\int_\R 
\left\{
\indic_{\{\rho_{i,\varepsilon}(t,x)>r_i\}}\left[
\rho_{i,\varepsilon}(t,x)-r_i
\right]\varphi'_t(t,x)
\right.\\
&& \qquad\qquad\qquad  +
\left.
\indic_{\{\rho_{i,\varepsilon}(t,x)>r_i\}}\left[
f_i(\rho_{i,\varepsilon}(t,x))-f_i(r_i)
\right]\varphi'_x(t,x)
\right\}dx\,dt\label{entropy_epsilon}\\
\nonumber
& \geq &  - \varepsilon^{-1}\int_0^{+\infty}\varphi(t,x)
\int_\R \sum_{i=0}^{n-1}\indic_{\{\rho_{i,\varepsilon}(t,x)>r_i\}}
\sum_{j=0}^{n-1}\left[
c_{ ji}(\vec{\rho_\varepsilon} (t,x))-c_{ ij}(\vec{\rho_\varepsilon}(t,x))
\right]dx\,dt
\label)
\end{eqnarray}
where $c_{ij}$ is given by \eqref{form_relax_2}. 
The right-hand side can be rewritten 
\begin{eqnarray}\nonumber
&& - \frac{\varepsilon^{-1}}{2}\int_0^{+\infty}\int_\R \left\{
\sum_{i=0}^{n-1}\sum_{j=0}^{n-1}
\left[
\indic_{\{\rho_{i,\varepsilon}(t,x)>r_i\}}
-
\indic_{\{\rho_{j,\varepsilon}(t,x)>r_j\}}
\right]\right.\\ \nonumber
&&\qquad\qquad\qquad\qquad\Bigl.\times\left[
c_{ji}(\vec{\rho}_\varepsilon(t,x))-c_{ij}(\vec{\rho}_\varepsilon(t,x))
\right]
\Bigr\} \varphi(t,x) dx\,dt\\
& = &  - \frac{\varepsilon^{-1}}{2}\int_0^{+\infty}\int_\R F\left(
\vec{\rho}_\varepsilon(t,x);\vec{r}
\right)\varphi(t,x) dxdt
\label{rhs}
\end{eqnarray}
where we set, for $\rho,r\in[0;1]^n$, 
\begin{eqnarray}\label{rhs_2}
F\left(
\vec{\rho};\vec{r}
\right) & := & \sum_{i=0}^{n-1}\sum_{j=0}^{n-1}F_{i,j}(\vec{\rho};\vec{r})
\end{eqnarray}
with
\be\label{def_fij}
F_{i,j}(\vec{\rho};\vec{r}):=\left[
\indic_{\{\rho_i>r_i\}}
-
\indic_{\{\rho_j>r_j\}}
\right]
\left[
c_{ji}(\vec{\rho};\vec{r})-c_{ij}(\vec{\rho};\vec{r})
\right]\\
\ee
The following lemma establishes the dissipation-equilibrium properties 
\eqref{dissip_1}--\eqref{dissip_2} and \eqref{equiv_eq}. 
\begin{lemma}\label{lem:relax}
For $\vec{\rho}\in[0;1]^n$, let
\begin{eqnarray}\label{rhs_3}
\widetilde{F}(\vec{\rho}) & := & \int_0^n F(\vec{\rho};\psi^{-1}(k))dk
 \end{eqnarray}
with (cf. Proposition \ref{lemma_phi_super})
\be\label{with_cf}
\psi^{-1}(k)=(\widetilde{\rho}_i(k))_{i=0,\ldots,n-1}\in\mathcal F .
\ee
Then $\widetilde{F}(\vec{\rho}) \leq  0$ 
for every $\vec{\rho}\in[0;1]^n$, and $\widetilde{F}(\vec{\rho})=0$ 
if and only if $\vec{\rho}\in\mathcal F$.
\end{lemma}
Next proposition establishes relaxation \eqref{approx_flux} to the equilibrium manifold.
\begin{proposition}\label{cor:relax}
For every $i=0,\ldots,n-1$ and every  $T,A>0$, 
\be\label{eq:cor:relax} 
\lim_{\varepsilon\to 0}\int_0^T\int_{-A}^A\left\vert
\rho_{i,\varepsilon}(t,x)-\widetilde{\rho}_i\left[
R_\varepsilon(t,x)
\right]
\right\vert dxdt=0 .
\ee
\end{proposition}
\begin{proof}[Proof of Lemma \ref{lem:relax}]
\mbox{}\\ \\
{\em Step one.} 
We show that, for all $(\vec{\rho}, \vec{r})\in[0;1]^n\times\mathcal F$ and 
$(i,j)\in\{0,\ldots,n-1\}^2$,
\be\label{fij}
F_{i,j}(\vec{\rho}, \vec{r})  \leq  0 .
\ee
If $\rho_i>r_i$ and $\rho_j>r_j$,  or $\rho_i\leq r_i$ and $\rho_j\leq r_j$, 
the result follows from \eqref{def_fij}. Otherwise, 
one of 
\eqref{either_that}--\eqref{or_that} below holds:
\begin{eqnarray}\label{either_that}
\rho_{i}>r_i,\quad\rho_{j}\leq r_j,\\
\label{or_that}
\rho_{i}\leq r_i,\quad\rho_{j}> r_j.
\end{eqnarray}
We consider \eqref{either_that}, \eqref{or_that} being similar.
In view of \eqref{form_relax_2} and \eqref{either_that}, 
\begin{eqnarray}
&&
q(j,i)\rho_{j}[1-\rho_{i}]-q(i,j)\rho_{i}[1-\rho_{j}]
\\
& \leq & q(j,i)r_j(1-r_i)-q(i,j)r_i(1-r_j)  = 0\label{integrand}
\end{eqnarray}
where the last  equality  follows  because by definition  
$\vec{r}=(r_i)_{i\in W}\in\mathcal F$, thus satisfies \eqref{detailed_q}. \\ \\
{\em Step two.} Assume  $\widetilde{F}(\rho)=0$. 
In view of \eqref{fij}, this implies 
\be\label{thenby}
F_{i,j}(\vec{\rho};\psi^{-1}(k))=0
\ee
for a.e. $k\in[0;n]$ and every $i,j\in\{0,\ldots,n-1\}$. 
 We argue by contradiction that $\vec{\rho}\in\mathcal F$. Assume the contrary; 
 then there exist $i,j\in\{0,\ldots,n-1\}$ such that \eqref{detailed_q} fails.  
Then by \eqref{thenby} and \eqref{def_fij},
for a.e. $k\in[0;n]$, one of \eqref{either_2}--\eqref{or_2} below holds:
\begin{eqnarray}
\label{either_2}
\rho_i>\widetilde{\rho}_i(k),\quad
\rho_j>\widetilde{\rho}_j(k)\\
\rho_i\leq\widetilde{\rho}_i(k),\quad
\rho_j\leq\widetilde{\rho}_j(k)
\label{or_2}
\end{eqnarray}
Since (cf. Proposition \ref{lemma_phi_super}) $\widetilde{\rho}_i(.)$
is nondecreasing and continuous, there exists a unique  $k_0$  such that
\eqref{either_2} holds for $k<k_0$, \eqref{or_2} holds for $k>k_0$, and
\be\label{equalities} 
\rho_i=\widetilde{\rho}_i(k_0),\quad
\rho_j=\widetilde{\rho}_j(k_0).
\ee
 But \eqref{with_cf}  implies that $\rho_i$ and $\rho_j$ given by 
 \eqref{equalities} satisfy 
 \eqref{detailed_q}, whence the contradiction. 
\end{proof}
\begin{proof}[Proof of Proposition \ref{cor:relax}]
Since the test function in \eqref{entropy_epsilon} is arbitrary, 
 using \eqref {entropy_epsilon},  we have
\be\label{dissip_vanish}
\lim_{\varepsilon\to 0}\int_0^T\int_{-A}^A
F\left[
\vec{\rho_\varepsilon}(t,x);\vec{r}
\right]dxdt=0
\ee
for every $\vec{r}\in\mathcal F$. It follows by dominated convergence that
\be\label{dissip_vanish_2}
\lim_{\varepsilon\to 0}\int_0^T\int_{-A}^A 
\widetilde{F}\left[
\vec{\rho_\varepsilon}(t,x)
\right]dxdt=0 .
\ee
Define the non-negative measure  (called {\em Young measure}) 
on $[0;+\infty)\times\R\times\R^n$,
\be\label{def-young}
M_\varepsilon(dt,dx,d\rho):=
 {\bf 1}_{[0,T]}(t){\bf 1}_{[-A,A]}(x)  
 \delta_{\rho_\varepsilon(t,x)}( d\vec{\rho})dxdt .
\ee
Since the family of measures $M_\varepsilon$ has compact support contained in 
$[0;T]\times[-A;A]\times[0;1]^n$ and constant mass  $2AT$, 
 it is tight with respect 
to weak convergence. Up to taking subsequences we may assume 
without loss of generality
that it converges to a limiting measure $M(dt,dx,d\rho)$.
Since the integral in \eqref{dissip_vanish_2} can be written
\[
\int_0^T\int_{-A}^A 
\widetilde{F}\left[
\vec{\rho}_\varepsilon(t,x)
\right]dxdt=\int\widetilde{F}(\vec{\rho}) M_\varepsilon(dt,dx, d\vec{\rho})
\stackrel{\varepsilon\to 0}{\to}\int \widetilde{F}(\vec{\rho}) M(dt,dx, d\vec{\rho})
\]
and  $\widetilde{F}$  is continuous, it follows from \eqref{dissip_vanish_2} 
and Lemma \ref{lem:relax} that $M$ is supported on $\mathcal F$. On the other hand, 
the integral on the left-hand side of \eqref{eq:cor:relax} can be written 
as the integral
 with respect to $M_\varepsilon(dt,dx, d\vec{\rho})$ of the function
\[
H_i(t,x,\vec{\rho}):= \left\vert
\rho_i-\widetilde{\rho}_i\left(
\sum_{j=0}^{n-1}\rho_j
\right)
\right\vert .
\]
Thus \eqref{eq:cor:relax} converges to the integral of $H_i$ with respect 
to $M(dt,dx, d\vec{\rho})$. This integral is $0$ because, by definition of 
$\widetilde{\rho}_i(.)$, 
$H_i$ vanishes on $\mathcal F$.
\end{proof}
The next step of the proof establishes the approximate entropy condition 
\eqref{entropy_ineq_eq}, its limit,  and shows that the family of Kru\v{z}kov 
entropies we considered generates all Kru\v{z}kov entropies 
for the limiting equations. \\ \\
{\em Step two: limiting entropy inequality.}
Let $c\in[0;+\infty)$, and set 
\be\label{def_ri}
r_i:=\widetilde{\rho}_i(c) .
\ee
In view of  Proposition  \ref{cor:relax}, since $h_{r_i+}$ and $g_{r_i+}$ 
are uniformly Lipschitz,  
on the left-hand side of \eqref{entropy_epsilon}, we can replace 
$\rho_{i,\varepsilon}(t,x)$  with $\widetilde{\rho}_i[R_\varepsilon(t,x)]$ 
with vanishing error 
as $\varepsilon\to 0$. Since $\widetilde{\rho}_i$ is an increasing function, 
the indicator function on the left-hand side can be replaced with
$\indic_{\{R_\varepsilon(t,x)>c\}}$ with vanishing error as $\varepsilon\to 0$.
Thus, collecting the left-hand side and the vanishing replacement error, we obtain
\begin{eqnarray}\label{collect}
&&\liminf_{\varepsilon\to 0} 
\int_0^{+\infty}\int_\R \indic_{\{R_\varepsilon(t,x)>c\}}\sum_{i=0}^{n-1}
\left\{
\left[
\widetilde{\rho}_i[R_\varepsilon(t,x)]-r_i
\right]\varphi'_t(t,x)
\right.\\
&+&
\left.
\left[
f_i(\widetilde{\rho}_i[R_\varepsilon(t,x)])-f_i(r_i)
\right]\varphi'_x(t,x)
\right\}dx\,dt\geq 0.\nonumber
\end{eqnarray}
By \eqref{relax_flux} and definition (cf. Lemma \ref{lemma_phi_super}) of 
$\widetilde{\rho}_i$ we have, for every $\rho\in[0;1]$,
\begin{eqnarray*}
\sum_{i=0}^{n-1} \widetilde{\rho}_i(\rho)  = \rho,\qquad\hbox{and}\qquad
\sum_{i=0}^{n-1} f_i\left(\widetilde{\rho}_i(\rho)\right)  = f(\rho)
\end{eqnarray*}
Thus, recalling \eqref{def_ri} we arrive at
\begin{eqnarray}\label{limit_entropy_cond}
&&\liminf_{\varepsilon\to 0} \int_0^{+\infty}\int_\R 
\left\{
\left[
R_\varepsilon(t,x)-c
\right]^+\varphi'_t(t,x)
\right.\\
&+&
\left.
\indic_{\{R_\varepsilon(t,x)>c\}}
\left[
f(R_\varepsilon(t,x))-f(c)
\right]\varphi'_x(t,x)
\right\}dx\,dt\geq 0.\nonumber
\end{eqnarray}
{\em Step three: passing to the limit.}
Here and in step  four below we make the temporary assumption 
that the initial datum $\rho^0(.)$ has locally bounded space variation. 
By \eqref{space_time_var}, the family $(\rho_\varepsilon(.,.))_{\varepsilon>0}$ 
has uniformly bounded space-time variation on any space-time rectangle.
 Thus is relatively compact in $L^1_{\rm loc}((0;+\infty)\times \R)^n$. 
This implies that the family $(R_\varepsilon)_{\varepsilon>0}$ 
is relatively compact in 
$L^1_{\rm loc}((0;+\infty)\times \R)$. 
Let $\rho(.,.)$ be any subsequential limit of this family as $\varepsilon\to 0$. 
Then \eqref{limit_entropy_cond} implies
\begin{eqnarray}\label{limit_entropy_cond_2}
&&\int_0^{+\infty}\int_\R 
\left\{
\left[
\rho(t,x)-c
\right]^+\varphi'_t(t,x)
\right.\\
&+&
\left.
\indic_{\{\rho(t,x)>c\}}
\left[
f(\rho(t,x))-f(c)
\right]\varphi'_x(t,x)
\right\}dx\,dt\geq 0.\nonumber
\end{eqnarray}
Therefore, $\rho(.,.)$ is an entropy solution to \eqref{relax_law} 
with flux function \eqref{relax_flux}.
 To ensure uniqueness of this subsequential  limit and hence 
 convergence of the whole sequence,
  we are left to prove the initial condition \eqref{initial_cond}. \\ \\
{\em Step four: initial condition.}
By step two, for every $a<b$ in $\R$,
\be\label{bysteptwo}
\lim_{\varepsilon\to 0}\frac{1}{t}\int_0^t
\int_a^b R_\varepsilon(s,x)dx\,ds=
\frac{1}{t}\int_0^t\int_a^b\rho(s,x)dx\,ds
\ee
where the limit is meant along the sequence of 
$\varepsilon$'s producing the limit point $\rho(.,.)$. 
By  \eqref{general_relax_system},
\be\label{sum_relax}
\partial_t R_\varepsilon(t,x)
+\sum_{i=0}^{n-1}\partial_x[ f_i(\rho_{i,\varepsilon}(t,x))]
=\sum_{i=0}^{n-1}
\varepsilon^{-1} c_i(\vec{\rho}_\varepsilon(t,x)).
\ee
Note that 
\be\label{sum_source}\sum_{i=0}^{n-1}c_i(\vec{\rho})=0\ee
for every $\rho\in[0;1]^n$. Thus for every $a<b$ in $\R$,
\begin{eqnarray}
\nonumber
\frac{d}{dt}\int_a^b R_\varepsilon(t,x)dx 
& = & \sum_{i=0}^{n-1}
\left\{f[\rho_{i,\varepsilon}(t,a^\pm)]-f[\rho_{i,\varepsilon}(t,b^\pm)]\right\}.
\label{thusforevery}
\end{eqnarray}
Note that the limits $\rho_\varepsilon(t,x^\pm)$ are always defined 
since $\rho_\varepsilon$
has bounded space variation by Corollary \ref{cor:entropy_sol}. 
Besides, \eqref{general_relax_system} 
implies that for every $x\in\R$,   
$\rho_{i,\varepsilon}(t,x^+)=\rho_{i,\varepsilon}(t,x^-)$ 
for a.e. $t>0$ ( this is why  \eqref{thusforevery} can be written indifferently with left or right limits).  
Integrating \eqref{sum_relax} in time, recalling \eqref{sum_source}, 
and using the initial condition 
\eqref{initial_cond} for $\rho_\varepsilon(.,.)$ yields 
\[
\int_a^b R_\varepsilon(t,x)dx=\sum_{i=0}^{n-1}\int_a^b\rho^0_i(x)dx+O_1[(b-a)t]
\]
where $\vert O_1[(b-a)t]\vert\leq C(b-a)t$ for a constant $C>0$ 
independent of $\varepsilon$, $a$, $b$ and $t$. This and \eqref{bysteptwo} imply
\[
\frac{1}{t}\int_0^t
\int_a^b R_\varepsilon(s,x)dx\,ds=\int_a^b\sum_{i=0}^{n-1}\rho^0_i(x)dx+O_2[(b-a)t].
\]
Letting $\varepsilon\to 0$, using \eqref{bysteptwo}, and then letting 
$t\to 0$, we obtain
\be\label{weobtain}
\lim_{t\to 0}\frac{1}{t}\int_0^t\int_a^b \rho(s,x)dx\,ds
=\int_a^b\sum_{i=0}^{n-1}\rho^0_i(x).
\ee
By \eqref{space_time_var}, $R_\varepsilon(.,.)$
satisfies a local variation bound independent of $\varepsilon$, 
thus $\rho(.,.)$ has locally bounded variation with  the same bound. Thus the limit
\[
\lim_{t\to 0}\rho(t,x)=:\rho(0^+,x)
\]
exists in $L^1_{\rm loc}(\R)$. It follows from \eqref{weobtain} that
\be\label{initial_relax_2}
\lim_{t\to 0^+}\rho(t,x)=\sum_{i=0}^{n-1}\rho^0_i(x)=R^0(x)
\ee
in $L^1_{\rm loc}(\R)$, where $R^0(.)$ is defined by \eqref{initial_relax}. 
The limit in \eqref{initial_relax_2} is the initial condition 
for entropy solutions 
of \eqref{relax_law}. Since the entropy solution to 
\eqref{relax_law}--\eqref{initial_relax} 
is unique, this ensures uniqueness of a subsequential limit for $R_\varepsilon$ 
as $\varepsilon\to 0$, hence convergence to this unique entropy solution. \\ \\
{\em Step  five: general initial condition.}
Assume now $\rho^0\in L^\infty((0;+\infty);[0;1])^n$. Then we can find a sequence 
$\left(\rho^{0,k}\right)_{k\in\N}$ such that $\rho^{0,k}$ has locally bounded 
space variation, and 
$\rho^{0,k}\to\rho^0$ in $L^1_{\rm loc}(\R)$. Let $\rho_{\varepsilon}^k$ 
denote the entropy solution
 to \eqref{general_relax_system} with initial datum $\rho^{0,k}$,  and
\[
R_{\varepsilon}^k(t,x):=\sum_{i=0}^{n-1}\rho_{i,\varepsilon}^k(t,x).
\]
Then, for $a<b$ in $\R$ and $T>0$, we have
\begin{eqnarray}
\nonumber
\int_0^T\int_a^b\vert R_\varepsilon(t,x)-\rho(t,x)\vert dx\,dt
 & \leq & \int_0^T\int_a^b\vert R_\varepsilon(t,x)-R_\varepsilon^k(t,x)\vert dx\,dt\\
\nonumber
& + & \int_0^T\int_a^b\vert R_\varepsilon^k(t,x)-\rho^k(t,x)\vert dx\,dt\\
\label{decomp_error}
& + & \int_0^T\int_a^b\vert \rho^k(t,x)-\rho(t,x)\vert dx\,dt .
\end{eqnarray}
Note that 
\be\label{decomp-distance}
\int_a^b\vert R^\varepsilon(t,x)-R^{\varepsilon,k}(t,x)\vert dx
\leq\sum_{i=0}^{n-1}\int_a^b\vert
 \rho_{i,\varepsilon}(t,x)-\rho^k_{i,\varepsilon}(t,x)
\vert dx .
\ee 
Thus, by Theorem \ref{th:entropy_sol}, the first term 
on the right-hand side of \eqref{decomp_error}
is bounded above by
\be\label{bound_1}
\sum_{i=0}^{n-1}\int_0^T\int_{a-Vt}^{b+Vt}\vert
\rho^0_i(x)-\rho^{0,k}_i(x)
\vert dx  dt
\ee
By specializing Theorem \ref{th:entropy_sol} to scalar conservation 
law \eqref{relax_law}, 
a similar bound holds for the third term on the right-hand side 
of \eqref{decomp_error}. Finally, 
for fixed $k\in\N$,  by the previous steps, the second term 
vanishes as $\varepsilon\to 0$. Thus, 
letting first $\varepsilon\to 0$ and then  $k\to+\infty$ 
in \eqref{decomp_error}, we obtain that 
the left-hand side vanishes as $\varepsilon\to 0$. 
 \section{Proof of many-lane limit: \thmref{approximating_G}}\label{sec:proof_limit}
\begin{proof}[Proof of \thmref{approximating_G}]
It is equivalent to show that for every $v\in\R$ 
and every sequence $v_n\to v$ as $n\to+\infty$, 
\be\label{equivalent_approx}
\widehat{u}(v-,1)\wedge\widehat{u}(v+,1)\leq \liminf_{n\to+\infty}\widehat{u}^n(v_n,1)
\leq \limsup_{n\to+\infty}\widehat{u}^n(v_n,1)
\leq \widehat{u}(v-,1)\vee\widehat{u}(v+,1) .
\ee 
We argue by contradiction.  Let $\varepsilon>0$ and
assume there exists a subsequence, still denoted by $(v_n)$, such that 
\be\label{contra}
d[\widehat{u}^n(v_n,1),\widehat{U}(v,1)]>\varepsilon .
\ee
Assume first  $\alpha\geq\beta$.
Let 
\[\rho_{n,i}:=i/n, \quad \rho^{n,i}=(i+1/2)/n\]
so that
\[
\widehat{G}^n(\rho_{n,i})=0,\quad \widehat{G}^n(\rho^{n,i})=F(\rho_{n,i}) .
\]
Recall that $\rho^{n,i}$ and $\rho_{n,i}$ respectively achieve 
the maximum and minimum of $\widehat{G}^n$ over $[\rho_{n,i},\rho_{n,i}+n^{-1}]$.
Let $u_n:=\lfloor n\widehat{u}^n(v_n,1)\rfloor$ and  $i_n:=\lfloor nu\rfloor$, so that
\be\label{sandwich_u}
\rho_{n,i_n}\leq u_n \leq \rho_{n,i_n}+\frac{1}{n}.
\ee
Then, using \eqref{sandwich_u},
\begin{eqnarray}
v_n u_n-\widehat{G}^n(u_n)
 \geq  v_n u_n-\widehat{G}^n(\rho^{n,i_n})
%
%
\geq   v_n \rho_{n,i_n}-F(\rho_{n,i_n})-\frac{1}{n}.\label{ineq_contra}
\end{eqnarray}
Since any minimizer of \eqref{dual_G} for $G=F$ lies in $\widehat{U}(v,1)$,
under \eqref{contra}, the last line of \eqref{ineq_contra} remains bounded 
away above  $F^*_{\alpha,\beta}(v)$. 
On the other hand,
taking $i'_n=\lfloor n\widehat{u}(v+,1)\rfloor$ and $u'_n=\rho^{n,i'_n}$, we have
\begin{eqnarray}
v_n u'_n-\widehat{G}^n(u'_n) =  v_n u'_n-F(\rho_{n,i'_n})
 \leq v_n \rho_{i,n}-F(\rho_{n,i'_n})+\frac{1}{n}\label{ineq_contra_2}
\end{eqnarray}
and the last r.h.s. of \eqref{ineq_contra_2} converges to 
$vu-F(u)=F^*_{\alpha,\beta}(v)$ for $u=\widehat{u}(v\pm,1)$. 
This contradicts the fact
that the  l.h.s. of \eqref{ineq_contra} achieves the minimum  
in \eqref{dual_G} for $G=\widehat{G}^n$.\\ \\
We now consider $\alpha\leq\beta$. By Proposition \ref{lemma:riemann}, 
$\widehat{u}(.,1)$ is the entropy solution
of the Riemann problem with initial datum \eqref{eq:init-riemann} 
for \eqref{burgers_H} with $F$ replaced by the null function $F_0$. 
Setting $u'_n=\rho_{n,i'_n}$, 
%
we have
\begin{eqnarray}
v_n u_n-\widehat{G}^n(u_n) & \leq  &  v_n u_n-F_0(u_n)\label{ineq_contra_3}\\
v_n u'_n-\widehat{G}^n(u'_n) & = & v_n u'_n-F_0(u'_n).\label{ineq_contra_4}
\end{eqnarray}
The r.h.s. of \eqref{ineq_contra_4} converges to 
$vu-F_0(u)=(F_0)^*_{\alpha,\beta}(v)$ for $u=\widehat{u}(v\pm,1)$.
This with \eqref{contra} contradicts the fact
that the l.h.s. of \eqref{ineq_contra_3} achieves the minimum 
in \eqref{dual_G} for $G=\widehat{G}^n$.
\end{proof}
\begin{appendix}
\section{Phase transitions in Theorem \ref{thm:phase_trans} }\label{app:graph}
We present below figures illustrating the different phases of the flux function 
stated in Theorem \ref{thm:phase_trans}.
Figures 1 to 3 show the phase diagram of inflection points with regions 
\textit{2, (i)} to \textit{2, (v)} of the theorem in the $(d;r)$ plane. 
Figures 3 to 9 illustrate the flux function parametrized by $r$ for a sample 
value of $d$ in each of these regions. Figures 4-5,  6-8, 9, 10, 11  
correspond respectively to \textit{(i), (ii), (iii), (iv), (v)}; 
while Figure 12 corresponds to the rescaled  limit $d=+\infty$, 
cf. \eqref{lim_G}. \\ \\
While the theorem is stated for $(d,r)\in[1;+\infty)^2$ and $d\geq 1$, by the 
symmetry property \eqref{sym}, all graphics can be extended to 
$(d,r)\in[0;+\infty)^2$. The transformation $d\to 1-d$  induces a symmetry 
with respect to $d=0.5$ in phase diagrams and around $\rho=1$ in fluxes. 
The transformation $r\to r^{-1}$ induces a nonlinear ''symmetry'' with respect 
to $r=1$ on  phase diagrams, and along parametrized curves in Figures 4-12, 
it induces a symmetry around $\rho=1$; in addition, the evolution of curves 
with respect to $r$ must be seen ''upside down''. \\ \\
We interpret  $d$ as an asymmetry parameter and $r$ as  a 
coexistence/segregation  parameter. The latter measures the 
separation/interplay between lanes (the larger $r\geq 1$ the stronger the 
separation and the weaker the interplay). When $r\to+\infty$, particles can only 
move to lane $0$, and only particles initially on lane $1$ use this lane until 
they eventually move to $0$ when possible.
Thus when the local density is less than $1$, only the ASEP flux on lane 
$0$ contributes, while a density above $1$ only activates the ASEP flux on 
lane $1$. When $r\to 0$, lanes $1$ and $0$ are inverted. 
As mentioned above, we interpret the number of inflection points of $G(\rho)$ 
between $0$ and $2$, as marking phase transitions in the $(d,r)$ plane. Indeed, 
as explained below, this number reflects the degree of separation of lanes (the 
more inflection points the more separation). We also view them as marking 
phase transitions in $\rho$ for given $(d,r)$. When two inflection points are present, 
the interval between them is a coexistence zone where both lanes contribute to 
the global flux, as particles begin to move significantly from lane $0$ to lane 
$1$ because of jamming. For $r\geq 1$, on the left of this interval, the low 
density zone is ''dominated'' by lane $0$ in a sense explained below, while the 
high density zone to the right is dominated by lane $1$, but to a lesser extent, 
since the drift on lane $1$ is smaller.\\ \\
Assume for instance the flux has two inflection points, as in Figure 9  below the 
pink layer (that is $r$ large enough). Between the two inflection points, as 
the density on lane $0$ gets higher, many particles depart from lane $0$ to lane 
$1$, where the density is still low; the flux on lane $1$ is thus more favourable 
to these particles, which begins to compensate for the decay of the flux due to 
strong exclusion on lane $0$, hence the transition from concave to convex. On 
the right of the second inflection point, lane $0$ is fairly jammed and no 
longer significantly contributes to the global flux. The density on lane 1 gets 
high enough for particle on lane $1$ to feel the exclusion effet there, hence 
the transition from convex to concave. 
Between the pink and yellow layers (that is intermediate values of $r$), the rate 
of particles moving to lane $1$ is sufficient to trigger coexistence but not 
to trigger the high density/lane 1 phase. Above the pink layer (that is $r$ 
close enough to $1$), there is never sufficient jamming on lane $0$ to 
significantly reduce its contribution, and too few particles on lane $1$ to 
start compensating the decay on lane $1$ at high density.
When $r=1$, there is complete coexistence as no lane is favoured. This results 
on Figures 3-11 in a rescaled TASEP flux where the maximum fluxes of the two lanes 
add up, cf. \eqref{ex:sym}.
When $r\to+\infty$, the separation on Figures 3-12 is complete
and there is a sharp transition from ASEP flux on lane $0$ to ASEP flux on lane 
$1$, cf. \eqref{ex:periodic}. \\ \\
The above arguments suggest that the number of inflection points should be a 
monotone function of $r$. Surprisingly, Theorem \ref{thm:phase_trans} proves this 
to be not always true (though ''most'' of the time). Figures 2 and 3 indeed 
exhibit a narrow region of the $(d,r)$ plane, that we call the 
''anomalous zone'', where the number of inflection points evolves from $0$ to $2$, 
$1$ and eventually $2$ as $r$ increases. This region, which corresponds to 
case \textit{2, (ii)} in Theorem \ref{thm:phase_trans}, is almost invisible on 
the global phase diagram of Figure 1 and only revealed by zooming as in 
Figures 2,3. It is also hard to detect on the graph of $G$ (cf. Figure 6), because 
the transitions    occur on a very narrow density interval where the flux is 
moreover almost linear. A plot of the second derivative of $G$ helps identify 
this region clearly (Figures 7,8). \\ \\
When $d>1$, lane $0$ has a positive drift and lane $0$ a negative one (Figure 11). 
The fully segregated phase $r=+\infty$ has a positive TASEP flux followed by 
a negative one. The progressive separation as $r$ grows does not exhibit more 
than one inflection point, cf. Figures 1 and 11, which we can view as the 
transition point between low density/lane 0 and high density/lane 1 phases. 
The special case $d\to+\infty$ (Figure 12) of two lanes with opposite drifts 
equal in absolute value, cf. \eqref{lim_G}, yields fluxes symmetric with respect 
to $(1;0)$. When $r=1$, we obtain a null flux. There is an immediate transition 
at $r=1+=\lim_{d\to+\infty}r_3(d)$ (cf.Figure 1) towards a segregated behaviour 
where all fluxes have an inflection point at $\rho=1$. In the borderline case 
$d=1$ (Figure 10), one of the lanes has $0$ drift, and the corresponding TASEP 
flux in the fully segregated phase is flat. 
\bigskip
%
\begin{center}
\begin{minipage}{406pt}
\includegraphics[width=406pt , height=6cm]
{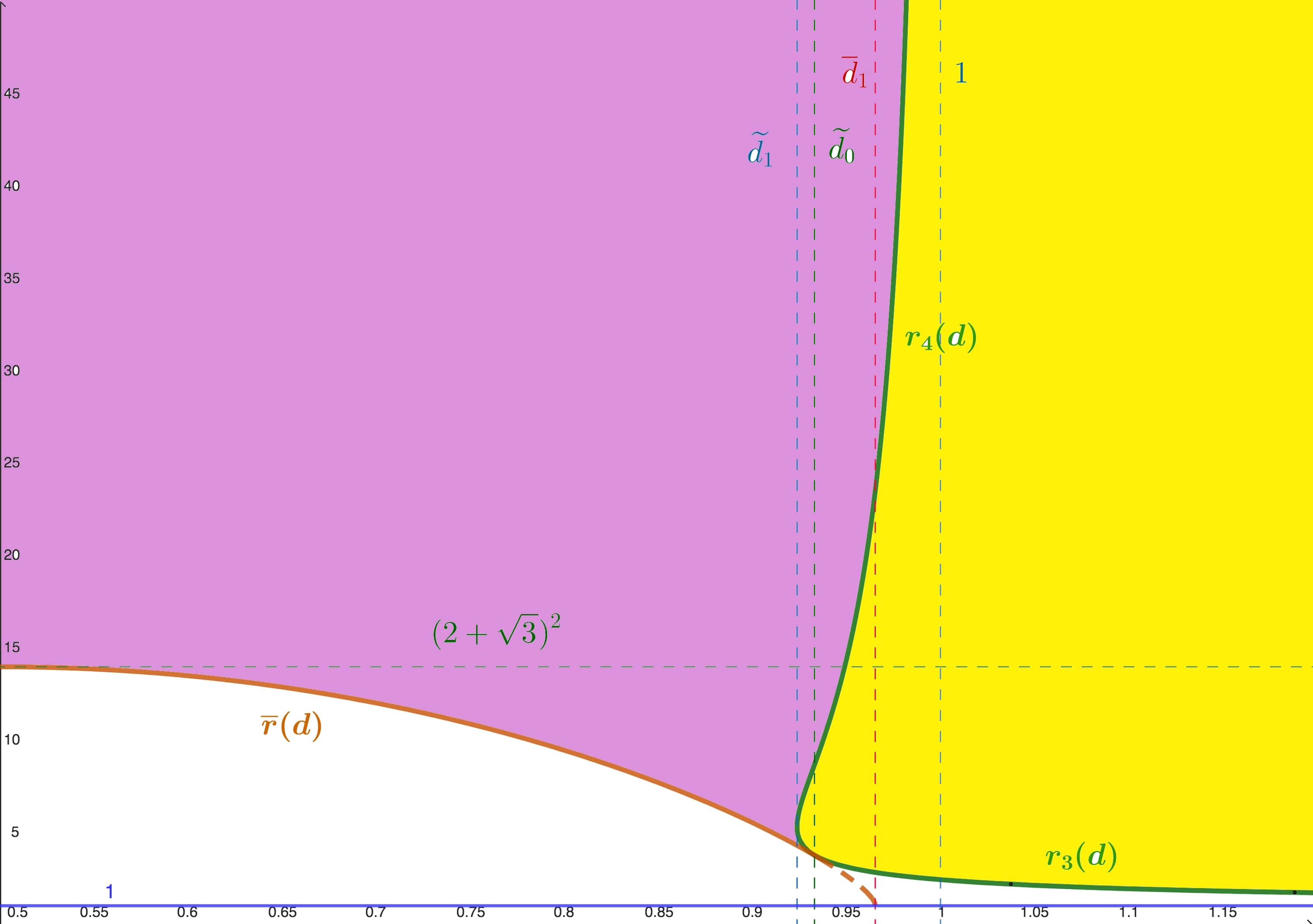}
\captionof{figure}{
\small{\bf Phase diagram 1:} $\tilde{d}_1$, $\tilde{d}_0$, $\bar{d}_1$ 
are defined in \eqref{critical_d}. The rightmost dashed line is $d=1$.  
In the white region, the flux function is strictly concave; 
in the yellow (resp. pink) region, it has one (resp. two) inflection points. 
Intervals $(-\infty;\tilde{d}_1$), $(\tilde{d}_1;\tilde{d}_0)$,
$(\tilde{d}_0;1)$ and $(1;+\infty)$ respectively correspond to cases 
\textit{(i), (ii), (iii)} and \textit{(iv)--(v)} in Theorem 
\ref{thm:phase_trans}. Curves $r_3(d)$ and $r_4(d)$ are the lower and upper 
branches of the solution $r$ of $g(r)=0$ (green curve), cf. \eqref{ineq:G2-2}, 
with  $r_4(d)\to+\infty$ as $d\to 1-$. Curve $\bar{r}_1(d)$ is defined 
by \eqref{def:bar-d1}--\eqref{ineq:G2-rho_0-2}. The blue line at the bottom is $r=1$. 
}
\label{phase-global}
\end{minipage}
\end{center}
%
%
\begin{center}
\begin{minipage}{406pt}
\includegraphics[width=406pt , height=7cm]
{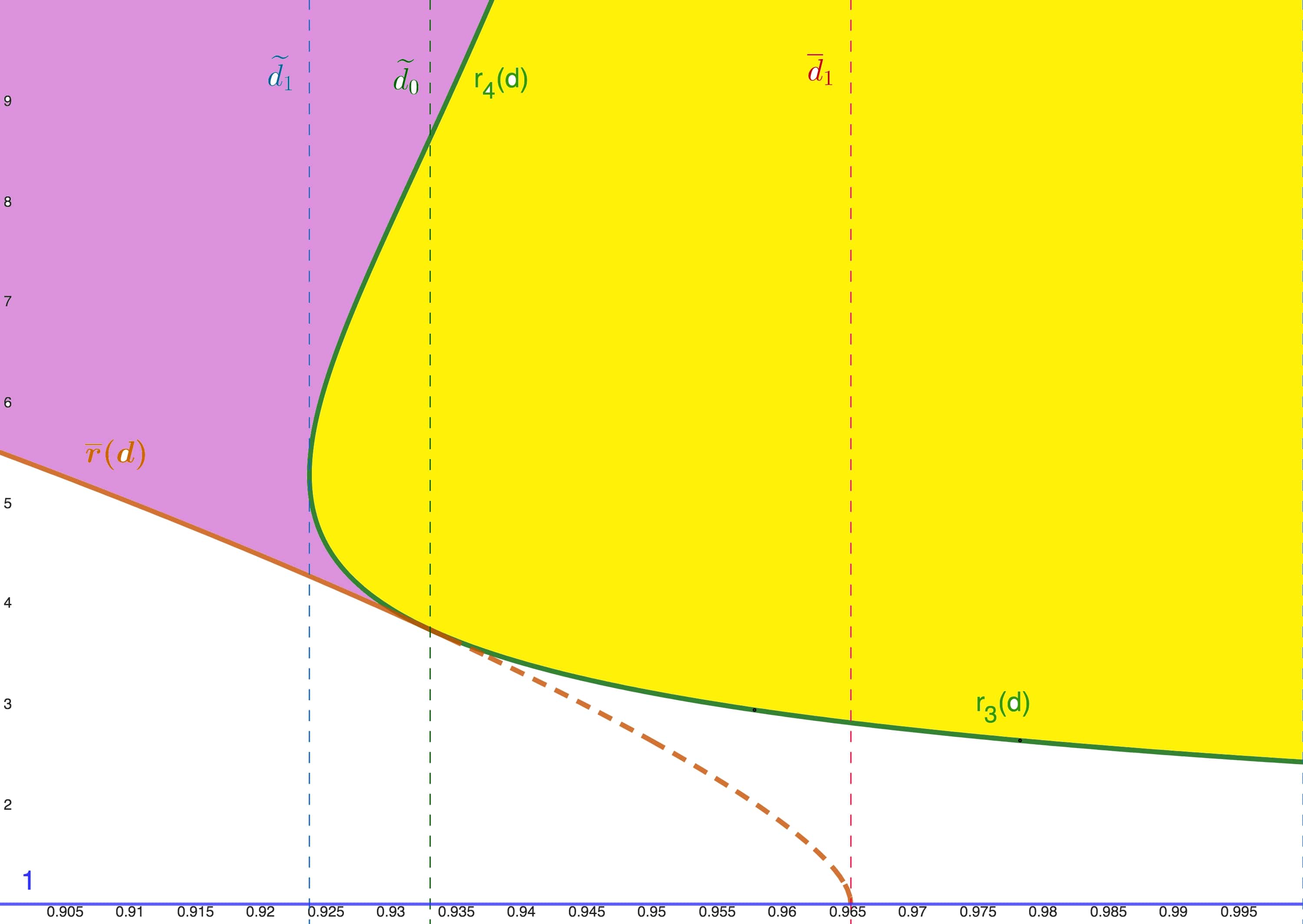}
\captionof{figure}{
\small{\bf Phase diagram 2:}  zooming around the intersection of 
$r_4(d)$ with $\bar{r}(d)$ reveals the anomalous region between 
$d=\tilde{d}_1$ and $d=\tilde{d}_0$.
}
\label{phase2}
\end{minipage}
\end{center}
%
\begin{center}
\begin{minipage}{406pt}
\includegraphics[width=406pt , height=8.5cm]
{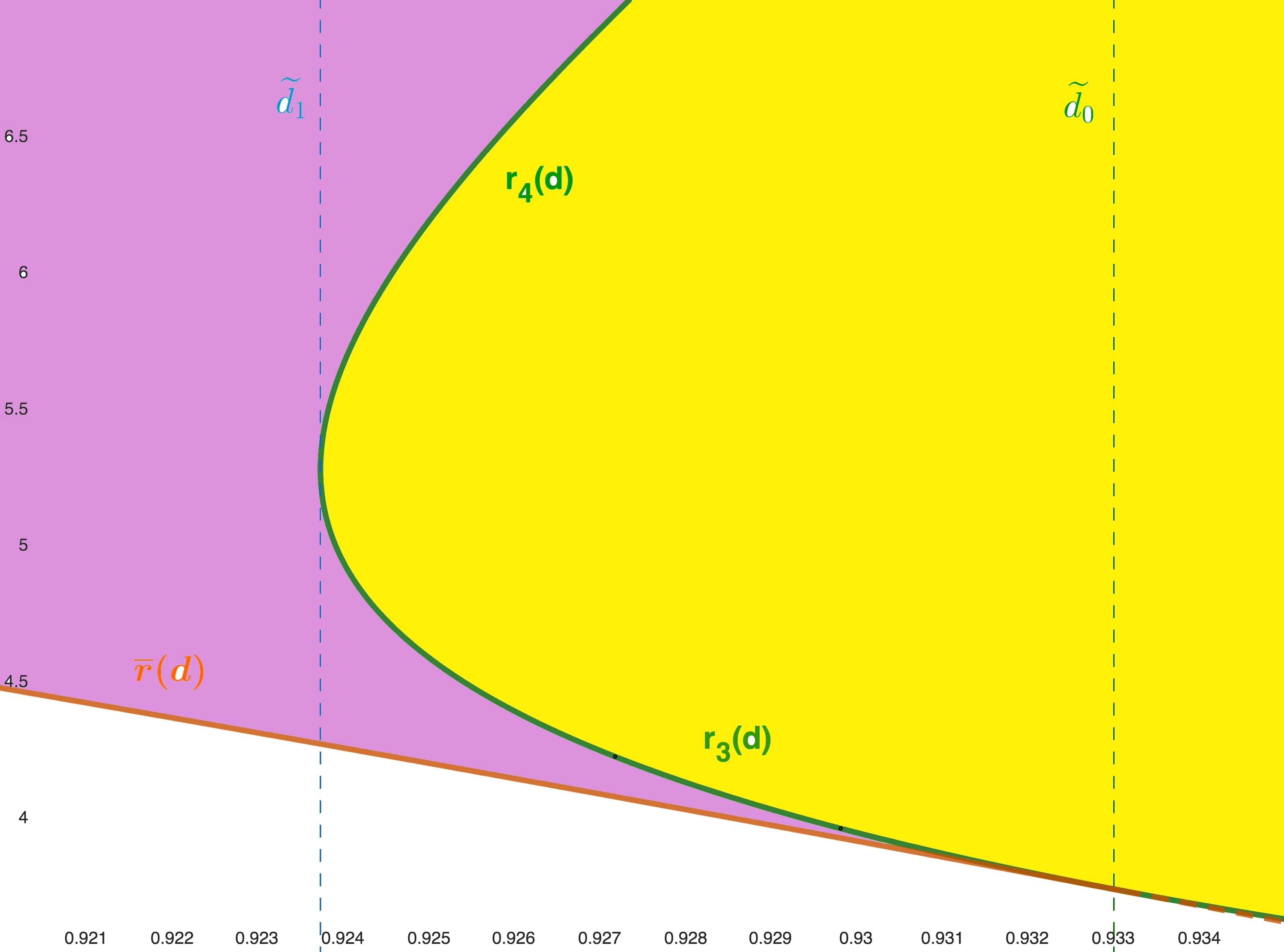}
\captionof{figure}{
\small{\bf Phase diagram 3:}  zooming around the anomalous region.
}
\label{phase3}
\end{minipage}
\end{center}
%
\begin{center}
\begin{minipage}{406pt}
\includegraphics[width=406pt ,  height=8cm]
{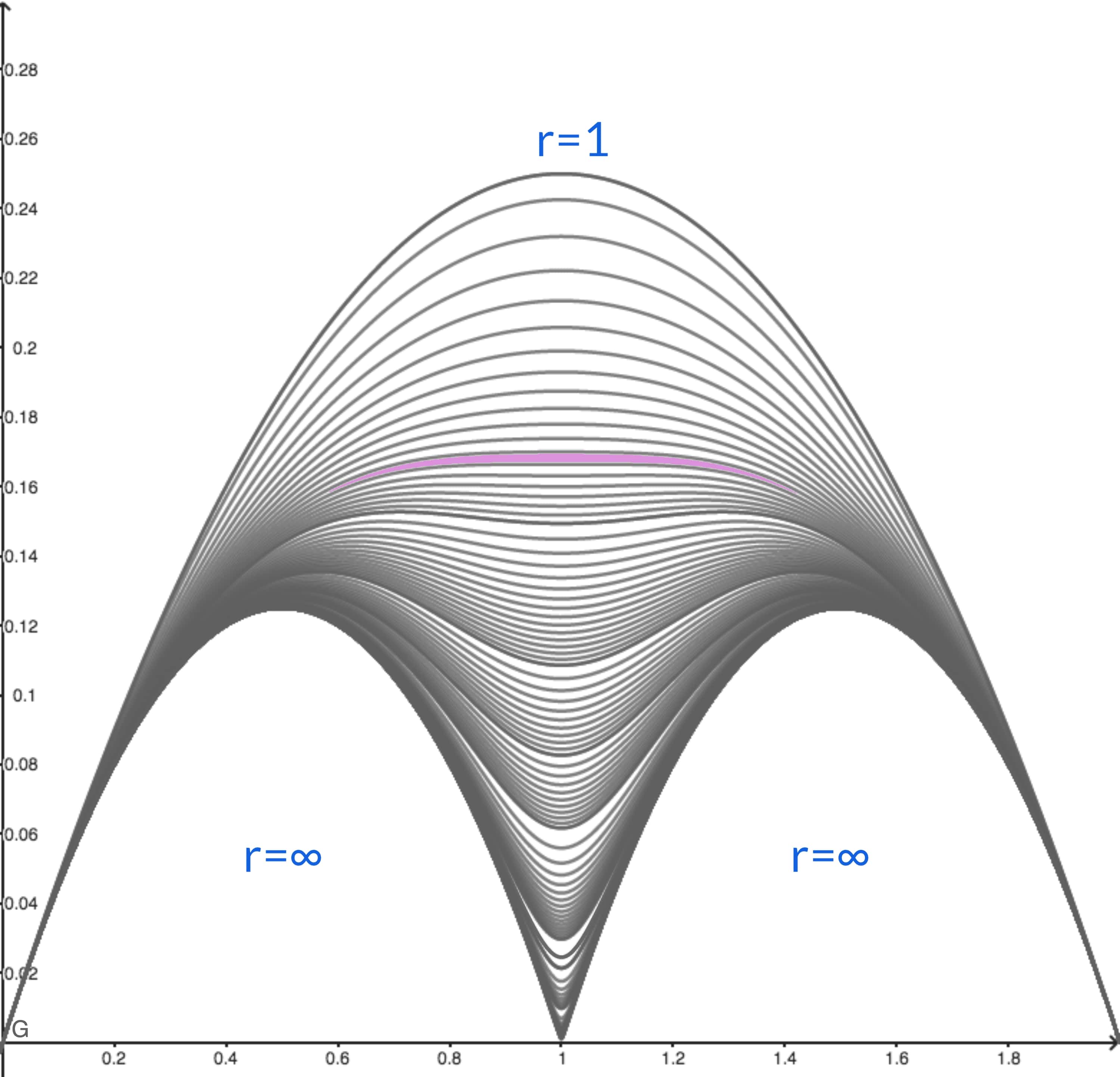}
\captionof{figure}{
\small{\bf Flux function 1:}  $d=0.5$, $\gamma_0=\gamma_1=0.5$, symmetric graph. 
Top concave curve: $r=1$; bottom double-bump curve $r=+\infty$. Unique 
phase transition from $0$ to $2$ inflection points occurs at 
$r=\bar{r}_1(0.5)=(2+\sqrt{3})^2\simeq 14$ (cf. fig 1) inside the pink layer.
}
\label{phase1}
\end{minipage}
\end{center}
%
\begin{center}
\begin{minipage}{406pt}
\includegraphics[width=406pt , height=7.5cm]
{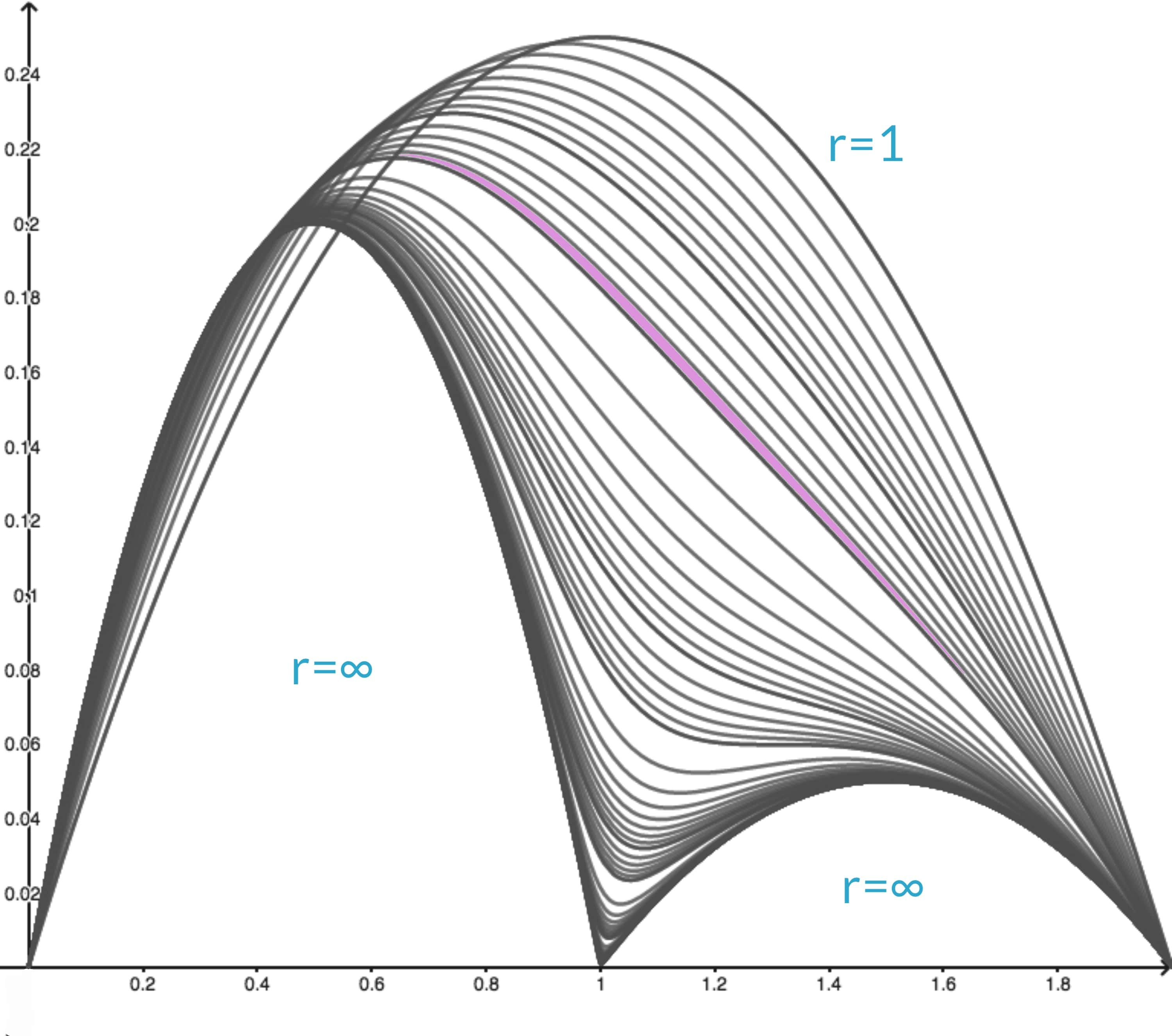}
\captionof{figure}{
\small{\bf Flux function 2:}  $d=0.8$, $\gamma_0=0.8$, $\gamma_1=0.2$.  
Top concave curve: $r=1$; bottom double-bump curve $r=+\infty$. Unique 
phase transition from $0$ to $2$ inflection points occurs at 
$r=\bar{r}_1(0.8)\simeq 9.4$ (pink layer).
}
\label{phase2bis}
\end{minipage}
\end{center}
%
\begin{center}
\begin{minipage}{406pt}
\includegraphics[width=406pt , height=7.5cm]
{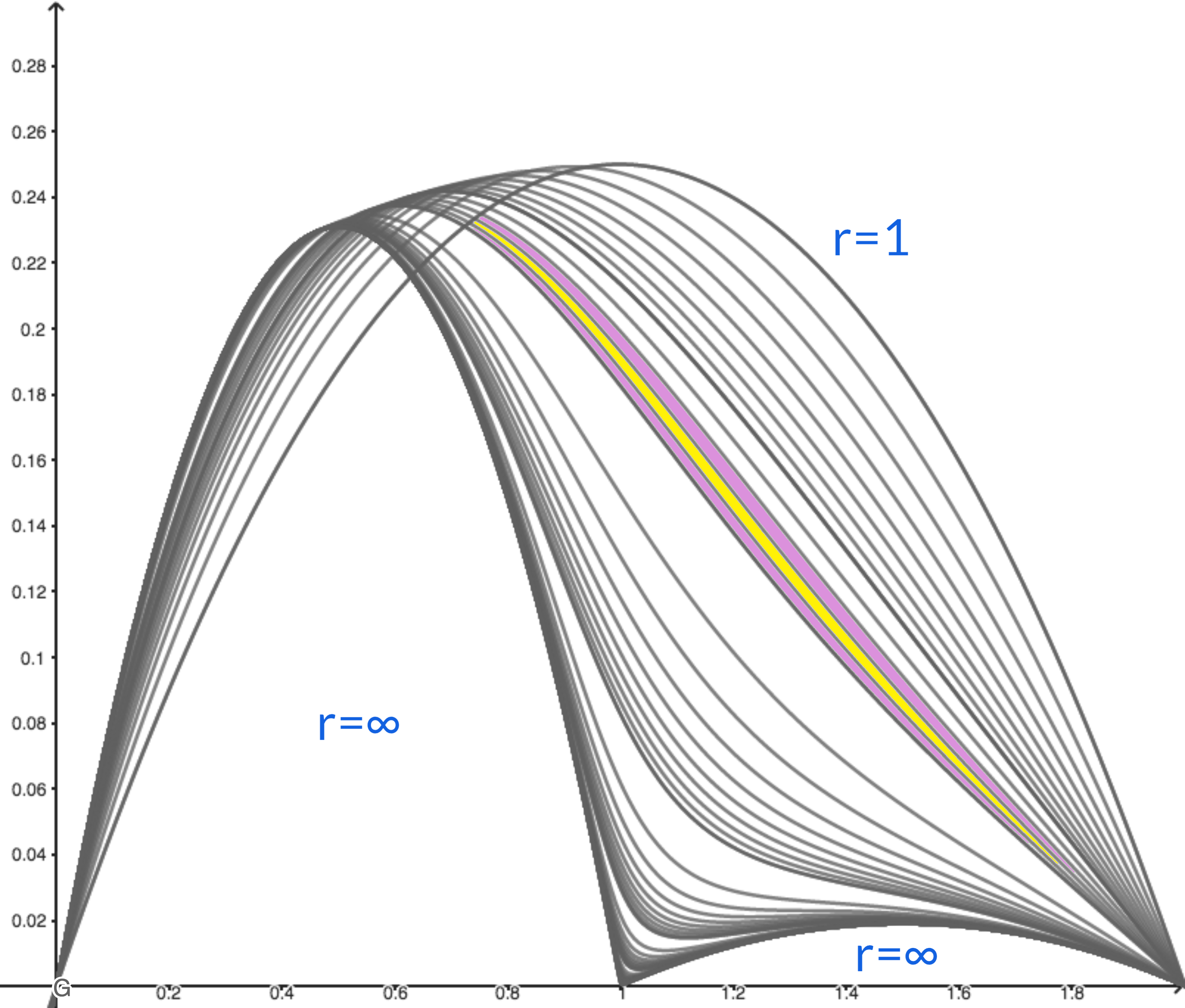}
\captionof{figure}{
\small{\bf Flux function 3a:}  ''Anomalous'' zone in phase diagrams 2 and 3.
 Here $\tilde{d}_1<d=0.924<\tilde{d}_0$, $\gamma_0=0.924$, $\gamma_1=0.076$.  
 Top curve: $r=1$; bottom curve $r=+\infty$. Three transitions occur: 
 $\bar{r}_1(d)\simeq 4,25$ ($0$ to $2$ inflection points, upper pink layer), 
 $r_3(d)\simeq 4,9$ ($2$ to $1$, yellow layer), $r_4(d)\simeq 5,7$ ($1$ to $2$, 
 lower pink layer). The anomalous interval is $(\bar{r}_1(d);r_3(d))$. 
}
\label{phase-anomalous}
\end{minipage}
\end{center}
%
%
\begin{center}
\begin{minipage}{406pt}
\includegraphics[width=406pt , height=7cm]
{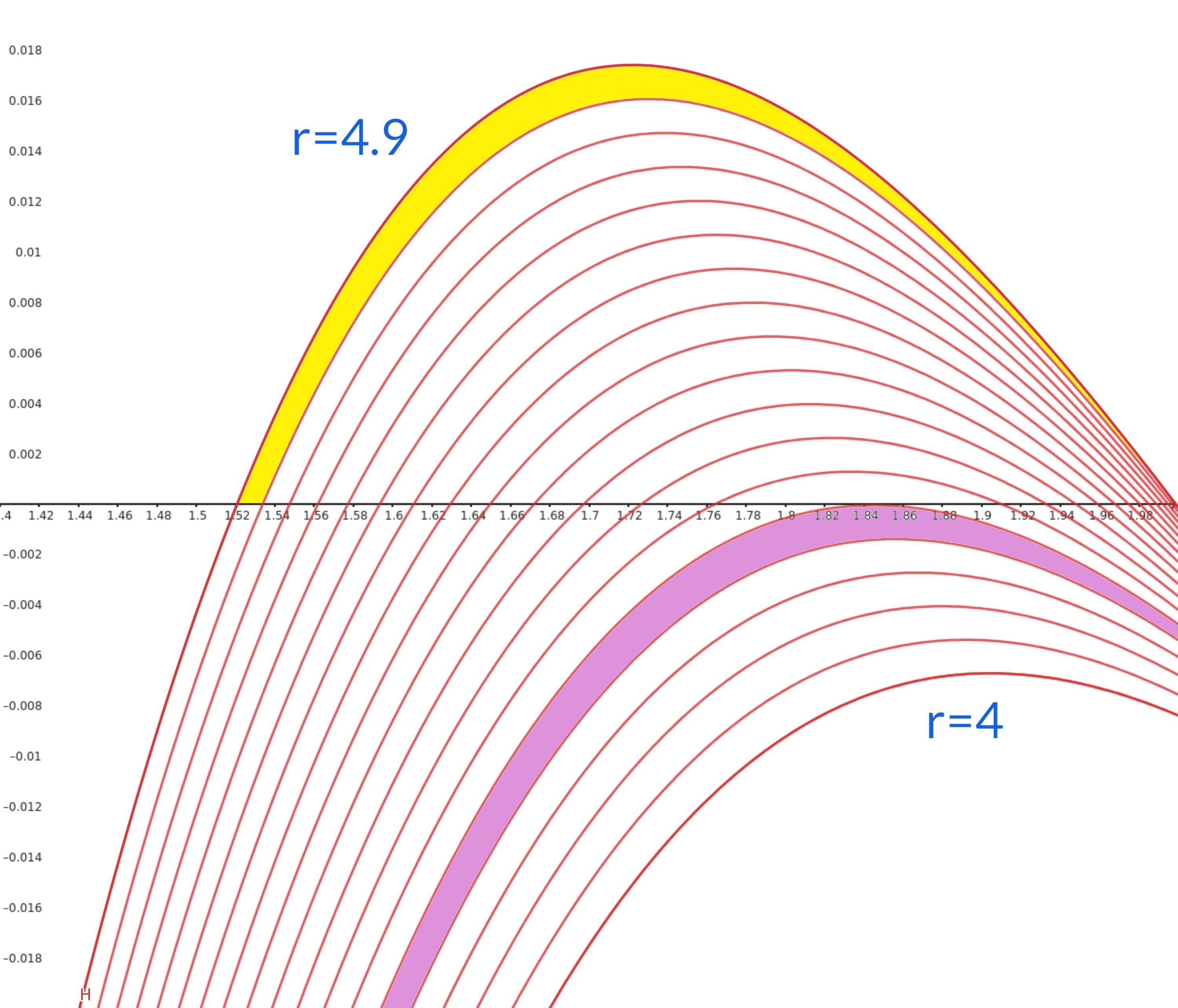}
\captionof{figure}{
{\bf Flux function 3b:}  ''Anomalous'' zone continued with plot of $G''(\rho)$ 
for $r=4$ (bottom curve) to $r=4.9\simeq r_3(d)$ (top curve) spaced by 
$\Delta r=0.05$. Again, $\tilde{d}_1<d=0.924<\tilde{d}_0$, $\gamma_0=0.924$, 
$\gamma_1=0.076$. The transition at $\bar{r}_1(d)\simeq 4,25$ ($0$ to $2$ 
inflection points) occurs on t top of the pink layer. The transition  
$r_3(d)\simeq 4,9$ ($2$ to $1$) occurs on the curve on  top of the yellow 
layer. Intersections with $x$-axis indicate positions of inflection points. 
}
\label{phase-anomalous1}
\end{minipage}
\end{center}
\begin{center}
\begin{minipage}{406pt}
\includegraphics[width=406pt , height=7cm]
{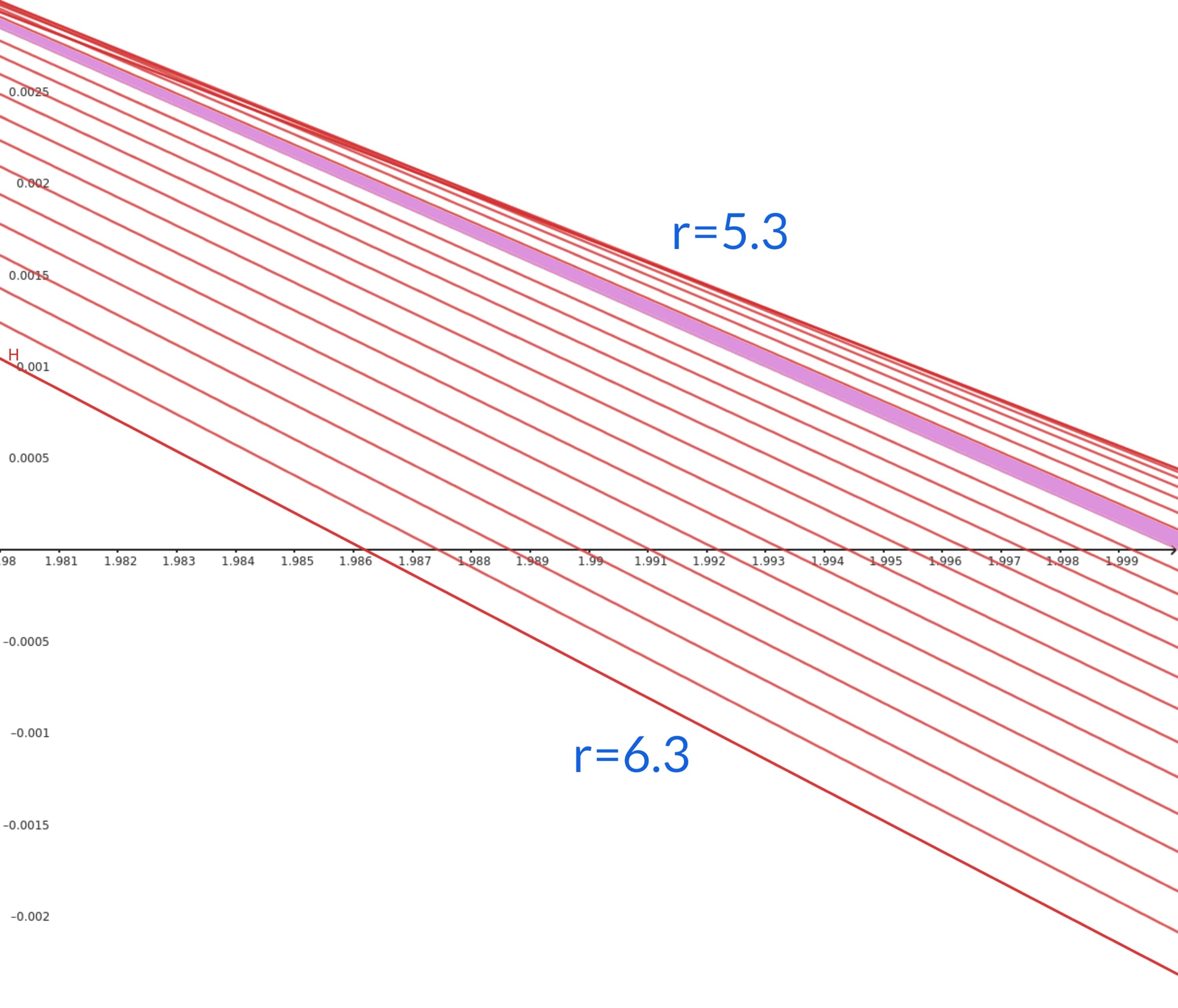}
\captionof{figure}{
\small{\bf Flux function 3c:}  Anomalous zone continued. Zoom of $G''(\rho)$ 
around its second inflection point  for $r=5.3$ (top curve) to $r=6.3$ (bottom 
curve) spaced by $\Delta r=0.05$. The transition at $r_4(d)\simeq 5.65$ ($1$ to 
$2$ inflection points) occurs on at the bottom of the pink layer.
}
\label{phase-anomalous2}
\end{minipage}
\end{center}
%
\begin{center}
\begin{minipage}{406pt}
\includegraphics[width=406pt , height=8cm]
{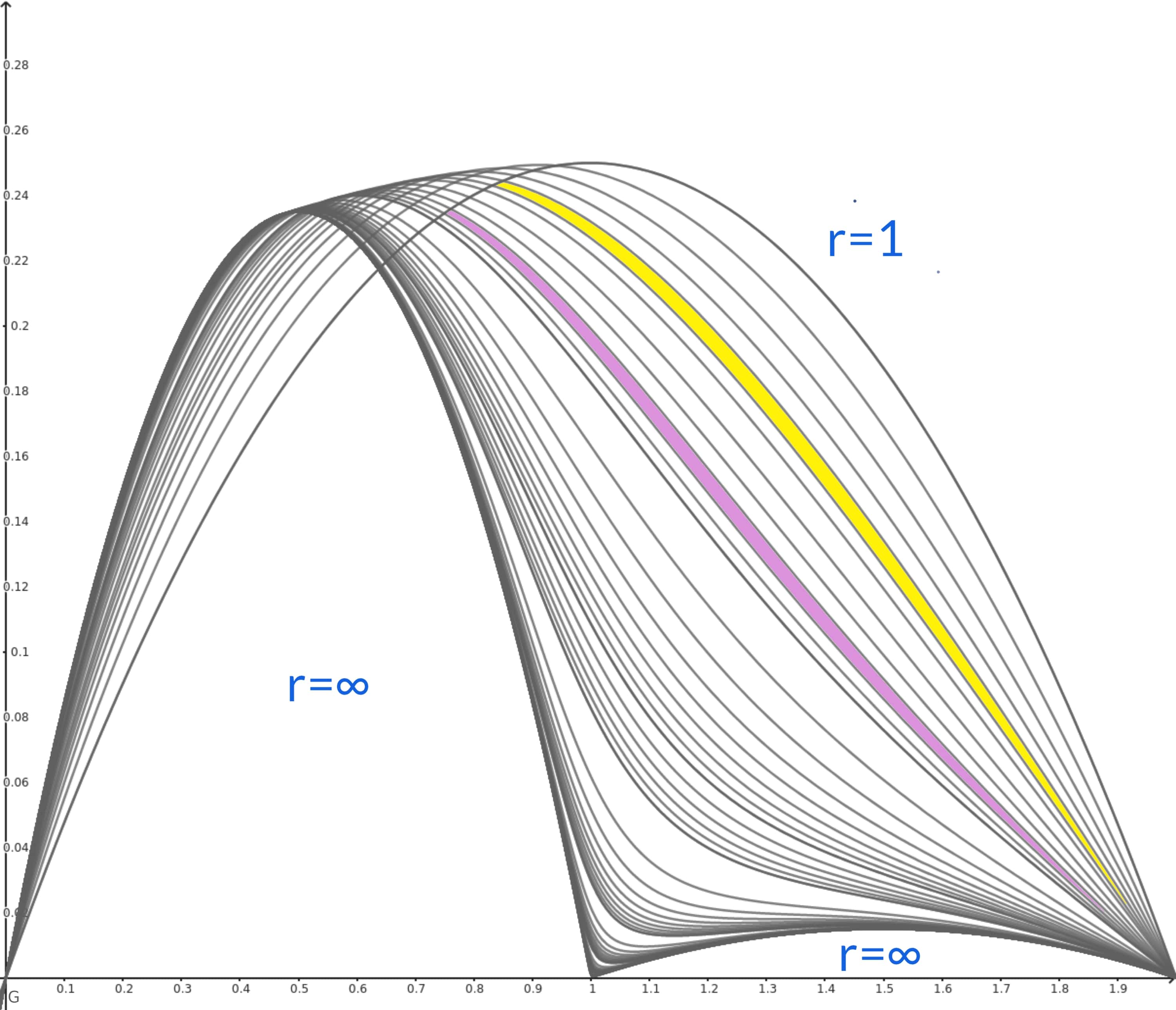}
\captionof{figure}{
\small{\bf Flux function 4:}  $\tilde{d}_0<d=0.94=\gamma_0$, $\gamma_1=0.06$. 
Top concave curve: $r=1$; bottom double-bump curve $r=+\infty$. The transitions 
occur at  $r_3(d)\simeq 3.3$ ($0$ to $1$ inflection point, yellow) 
and $r_4(d)\simeq 7.6$ ($1$ to $2$ inflection points, pink).
}
\label{phase4}
\end{minipage}
\end{center}
%
\begin{center}
\begin{minipage}{406pt}
\includegraphics[width=406pt , height=8cm]
{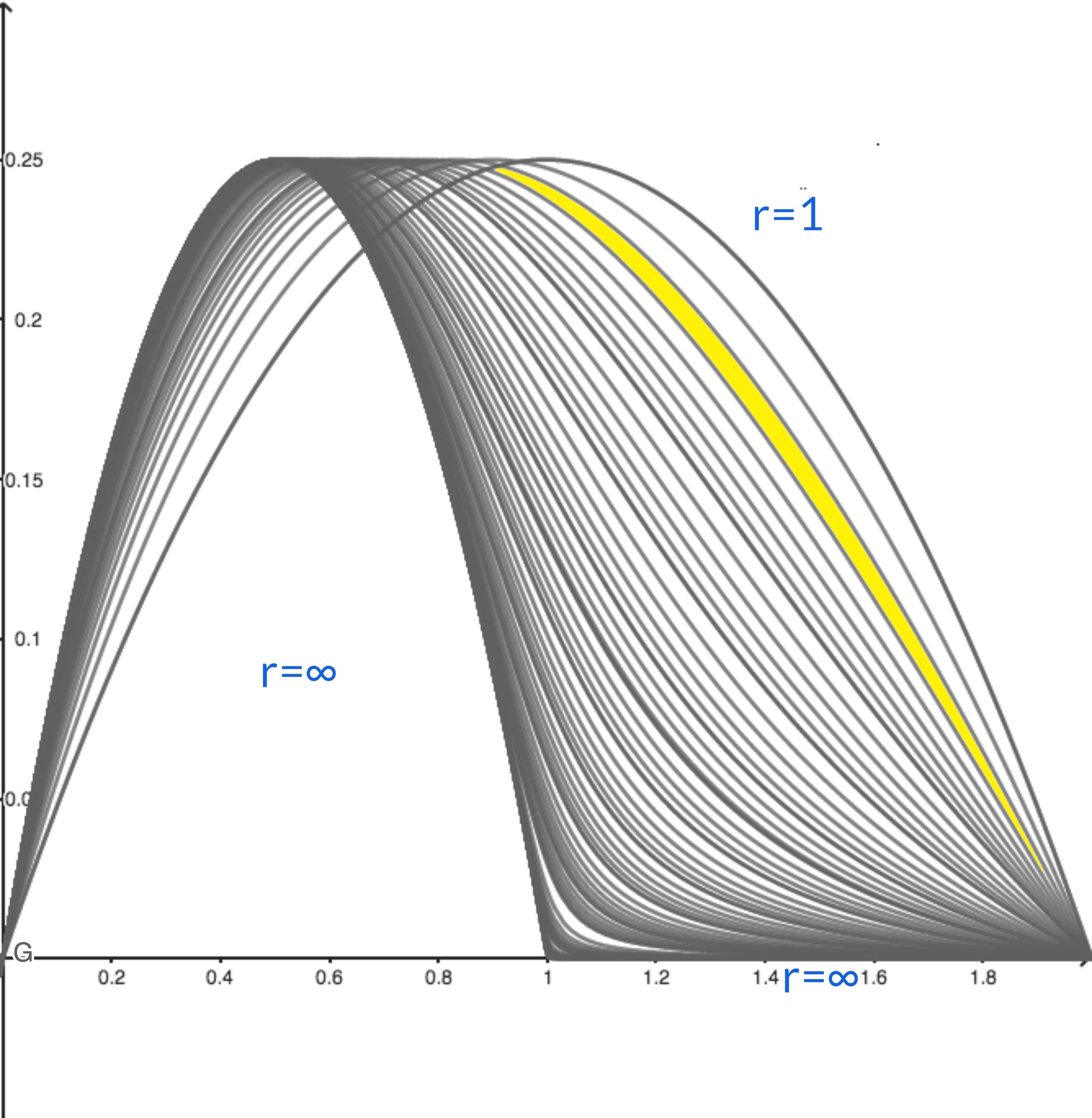}
\captionof{figure}{
\small{\bf Flux function 5:}  
$d=1=\gamma_0$, $\gamma_1=1-d=0$. For $r=1$: concave symmetric curve; 
$r=+\infty$: bump-flat curve.  The transition ($0$ to $1$ inflection point) occurs 
at $r_3(d)\simeq 2.4$ (yellow).
}
\label{phase5}
\end{minipage}
\end{center}
%
\begin{center}
\begin{minipage}{406pt}
\includegraphics[width=406pt , height=7.5cm]
{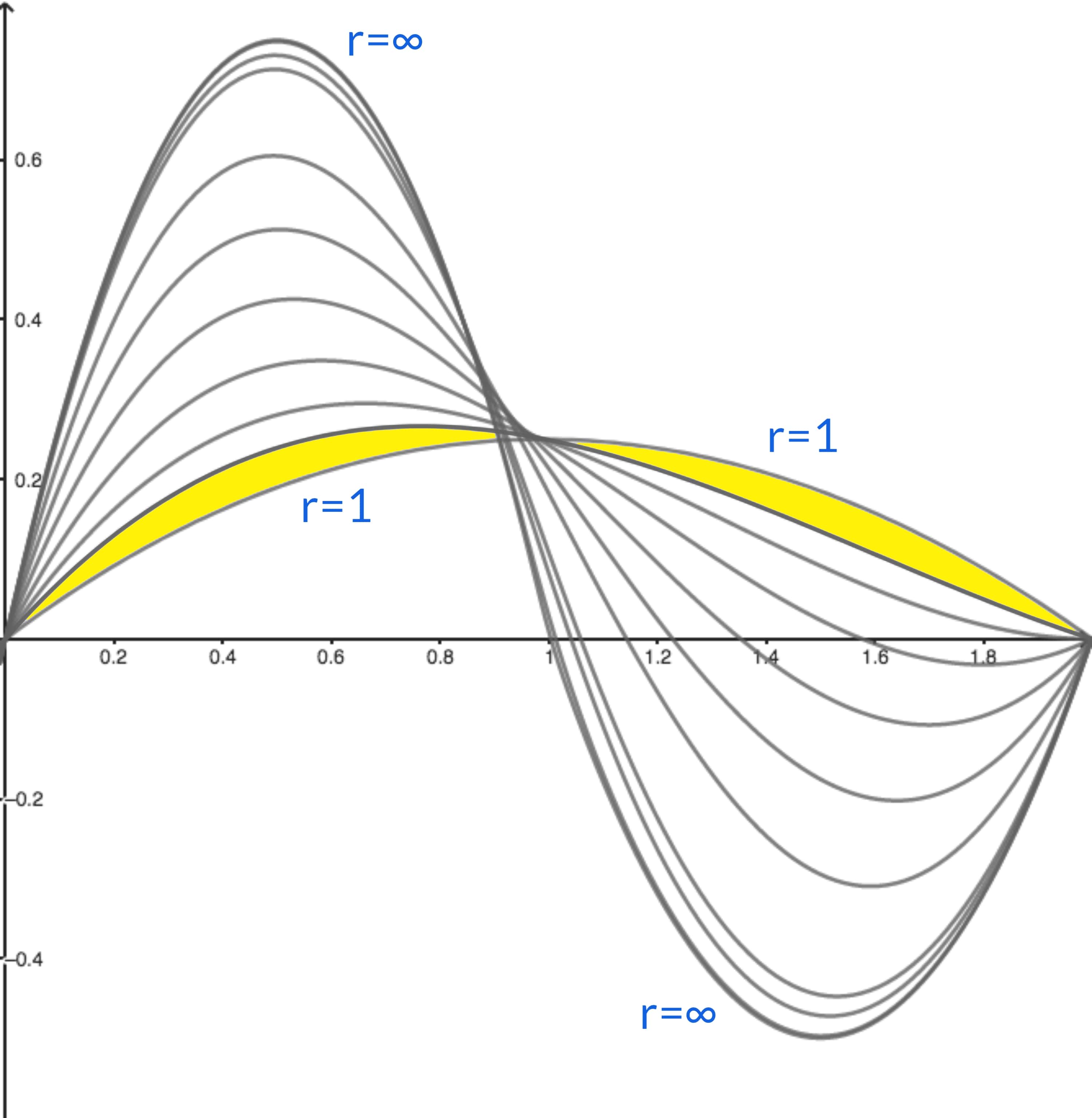}
\captionof{figure}{
\small{\bf Flux function 6:}  $d=3=\gamma_0$, $\gamma_1=1-d=-2$. Outer curves 
are  $r=1$ (concave curve) and
$r=+\infty$ (bump-well curve). The transition between $0$ and $1$ inflection 
point occurs at $r_3(d) \simeq 1.14$, inside the yellow layer between curves 
$r=1$ (concave curve) and $r=1.25$ (concave-convex curve). 
}
\label{phase6}
\end{minipage}
\end{center}
%
\begin{center}
\begin{minipage}{406pt}
\includegraphics[width=406pt , height=7.5cm]
{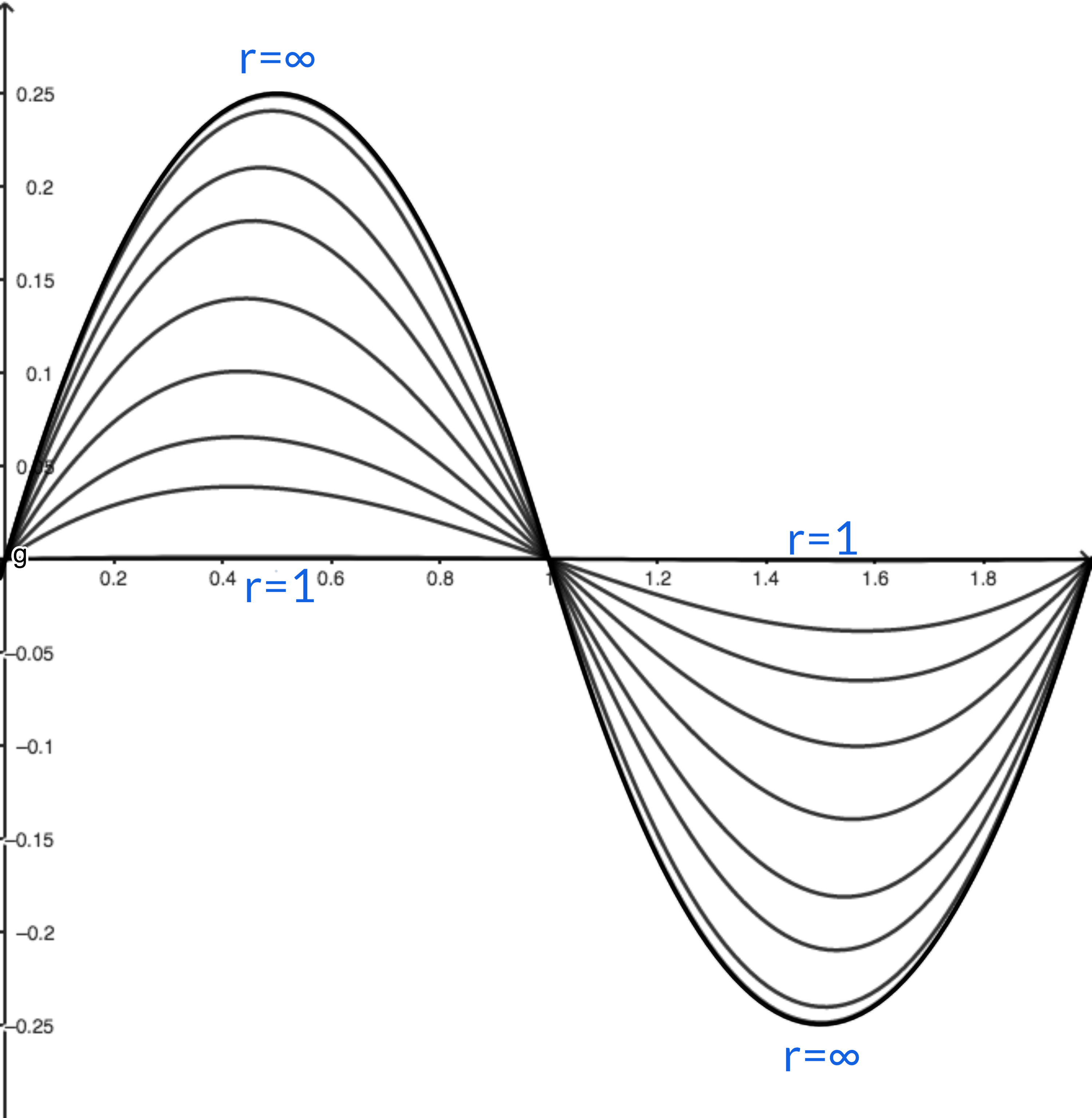}
\captionof{figure}{
\small{\bf Flux function 7:}  
$d_0=\infty$, $\gamma_1=1-d=-\infty$, plot of $d^{-1}G(\rho)=G_{1,-1}(\rho)$ 
(cf. \eqref{lim_G}); $r=1$  (identically $0$ flux) to $r=+\infty$ (in bold). 
Here $r_3(\infty)=1$, there is an instant transition from $0$ curve to $1$ 
inflection point. All curves are symmetric.
}
\label{phase7}
\end{minipage}
\end{center}
\section{Existence and uniqueness for \protect(\ref{general_system}) 
with $[0;1]^n$-valued solutions}\label{app:exist}
We prove the following:
\begin{theorem}\label{th:exist}
Under assumption
\eqref{assumption_system_exclusion_2}, there exists a unique 
(global in time) entropy solution to \eqref{general_system}  
in the sense of Definition \ref{def:relax_entropy}
\end{theorem}
\begin{proof}[Proof of Theorem \ref{th:exist}]
In \cite{hanat}, uniqueness (Theorem 2.1.) and local in time existence 
(Theorem 3.1.) of an entropy solution are proved when the solution is not
 restricted to take values in a bounded subset of $\R^n$, and functions 
 $f_i$ and $c_i$ are defined on the whole of $\R^n$. To deduce similar 
 results in our setting, we proceed as follows. \\ \\
{\em Step one.}  For $i=0,\ldots,n-1$, extend functions $f_i$ and $c_i$ 
(defined on $[0;1]^n$) into functions of class $C^1$, respectively on $\R$ 
and $\R^n$,  denoted $\widetilde{f}_i$ and $\widetilde{c}_i$. \\ \\
{\em Step two (uniqueness).}  Observe that   a $[0;1]^n$-valued 
entropy solution to \eqref{general_system} in the sense of 
Definition \ref{def:relax_entropy} is an $\R^n$-valued entropy solution
of the extended system
\be\label{extended_system}
{\partial_t}\rho_i(t,x)+{\partial_x}\widetilde{f}_i(\rho_i(t,x))
=\widetilde{c}_i(\rho_0(t,x),\ldots,\rho_{n-1}(t,x)),\quad i=0,\ldots,n-1
\ee
for the same initial data. It is thus unique by \cite[Theorem 2.1]{hanat}.\\ \\
{\em Step three (global in time existence).}
Here we shall use condition 
\eqref{assumption_system_exclusion_2}. Let 
$\vec{\rho}=(\rho_0,\ldots,\rho_{n-1})$ denote the entropy 
solution to \eqref{extended_system} with initial datum 
$\overrightarrow{\rho}^0=(\rho_0^0,\ldots,\rho_{n-1}^ 0)$. 
Our goal is to show that 
\be\label{our_goal}
\vec{\rho}(t,x)\in[0;1]^n,\quad\forall (t,x)\in[0;+\infty)\times\R
\ee
It will follow that $\vec{\rho}$ is the unique entropy solution 
to \eqref{general_system} with initial datum  $\vec{\rho}^0$, 
thus that the latter takes values in $[0;1]^n$.\\ \\
To this end, for $i=0,\ldots,n-1$, define functions 
${c}^*_i$ on $(0;+\infty)\times \R\times\R$ by
\be\label{def_cstar}
{c}^*_i(t,x,\rho)
:=\widetilde{c}_i(\rho_0(t,x),\ldots,\rho_{i-1}(t,x),\rho,\rho_{i+1}(t,x),
\ldots,\rho_{n-1}(t,x))
\ee
By Definition \ref{def:relax_entropy}, since $\vec{\rho}$ is an 
entropy solution to \eqref{extended_system}, each component $\rho_i$ 
is an entropy solution of the autonomous scalar conservation law

\be\label{decoupled}
{\partial_t}\rho_i(t,x)+ 
{\partial_ x}\widetilde{f}_i(\rho_i(t,x))={c}^*_i(t,x,\rho_i(t,x)),
\quad i=0,\ldots,n-1
\ee
Note that in the balance terms of \eqref{decoupled}, as opposed to 
those of $\widetilde{c}_i$, all variables other than $\rho_i$ are 
treated as {\em inputs} determined by the solution itself (viewed 
as an exeogeneous function). This is why these equations are decoupled.\\ \\
Let us now consider the constant and uniform functions $\vec{u}$ and 
$\vec{v}$ defined by
\be\label{def_supersol}
u_i(t,x):=0,\,
v_i(t,x):=1
\ee
for all $(t,x)\in[0;+\infty)\times\R$ and $i\in\{0,\ldots,n-1\}$.
By condition 
\eqref {assumption_system_exclusion_2}, 
\begin{eqnarray}
\label{subsol}
0=\partial_t u_i(t,x) & +& \partial_x\widetilde{f}_i(u_i(t,x))
\leq c_i^*(t,x,u_i(t,x))\\
\label{supersol}
0=\partial_t v_i(t,x) & + &\partial_x\widetilde{f}_i(v_i(t,x))\geq c_i^*(t,x,v_i(t,x))
\end{eqnarray}
Equalities in \eqref{subsol}--\eqref{supersol} hold because $u_i$ 
and $v_i$ are constant and uniform. Inequalities hold because in 
$c_i^*$, we apply $\widetilde{c}_i$ to a vector whose $i$-th coordinate 
is either $u_i(t,x)=0$
or $v_i(t,x)=1$, and other coordinates  (determined by $\vec{\rho}(t,x)$) 
are irrelevant to \eqref {assumption_system_exclusion_2}.
Thus, $u_i$ and $v_i$ are respectively an entropy subsolution and 
supersolution of \eqref{decoupled}. It follows, see e.g. \cite{kru}, 
that $u_i\leq\rho_i\leq v_i$. This proves \eqref{our_goal}.
\end{proof}
\end{appendix}
%
\noindent{\bf Acknowledgments.}
 We thank the anonymous referees for careful reading and many useful 
 comments that helped us improve the manuscript. 
Part of this work was done during the stay of C.B, O.B. and E.S. at the
Institut Henri Poincar\'e (UMS 5208 CNRS-Sorbonne Universit\'e) -
Centre Emile Borel for
the trimester ``Stochastic Dynamics Out of Equilibrium''.
The authors thank these institutions for hospitality and support.
C.B., O.B. and E.S. thank
Universit\'{e} Paris Cit\'e for hospitality, as well as Villa Finaly
(where they attended the conference ``Equilibrium and 
Non-equilibrium Statistical Mechanics'').
This work  has been conducted within the FP2M federation 
(CNRS FR 2036) and was partially supported by laboratoire MAP5,
grants ANR-15-CE40-0020-02 and ANR-14-CE25-0011 (for C.B. and E.S.),
LabEx CARMIN (ANR-10-LABX-59-01).
G.A. was supported by the Israel Science Foundation grant \# 957/20.
O.B. was supported by EPSRC's EP/R021449/1 Standard Grant, and he was
partly funded by the Deutsche Forschungsgemeinschaft 
(DFG, German Research Foundation) under Germany's Excellence Strategy - GZ 2047/1,
projekt-id 390685813.
\bibliographystyle{plain}
\bibliography{MLT2}
\end{document}